\documentclass[letterpaper, oneside]{amsart}

\usepackage[english]{babel}
\usepackage[utf8]{inputenc}
\usepackage[T1]{fontenc}
\usepackage{fancyhdr}
\usepackage{csquotes}
\usepackage{cite}

\usepackage{amssymb}
\usepackage{amsmath}
\usepackage{nicefrac}
\usepackage{bbm}
\usepackage{amsfonts}
\usepackage{tensor}
\usepackage{amsthm}
\usepackage{mathrsfs} 
\usepackage{mathtools}
\usepackage{extpfeil}
\usepackage{tikz-cd}
\usepackage{bm}
\usepackage{bbold}
\usepackage[mathscr]{euscript}

\usepackage{color}
\usepackage{graphicx}
\usepackage{wrapfig}
\usepackage{subfigure}
\usepackage{placeins}

\usepackage{multicol}
\usepackage{ltxtable}
\usepackage{url}
\usepackage{geometry}
\usepackage{paralist}
\usepackage{verbatim}
\usepackage{hyperref}
\usepackage{glossaries}
\usepackage{float}
\usepackage{nameref}
\usepackage{array}
\usepackage{marginnote}
\usepackage{adjustbox}
\usepackage{blindtext}
\usepackage{todonotes}
\usetikzlibrary{decorations.pathmorphing} 
\usepackage{quiver}

\newcommand{\field}[1]{\mathbb{#1}}
\DeclareMathOperator{\colim}{colim }

\DeclareMathOperator{\map}{Map}
\DeclareMathOperator{\mor}{Mor}
\DeclareMathOperator{\alg}{Alg}
\DeclareMathOperator{\e}{End}
\DeclareMathOperator{\lmod}{LMod}

\DeclareMathOperator{\ch}{Ch}
\DeclareMathOperator{\an}{\mathscr{S}}

\DeclareMathOperator{\fin}{\mathscr{F}in_{\ast}}
\DeclareMathOperator{\spec}{\textnormal{Spec}}

\newcommand{\lm}{\mathscr{LM}^{\otimes}}
\newcommand{\ass}{{\mathscr{A}ssoc}^{\otimes}}
\newcommand{\braces}{\mathbf{Braces}}
\newcommand{\ger}{\mathbf{Ger}}
\renewcommand{\hom}{\textnormal{Hom}}

\newcommand{\eone}{C_{\ast}(\mathbb{E}^T_1)}

\newcommand{\dgpsh}[1]{\textnormal{dgPSh}(#1)}
\newcommand{\dgsh}[1]{\textnormal{dgSh}(#1)}
\newcommand{\affpsh}[1]{\textnormal{dgPSh}^{\textnormal{aff}}(#1)}
\newcommand{\affsh}[1]{\textnormal{dgSh}^{\textnormal{aff}}(#1)}
\newcommand{\diag}{\textnormal{Diag}}

\newcounter{intro}

\newtheorem{theorem}{Theorem}
\newtheorem{Itheorem}[intro]{Theorem}
\newtheorem{lemma}[theorem]{Lemma}
\newtheorem{proposition}[theorem]{Proposition}
\newtheorem{corollary}[theorem]{Corollary}

\newtheorem{conjecture}[theorem]{Conjecture}

\theoremstyle{definition}
\newtheorem{definition}[theorem]{Definition}
\newtheorem{notation}[theorem]{Notation}

\newtheorem{example}[theorem]{Example}

\theoremstyle{remark}
\newtheorem{remark}[theorem]{Remark}
\newtheorem*{remark*}{Remark}

\title{$\mathbb{E}_2$-algebra structures on the derived center of an algebraic scheme}
\author[S. Farr]{Sonja M. Farr}
\address{Department of Mathematics \& Statistics, University of Nevada,  Reno.}
\email{sonjaf@unr.edu}

\begin{document}
\begin{abstract}
This paper provides an explicit interface between J. Lurie's work on higher centers, and the Hochschild cohomology of an algebraic $\field{k}$-scheme within the framework of deformation quantization. We first recover a canonical solution to Deligne's conjecture on Hochschild cochains in the affine and global cases, even for singular schemes, by exhibiting the Hochschild complex as an $\infty$-operadic center. We then prove that this universal $\mathbb{E}_2$-algebra structure precisely agrees with the classical Gerstenhaber bracket and cup product on cohomology in the affine and smooth cases. This last statement follows from our main technical result which allows us to extract the Gerstenhaber bracket of any $\mathbb{E}_2$-algebra obtained from a 2-algebra via Lurie's Dunn Additivity Theorem.
\end{abstract} 
\maketitle
\tableofcontents
\section{Introduction}
M. Kontsevich's Formality Theorem in deformation quantization states that the Hochschild-Kostant-Rosenberg map lifts to a quasi-isomorphism of homotopy Lie algebras between polyvector fields and polydifferential operators of a smooth manifold. In \cite{Tam2}, D. Tamarkin extended this to a quasi-isomorphism of homotopy Gerstenhaber algebras for affine spaces $\mathbb{A}^n_{\field{k}}$, where $\field{k}$ is any field of characteristic zero by showing that the operad of little 2-disks is formal and using Deligne's conjecture on Hochschild cochains. This proof also highlights that these formality morphisms for the smooth Hochschild cochains are non-canonical, since Tamarkin's little disks formality depends on the choice of a Drinfeld Associator.\par
In fact, using results from D. Bar-Natan \cite{BN}, Tamarkin later showed in \cite{Tam3} that his proof essentially identifies Drinfeld Associators with operad isomorphisms between the operad of parenthesized braids and the operad of parenthesized chord diagrams which are the identity on objects.\par
Based on Tamarkin's results, Kontsevich conjectured in \cite{K} that the Grothendieck-Teichmüller group (GT) should act on formality isomorphisms between $T_{\text{poly}}(X)$ and $D_{\text{poly}}(X)$ of a smooth complex variety, and that this action should be of motivic nature, arising as a consequence of the fact that the equations in the Knizhnik-Zamoldchikov Associator are periods. This conjecture was proved by V. Dolgushev, C. Rogers and T. Willwacher in \cite{DRW}. They were able to show that the Deligne-Drinfeld elements of the Grothendieck-Teichmüller group act by contraction with the odd components of the Chern character of the variety on the cohomology of the sheaf of polyvector fields. In particular, they were able to give examples for which this action is non-trivial.\par
Dolgushev-Rogers-Willwacher use a result by D. Calaque and M. Van den Bergh in \cite{CVdB} which shows that the Kontsevich Formality Theorem can be extended to non-affine cases by adding a correction term to the HKR map which depends on the Atiyah class of the variety. Astonishingly, this correction term has the form $J^{1/2}$ with $J = \det\left(q(\text{At}(X))\right)$ the Todd class of the variety and 
\begin{align*}
    q(x)= \frac{x}{1-e^{-x}}.
\end{align*}
This is clearly reminiscent of (Kontsevich's generalization of) the classical Duflo Isomorphism Theorem, which states that for any finite dimensional Lie algebra $\mathfrak{g}$, we get an algebra isomorphism
\begin{align*}
    \text{PBW}\circ \det(q^{-1}(\text{ad}))^{1/2}: S(\mathfrak{g})^{\mathfrak{g}} \xrightarrow{\cong} U(\mathfrak{g})^{\mathfrak{g}}.
\end{align*}
It is a result by A. Alekseev and C. Torossian \cite{AT} that the Grothendieck-Teichmüller group also acts on "classical" Duflo isomorphisms as above by changing the Dulfo element. Unfortunately, this action was shown in \cite{AA} to be trivial. \\\par
Note that the codomain $U(\mathfrak{g})^{\mathfrak{g}} = H^{0}(\mathfrak{g},U(\mathfrak{g}))$ is the center $Z(\mathfrak{g})$ of the Lie algebra, and can be computed as the 0th cohomology with coefficients in the universal enveloping algebra. A. Ramadoss \cite{R}, J. Roberts and S. Willerton \cite{RW} and N. Markarian \cite{M} showed that various versions of the Hochschild cochain complex of a smooth scheme $X$ satisfy the universal property of a universal enveloping algebra of the tangent Lie algebra $\mathcal{T}_X[-1]$ in the derived 1-category of $\mathcal{O}_X$-modules. Therefore, taking hypercohomology, the Hochschild cohomology of $X$ computes a "derived center" of $\mathcal{O}_X$. The definition of the universal enveloping algebra object in the derived category comes from \cite{HV} by V. Hinich and A. Vaintrob, where they also defined the center of an associative algebra object. However, this center does not satisfy any universal property, and it also is "external" in the sense that it is just a commutative algebra in the category of sets.\par 
In contrast to this, J. Lurie in \cite[Chapter 5]{HA} defines a higher center of an algebra $A$ over an $\infty$-operad $\mathcal{O}$ in some symmetric monoidal $\infty$-category as a \textit{universal object} in a suitable $\infty$-category of algebra actions on $A$. In particular, the center of $A$ is an associative algebra object in the category of $\mathcal{O}$-algebras, and in case of the little $k$-disks $\infty$-operads $\mathbb{E}_k^{\otimes}$, Lurie uses an $\infty$-categorical version of the Dunn Additivity Theorem to show that the higher center of an $\mathbb{E}_k$-algebra is indeed an $\mathbb{E}_{k+1}$-algebra. This may be viewed as a generalized Deligne Conjecture.\\\par

In this paper we will lay the groundwork for explaining the $\text{GT}$ action on Duflo isomorphisms within the theory of centers by connecting Lurie's work on higher centers with the classical results on the Hochschild cohomology of schemes. In particular, we will argue that the Hochschild complex of any quasi-compact separated scheme should be defined as the $\mathbb{E}_1$-operadic center of its structure sheaf, thereby equipping it with a canonical $\mathbb{E}_2$-algebra structure. \\\par
Given a symmetric monoidal $\infty$-category $\mathcal{C}$ and an $\infty$-operad $\mathcal{O}$, we let $\alg_{\mathcal{O}}(\mathcal{C})$ denote the $\infty$-category of $\mathcal{O}$-algebras in $\mathcal{C}$. We let $\mathcal{D}_{\infty}(\field{k})$ be the derived $\infty$-category of $\field{k}$-vector spaces. We denote the higher center of an $\mathcal{O}$-algebra object $A$ by\footnote{We will later drop the subscript "$\mathcal{O}$".} $\mathfrak{Z}_{\mathcal{O}}(A)$.\\\par
We will first consider the affine case, and we will prove that we recover a solution to the classical Deligne Conjecture. This is done in Section \ref{Recovering_the_Hochschild_complex_as_e1_center}. The main result is the following
\begin{Itheorem}[Theorem \ref{thm1}, Corollary \ref{cor9}]
    Let $A$ be an associative $\field{k}$-algebra. The Hochschild complex
    \begin{align*}
        C^{\ast}(A,A) = \hom_{\field{k}}(A^{\otimes \ast},A)
    \end{align*}
    together with the evaluation map is a center for $A\in \alg_{\mathbb{E}_1}(\mathcal{D}_{\infty}(\field{k}))$. In particular, it is an object of $\alg_{\mathbb{E}_1}(\alg_{\mathbb{E}_1}(\mathcal{D}_{\infty}(\field{k})))\simeq \alg_{\mathbb{E}_2}(\mathcal{D}_{\infty}(\field{k}))$. Its underlying Gerstenhaber bracket and cup product in cohomology agree with the classical Gerstenhaber algebra structure obtained from the $\braces$-algebra structure.
\end{Itheorem}
In Section \ref{The_Hochschild_complex_of_a_scheme}, we recall a construction of the $\infty$-category $\text{Sh}_{\infty}(X)$ of dg sheaves over a  quasi-compact separated scheme $X$, and examine some basic properties of the center of the structure sheaf as an $\mathbb{E}_1$-algebra in this $\infty$-category. We define the  Hochschild cochain complex of $X$ as the center of $\mathcal{O}_X$ in the $\infty$-category of dg sheaves on $X$. We show that it has the desired local properties, even in the singular case. In particular, we show 
\begin{Itheorem}[Theorem \ref{thm5}]
Let $U = \textnormal{Spec}(A)\subseteq X$ be an affine open. The functor $\mathbb{R}\Gamma_U: \textnormal{Sh}_{\infty}(X) \rightarrow \mathcal{D}_{\infty}(\field{k})$ is lax symmetric monoidal and hence induces a functor $\mathbb{R}\Gamma_U: \alg_{\mathbb{E}_2}(\textnormal{Sh}_{\infty}(X)) \rightarrow \alg_{\mathbb{E}_2}(\mathcal{D}_{\infty}(\field{k}))$. We have
\begin{align*}
    \mathbb{R}\Gamma_U(\mathfrak{Z}_{\mathbb{E}_1}(\mathcal{O}_X)) \simeq \mathfrak{Z}_{\mathbb{E}_1}(A).
\end{align*}
\end{Itheorem}
In the smooth case, Kontsevich defined the Hochschild cochains to be the sheaf $\mathcal{D}_{\text{poly}}(X)$ of polydifferential operators. This sheaf comes with the structure of a $\braces$-algebra inherited from the local Hochschild complexes, and therefore has the structure of a homotopy Gerstenhaber algebra by Tamarkin's results. In order to compare to the existing GT actions, we need to compare our newly obtained $\mathbb{E}_2$-algebra structure to this homotopy Gerstenhaber algebra structure. 
\begin{Itheorem}[Theorem \ref{thm6}, Corollary \ref{thm9}]
For a smooth quasi-compact separated scheme $X$ of finite type over $\field{k}$, the sheaf of polydifferential operators $\mathcal{D}_{\textnormal{poly}}(X)$ is a center of $\mathcal{O}_X$:
\begin{align*}
    \mathcal{D}_{\textnormal{poly}}(X) \simeq \mathfrak{Z}_{\mathbb{E}_1}(\mathcal{O}_X).
\end{align*}
This equips $\mathcal{D}_{\textnormal{poly}}(X)$ with the structure of an $\mathbb{E}_2$-algebra. The corresponding Gerstenhaber algebra in the $\field{k}$-linear derived 1-category agrees with the classical one coming from the $\braces$-algebra structure.
\end{Itheorem}
In the course of proving these results, we also obtain a couple of technical results about higher centers. In particular, we show how to explicitly compute the Gerstenhaber bracket of a 2-algebra up to homotopy using Lurie's version of the Dunn Additivity Theorem (see Corollary \ref{cor1}). We also show that the 2-algebra structure on a center which is obtained as an endomorphism object indeed corresponds to the composition product (Yoneda product) and the convolution product (see Corollary \ref{cor7}). 
\\ \paragraph*{\bf Related work.} The connection between the Hochschild cochain complex of an associative algebra and its higher center in the form of the universal 2-algebra acting on it was already described by Kontsevich in \cite{K}. In fact, in Deligne's original letter outlining the conjecture, he already stated that he expected the little 2-disk algebra structure to come from the composition and convolution product via some type of Eckmann-Hilton argument. P. Hu, I. Kriz and A. Voronov in \cite{HKV} proved a simplicial version of Deligne's Conjecture using this idea and explicit models of the little disks operads. In particular, for their result they proved a version of the Dunn Additivity Theorem for these specific models of the little $k$-disks operads. \par
In 2010, D. Ben-Zvi, J. Francis and D. Nadler \cite{BZFN} defined the Hochschild cohomology of an associative algebra object $A$ in a presentable closed symmetric monoidal $\infty$-category as its derived center, which they set to be the internal endomorphism object of $A$ as module over $A\otimes A^{\text{op}}$. In particular, this yields an "internal Hochschild cohomology". They also mention that they obtain an $\mathbb{E}_2$-algebra structure on the Drinfeld center of certain categories by considering the connection of the derived center of topological field theories.   \par
In 2013, Francis \cite{FR} used Lurie's Dunn Additivity Theorem to examine centers of stable $\infty$-categories and to relate the center to the module of derivations. He also defines the Hochschild cohomology of an algebra over an $\infty$-operad $\mathcal{O}$ as the hom-set of $\mathcal{O}$-module maps over $A$ from $A$ to itself, which is closely related to Lurie's definition of the center. \par
In 2020, I. Iwanari \cite{I} used this definition of the Hochschild cochain complex to show that pair of Hochschild cohomology and homology of a linear category over some commutative ring spectrum yields an algebra over the KS operad, which is a generalization of the $\mathbb{E}_2$-operad.\par
In 2023, C. Brav and N. Rozenblyum \cite{BR} proved a cyclic version of Deligne's Conjecture (i.e. replacing the little 2-disk operad by the framed version) also using the above described techniques. In particular, their method also relies on Lurie's version of Dunn Additivity. However, none of the above results provide a direct comparison to the classical solutions of Deligne's Conjecture, or make any claim about the underlying Gerstenhaber algebra structure. Similarly, the author is not aware of any comparison of the center of a scheme to the classical sheaf of polydifferential operators. 
\\ \paragraph*{\bf Conventions.}
Throughout this paper, $\field{k}$ is a field of characteristic zero. Complexes are generally chain graded unless states otherwise, and we view non-negatively graded cochain complexes as non-positive chain complexes. The term "operad" is reserved for non-reduced unital symmetric operads. We denote $\infty$-operads by calligraphic letters (e.g. $\ass$) and dg operads by bold letters (e.g. $\textbf{Assoc}$). If $A$ is an associative algebra object in a monoidal 1-category $C$, we denote by $\lmod_A(C)$ the category of left modules over $A$. We try to use the term "$\infty$-category" for $\infty$-categories, but if nothing else is stated, "category" refers to $\infty$-category and "1-category" refers to ordinary categories. The $\infty$-category of spaces is denoted by $\an$.  If $C$ is a model category, we denote by $C^c$ the full subcategory of cofibrant objects, and by $C^{\circ}$ the full subcategory of fibrant-cofibrant objects. We denote by $R$ and $Q$ the fibrant and cofibrant replacement functor respectively. We denote the presheaf tensor product simply by "$\otimes$", and we decorate symbols with "$(-)^a$" to indicate sheafification.
\\ \paragraph*{\bf Acknowledgments.} I thank my thesis advisor, Chris Rogers, for proposing the problem and providing essential guidance. I would also like to thank Nick Rozenblyum for invaluable comments on J. Lurie's proof of the Dunn Additivity Theorem and Dennis Borisov for helpful discussions on dg sheaves. This work was supported by NSF Grant DMS-2305407.
\color{black}
\section{Preliminaries}
We follow Lurie's formalism for $\infty$-operads as developed in \cite{HA}. In particular, an $\infty$-operad is a morphism $p: \mathcal{O}^{\otimes} \rightarrow \fin$ of $\infty$-categories satisfying a list of conditions making $\mathcal{O}^{\otimes}$ into an "$\infty$-category of operators". We will freely use notation from \cite{HA} regarding $\infty$-operads, algebras and modules over these.
\subsection{Morphism objects and operadic centers}
We recall the definition of endomorphism objects, centers and centralizers, and we show how these notions are related. The definitions we use can be found in \cite[Section 4.2, 4.7 and 5.3]{HA}.\\ \par
Let $\lm$ be the 2-colored $\infty$-operad whose algebras are given by pairs of homotopy associative algebras and modules over them as defined in \cite[Definition 4.2.1.7]{HA}. Denote by $\mathfrak{a}$ and $\mathfrak{m}$ the "algebra" and "module" color respectively.
\begin{definition}
Let $q: \mathcal{C}^{\otimes} \rightarrow \lm$ be a coCartesian fibration of $\infty$-operads. We then say that $q$ exhibits the $\infty$-category $\mathcal{C}_{\mathfrak{m}}$ as \textbf{left tensored} over the monoidal $\infty$-category $\mathcal{C}_{\mathfrak{a}}^{\otimes} := \mathcal{C}^{\otimes} \times_{\lm} \ass$. 
\end{definition}
In particular, if $q$ exhibits $\mathcal{C}_{\mathfrak{m}}$ as left tensored over $\mathcal{C}^{\otimes}_{\mathfrak{a}}$, then it determines a tensoring
\begin{align*}
    \otimes: \mathcal{C}_{\mathfrak{a}} \times \mathcal{C}_{\mathfrak{m}} \rightarrow \mathcal{C}_{\mathfrak{m}},
\end{align*}
well-defined up to homotopy, that is compatible with the monoidal structure on $\mathcal{C}_{\mathfrak{a}}$ up to homotopy. This is obtained as a coCartesian lift of the map $\{\mathfrak{a},\mathfrak{m}\} \rightarrow \mathfrak{m}$ over the active map $\langle 2\rangle \rightarrow \langle 1 \rangle$ in $\lm$. Note that any monoidal $\infty$-category $\mathcal{C}^{\otimes}\rightarrow \ass$ is left tensored over itself by considering the fiber product $\mathcal{C}^{\otimes} \times_{\ass} \lm \rightarrow \lm$.
\begin{notation}
Let $q: \mathcal{C}^{\otimes} \rightarrow \lm$ exhibit $\mathcal{C}_{\mathfrak{m}}$ as left tensored over $\mathcal{C}_{\mathfrak{a}}^{\otimes}$. Then we denote by 
\begin{align*}
    \lmod(\mathcal{C}) := \alg_{/\lm}(\mathcal{C})
\end{align*}
the $\infty$-category of pairs of associative algebras in $\mathcal{C}_{\mathfrak{a}}$ and left modules over them.
\end{notation}
Recall that for ordinary categories, internal homs and more generally enrichments are right adjoint to a tensoring. Similarly, one makes the following definition.
\begin{definition}\label{def3}
    Let $\mathcal{C}^{\otimes}\rightarrow \lm$ be a coCartesian fibration of $\infty$-operads exhibiting $\mathcal{C}_{\mathfrak{m}}$ as left tensored over $\mathcal{C}^{\otimes}_{\mathfrak{a}}$. If $M,N\in \mathcal{C}_{\mathfrak{m}}$, a \textbf{morphism object} for $M$ and $N$ is an object $\mor_{\mathcal{C}_{\mathfrak{m}}}(M,N)\in \mathcal{C}_{\mathfrak{a}}$ together with a map $\alpha\in \map_{\mathcal{C}_{\mathfrak{m}}}(\mor_{\mathcal{C}_{\mathfrak{m}}}(M,N)\otimes M,N)$ such that for each $C\in \mathcal{C}_{\mathfrak{a}}$, post-composition with $\alpha$ induces a homotopy equivalence
    \begin{align}
        \map_{\mathcal{C}_{\mathfrak{a}}}(C, \mor_{\mathcal{C}_{\mathfrak{m}}}(M,N))\xrightarrow{\simeq} \map_{\mathcal{C}_{\mathfrak{m}}}(C\otimes M,N).
    \end{align}
    We call $\mathcal{C}_{\mathfrak{m}}$ \textbf{enriched} over $\mathcal{C}^{\otimes}_{\mathfrak{a}}$ if all the morphisms objects exist.
\end{definition}
The following result shows that we can think of a morphism object as the classifying object of maps $A\otimes M \rightarrow N$ in $\mathcal{C}_{\mathfrak{m}}$ with $A\in \mathcal{C}_{\mathfrak{a}}$. Denote by $\mathcal{C}_{\mathfrak{a}} \times_{\mathcal{C}_{\mathfrak{m}}} {\mathcal{C}_{\mathfrak{m}}}_{/N}$ the pullback of the forgetful functor ${\mathcal{C}_{\mathfrak{m}}}_{/N} \rightarrow \mathcal{C}_{\mathfrak{m}}$ and the map
\begin{align}\label{eq2}
    -\otimes M: \mathcal{C}_{\mathfrak{a}} \rightarrow \mathcal{C}_{\mathfrak{m}}.
\end{align}
\begin{proposition}\label{prop6}
    Let $q: \mathcal{C}^{\otimes} \rightarrow \lm$ be a coCartesian fibration of $\infty$-operads. Let $M,N\in\mathcal{C}_{\mathfrak{m}}$. Then an object $\mor_{\mathcal{C}_{\mathfrak{m}}}(M,N)\in \mathcal{C}_{\mathfrak{a}}$ together with a map $\alpha: \mor_{\mathcal{C}_{\mathfrak{m}}}(M,N)\otimes M \rightarrow N$ is a morphism object of $M$ and $N$ if and only if $(\mor_{\mathcal{C}_{\mathfrak{m}}}(M,N),\alpha)\in \mathcal{C}_{\mathfrak{a}}\times_{\mathcal{C}_{\mathfrak{m}}}{\mathcal{C}_{\mathfrak{m}}}_{/N}$ is final.
\end{proposition}
\begin{proof}
    Note that ${\mathcal{C}_{\mathfrak{m}}}_{/N}\rightarrow \mathcal{C}_{\mathfrak{m}}$ is a right fibration, and since these are stable under base change, so is $f:\mathcal{C}_{\mathfrak{a}}\times_{\mathcal{C}_{\mathfrak{m}}}{\mathcal{C}_{\mathfrak{m}}}_{/N}\rightarrow \mathcal{C}_{\mathfrak{a}}$. Consider the functor $F: \mathcal{C}_{\mathfrak{a}}^{\textnormal{op}}\rightarrow \an$ classifying $f$. Then by \cite[Lemma 2.2.2.4]{HTT}, its underlying functor $hF: h\mathcal{C}_{\mathfrak{a}}^{\textnormal{op}}\rightarrow h\an$ can be recovered as follows. On objects, $X\in \mathcal{C}_{\mathfrak{a}}$ is sent to its fiber
    \begin{align*}
(\mathcal{C}_{\mathfrak{a}}\times_{\mathcal{C}_{\mathfrak{m}}}{\mathcal{C}_{\mathfrak{m}}}_{/N})\times_{\mathcal{C}_{\mathfrak{a}}}\{X\} \simeq \{X\otimes M\}\times_{\mathcal{C}_{\mathfrak{m}}}{\mathcal{C}_{\mathfrak{m}}}_{/N} \simeq \map_{\mathcal{C}_{\mathfrak{m}}}(X\otimes M,N).
    \end{align*}
    Given a morphism $e: Y\rightarrow X$ in $h\mathcal{C}_{\mathfrak{a}}$, the induced map between the fibers comes from solving the lifting problem
    \begin{center}
        \begin{tikzcd}
{\{1\}\times \text{Map}_{\mathcal{C}_{\mathfrak{m}}}(X\otimes M,N)} \arrow[d, hook] \arrow[rr, hook]  &                           & \mathcal{C}_{\mathfrak{a}}\times_{\mathcal{C}_{\mathfrak{m}}}{\mathcal{C}_{\mathfrak{m}}}_{/N} \arrow[d,"f"] \\
{\Delta^1\times \text{Map}_{\mathcal{C}_{\mathfrak{m}}}(X\otimes M,N)} \arrow[rd] \arrow[rru, dashed] &                           & \mathcal{C}_{\mathfrak{a}}                                               \\
                                                                                      & \Delta^1 \arrow[ru, "e"'] &                                                                         
\end{tikzcd},
    \end{center}
    and restricting the lift to $\{0\}\times \map_{\mathcal{C}_{\mathfrak{m}}}(X\otimes M,N)$. Since $f$ is a pullback of the right fibration ${\mathcal{C}_{\mathfrak{m}}}_{/N}\rightarrow \mathcal{C}_{\mathfrak{m}}$, the lift above is induced by the solution to 
    \begin{center}
        \begin{tikzcd}
{\{1\}\times \map_{\mathcal{C}_{\mathfrak{m}}}(X\otimes M,N)} \arrow[r] \arrow[d]                                      & {\mathcal{C}_{\mathfrak{m}}}_{/N} \arrow[d] \\
{\Delta^1\times \map_{\mathcal{C}_{\mathfrak{m}}}(X\otimes M,N)} \arrow[r, "e\otimes \text{id}_M"'] \arrow[ru, dashed] & \mathcal{C}_{\mathfrak{m}}               
\end{tikzcd}.
    \end{center}
    But the restriction of this lift to $\{0\}\times \map_{\mathcal{C}_{\mathfrak{m}}}(X\otimes M,N)$ is given by pre composition with $e\otimes \text{id}_M$. Therefore, we see that $hF$ is given by the composition of $-\otimes M$ and $\map_{\mathcal{C}_{\mathfrak{m}}}(-,N)$. Recall that by \cite[Proposition 4.4.4.5]{HTT}, an object $(X,X\otimes M\xrightarrow{\eta} N)$ is final in $\mathcal{C}_{\mathfrak{a}}\times_{\mathcal{C}_{\mathfrak{m}}}{\mathcal{C}_{\mathfrak{m}}}_{/N}$ if and only if the pair $(X,\eta\in hF(X))$ represents $hF$. Then we are done after noting that by definition, $(\mor_{\mathcal{C}_{\mathfrak{m}}}(M,N),\alpha)$ is a morphism object exactly if it represents the functor $X\mapsto \map_{\mathcal{C}_{\mathfrak{m}}}(X\otimes M,N)$.
\end{proof}
Now put $N = M$. Then by the above proposition, $\text{End}(M):= \mor_{\mathcal{C}_{\mathfrak{m}}}(M,M)\in \mathcal{C}_{\mathfrak{a}}$ classifies maps $A \otimes M \rightarrow M$. Here, the object $A\in \mathcal{C}_{\mathfrak{a}}$ need not carry the structure of an algebra object. However, we do expect $\text{End}(M)$ to carry the structure of an associative algebra coming from composition, and $M$ to be a module over it. 
\begin{definition}\label{def1}
    Let $q: \mathcal{C}^{\otimes}\rightarrow \lm$ be a coCartesian fibration of $\infty$-operads. Define the \textbf{endomorphism} $\infty$\textbf{-category} of $M\in \mathcal{C}_{\mathfrak{m}}$ as
    \begin{align*}
        \mathcal{C}_{\mathfrak{a}}[M] := \mathcal{C}_{\mathfrak{a}} \times_{\mathcal{C}_{\mathfrak{m}}} {\mathcal{C}_{\mathfrak{m}}}_{/M}
    \end{align*}
    where the pullback is given by the forgetful functor and equation (\ref{eq2}).
\end{definition}
\begin{proposition}\label{prop9}
The $\infty$-category $\mathcal{C}_{\mathfrak{a}}[M]$ is the underlying $\infty$-category of a monoidal $\infty$-category with tensor product given up to homotopy by 
\begin{align*}
    (A, A\otimes M \xrightarrow{\alpha} M) \otimes (B, B\otimes M \xrightarrow{\beta} M) = (A \otimes B, A\otimes B \otimes M \xrightarrow{\text{id}\otimes \beta} A \otimes M \xrightarrow{\alpha} M).
\end{align*}
\end{proposition}
\begin{proof}
We show in Appendix \ref{appendixA} that $\mathcal{C}_{\mathfrak{a}}[M]$ agrees with Lurie's endomorphism $\infty$-category as in \cite[Definition 4.7.1.1]{HA}. It is shown in \cite[Proposition 4.7.1.30]{HA} that this admits the structure of a monoidal $\infty$-category with the described underlying tensor product.
\end{proof}
Then by \cite[Corollary 3.2.2.5]{HA} with $K= \emptyset$ and $\mathcal{O}^{\otimes} = \ass$ we automatically get the following.
\begin{proposition}\label{cor6}
    Assume that the $\infty$-category $\mathcal{C}_{\mathfrak{a}}[M]$ has a final object $(\e(M),\alpha)$. Then the $\infty$-category $\alg_{\mathscr{A}ssoc}(\mathcal{C}_{\mathfrak{a}}[M])$ also admits a final object, and if $X$ is final in $\alg_{\mathscr{A}ssoc}(\mathcal{C}_{\mathfrak{a}}[M])$ then we have an equivalence $X(\mathfrak{a})\simeq (\e(M),\alpha)$.
\end{proposition}
Roughly speaking, Proposition \ref{cor6} implies that $\e(M)$ can be "promoted" to an algebra object in an essentially unique way. By abuse of notation, from here on, we denote by $\e(M)$ a final object of $\alg_{\mathscr{A}ssoc}(\mathcal{C}_{\mathfrak{a}}[M])$ with underlying object $\e(M)\in \mathcal{C}_{\mathfrak{a}}[M]$. \\ \par
By Proposition \ref{prop9}, the tensor product of $\e(M)$ with itself in $\mathcal{C}_{\mathfrak{a}}[M]$ is given by
\begin{align*}
    (\e(M) \otimes \e(M))\otimes M \simeq \e(M)\otimes (\e(M)\otimes M) \xrightarrow{\text{id}_{\e(M)}\otimes \alpha} \e(M) \otimes M \xrightarrow{\alpha} M,
\end{align*}
and hence by the proof of \cite[Proposition 3.2.2.1]{HA}, the algebra structure on $\e(M)$ has a multiplication $\circ$ making the following diagram commute in $\mathcal{C}_{\mathfrak{m}}$
\begin{center}
    \begin{tikzcd}
                                                     & \e(M) \otimes M \arrow[rd, "\alpha"] &   \\
(\e(M)\otimes \e(M)) \otimes M \arrow[rr, "\alpha  (\text{id}\otimes \alpha)"'] \arrow[ru, "\circ \otimes \text{id}"] &                            & M
\end{tikzcd}.
\end{center}
The above diagram implies that $M$ is automatically a module over the algebra $\e(M)$, and this action $\e(M) \otimes M \xrightarrow{\alpha} M$ is universal among algebra actions making $M$ into a module. In particular, the fiber over $M$ of the functor $\lmod(\mathcal{C}) \rightarrow \mathcal{C}_{\mathfrak{m}}$ is non-trivial. Moreover, we have
\begin{proposition}
There is an equivalence of $\infty$-categories 
\begin{align*}
    \alg_{\mathscr{A}ssoc}(\mathcal{C}_{\mathfrak{a}}[M]) \xrightarrow{\simeq} \lmod(\mathcal{C})\times_{\mathcal{C}_{\mathfrak{m}}}\{M\}.
\end{align*}
\end{proposition}
\begin{proof}
By Appendix \ref{appendixA}, $\mathcal{C}_{\mathfrak{a}}[M]$ agrees with Lurie's endomorphism $\infty$-category. Then using \cite[Theorem 4.7.1.34]{HA} and \cite[Remark 4.7.1.35]{HA}, we get an equivalence of $\infty$-categories 
\begin{align*}
    \alg_{\mathbb{A}_{\infty}}(\mathcal{C}_{\mathfrak{a}}[M]) \xrightarrow{\simeq} \lmod(\mathcal{C})\times_{\mathcal{C}_{\mathfrak{m}}} \{M\}.
\end{align*}
Using \cite[Proposition 4.1.3.19]{HA}, we also have an equivalence of $\infty$-categories $\alg_{\mathscr{A}ssoc}(\mathcal{C}_{\mathfrak{a}}[M]) \xrightarrow{\simeq} \alg_{\mathbb{A}_{\infty}}(\mathcal{C}_{\mathfrak{a}}[M])$, and the statement follows by composing these.
\end{proof}
There are a variety of interesting situations in which such an (endo)morphism object fails to exist, in particular if we consider $\infty$-categories arising as categories of algebra objects. The archetypal example is the following.
\begin{example}\label{ex1}
    Let $\field{k}$ be a field and consider the (symmetric) monoidal category $\mathcal{C} = \alg_{\field{k}}$ as left tensored over itself. Let $M\in \mathcal{C}$ be a $\field{k}$-algebra. Then for any endomorphism $\varphi \in \hom_{\mathcal{C}}(M,M)$, the pair $(\field{k},\field{k}\otimes M \cong M \xrightarrow{\varphi}M)$ is an object of $\mathcal{C}[M]$. Therefore, if $(A,\alpha)\in \mathcal{C}[M]$ is a final object, then $\alpha\circ (u_A\otimes \text{id}_M) = \varphi$ for any endomorphism $\varphi: M\rightarrow M$, where $u_A$ is the unit of $A$. But this is only possible if $M$ only admits a single endomorphism. This difficulty was to be expected, since we know that the monoidal category of $\field{k}$-algebras is not closed.
\end{example}
Lurie's solution to this problem is to relax the expectations on the morphism object. In the above discussion, we start out requiring that $\mor_{\mathcal{C}_{\mathfrak{m}}}(M,N)$ classify all morphisms $C\otimes M\rightarrow N$ in $\mathcal{C}_{\mathfrak{m}}$, and then in the case $N=M$ get for free that $\e(M)$ also classifies algebra actions of algebras on $M$. Instead, we now consider objects that only classify the algebra actions. 
\begin{definition}\label{def5}
    Let $\mathcal{C}^{\otimes}\rightarrow \lm$ be a coCartesian fibration of $\infty$-operads, and let $M\in \mathcal{C}_{\mathfrak{m}}$. A \textbf{center} $\mathfrak{Z}(M)$ of $M$ is a final object of $\lmod(\mathcal{C})\times_{\mathcal{C}_{\mathfrak{m}}}\{M\}$. We generally identify $\mathfrak{Z}(M)$ with its image in $\alg_{\mathscr{A}ssoc}(\mathcal{C}_{\mathfrak{a}})$.
\end{definition}
Clearly if $M$ admits an endomorphism object, then this endomorphism object is also a center of $M$. The converse does not hold: The category $\alg_{\mathscr{A}ssoc}(\mathcal{C}_{\mathfrak{a}}[M])$ might have final objects although $\mathcal{C}_{\mathfrak{a}}[M]$ does not.
\begin{example}[Example \ref{ex1} continued]
The ordinary center $Z(A)$ of an associative $\field{k}$-algebra $A$ is indeed the universal algebra object acting on $A$. To see this, note first that the center is a commutative algebra, and therefore an algebra object in the category of associative algebras. It comes with a natural action on $A$ given by multiplication in $A$. Now suppose that $B$ is a commutative algebra with action $\beta: B\otimes A \rightarrow A$ making $A$ into a $B$-module. Then the restriction of $\beta$ to $A$ yields the identity on $A$ and $\beta$ must be an algebra morphism. Hence 
\begin{align*}
    \beta(b\otimes a) &= \beta(b\otimes 1)\cdot \beta(1\otimes a) = \beta(b\otimes 1) \cdot a \quad \text{and} \\
    \beta(b\otimes a) &= \beta(1\otimes a) \cdot \beta(b\otimes 1) = a\cdot \beta(b\otimes 1),
\end{align*}
showing that $\beta$ sends $B$ to $Z(A)$.
\end{example}
There also is a relative version of the center. 
\begin{definition}
Let $\mathcal{C}^{\otimes} \rightarrow \lm$ be a coCartesian fibration of $\infty$-operads, let $\mathbb{1}$ denote the monoidal unit of $\mathcal{C}_{\mathfrak{a}}^{\otimes}$, and let $f: M \rightarrow N$ be a morphism in $\mathcal{C}_{\mathfrak{m}}$. A \textbf{centralizer} $\mathfrak{Z}(f)$ of $f$ is a final object in
\begin{align*}
    \textnormal{Act}(f) := (\mathcal{C}_{\mathfrak{a}})_{\mathbb{1}/} \times_{{\mathcal{C}_{\mathfrak{m}}}_{M/}} ({\mathcal{C}_{\mathfrak{m}}}_{M/})_{/f}.
\end{align*}
We generally identify $\mathfrak{Z}(f)$ with its image in $\mathcal{C}_{\mathfrak{a}}$.
\end{definition}
The objects of this $\infty$-category are given by commuting triangles in $\mathcal{C}_{\mathfrak{m}}$
\begin{center}
    \begin{tikzcd}
                                          & C\otimes M \arrow[rd] &   \\
\mathbb{1}\otimes M \arrow[rr,"f"'] \arrow[ru] &                       & N
\end{tikzcd}.
\end{center}
In particular, the centralizer is equipped with an action $\mathfrak{Z}(f) \otimes M \rightarrow N$ making the above diagram commute. 
\begin{lemma}\label{lem5}
Let $\mathcal{C}^{\otimes} \rightarrow \lm$ be a coCartesian fibration of $\infty$-operads. Let $f: M\rightarrow N$ be a morphism in $\mathcal{C}_{\mathfrak{m}}$. Let $\overline{M}\in \lmod_{\mathbb{1}}(\mathcal{C}_{\mathfrak{m}})$ be a lift of $M$ as module over the trivial algebra. Let $\mathcal{C}^{\otimes}_{\overline{M}_{\mathscr{LM}/}}\rightarrow \lm$ be defined as in \cite[Notation 2.2.2.3]{HA}. Then centralizers of $f$ can be identified with morphism objects
\begin{align*}
    \mor_{{\mathcal{C}_{\mathfrak{m}}}_{M/}}(\textnormal{id}_M,f)\in (\mathcal{C}_{\mathfrak{a}})_{\mathbb{1}/}.
\end{align*}
\end{lemma}
\begin{proof}
Let $\mathcal{C}'^{\otimes} := \mathcal{C}^{\otimes}_{\overline{M}_{\lm}}$. By Proposition \ref{prop6}, it suffices to show that we have an equivalence of $\infty$-categories $\text{Act}(f)\simeq (\mathcal{C}')_{\mathfrak{a}} \times_{\mathcal{C}_{\mathfrak{m}}'} {\mathcal{C}_{\mathfrak{m}}'}_{/f}$. But we have $(\mathcal{C}')_{\mathfrak{a}} \times_{\mathcal{C}_{\mathfrak{m}}'} {\mathcal{C}_{\mathfrak{m}}'}_{/f} \simeq (\mathcal{C}_{\mathfrak{a}})_{\mathbb{1}/} \times_{{\mathcal{C}_{\mathfrak{m}}}_{M/}} ({\mathcal{C}_{\mathfrak{m}}}_{M/})_{/f}$, so this is clear.
\end{proof}
We would like to see that these notions are compatible, in the sense that the centralizer of an identity morphism recovers the center. 
\begin{proposition}[Proposition 5.3.1.8 \cite{HA}]\label{prop7}
Let $M \in \mathcal{C}_{\mathfrak{m}}$, and suppose there exists a centralizer $\mathfrak{Z}(\textnormal{id}_M)\in \mathcal{C}_{\mathfrak{a}}$. Then there exists a center $\mathfrak{Z}(M)\in \alg_{\mathscr{A}ssoc}(\mathcal{C}_{\mathfrak{a}})$. Further, a lift of $M$ to a module over an algebra $A\in \alg_{\mathscr{A}ssoc}(\mathcal{C}_{\mathfrak{a}})$ exhibits $A$ as a center of $M$ if and only if the action map $A\otimes M \rightarrow M$ exhibits $A$ as a centralizer of $\textnormal{id}_M$.
\end{proposition}
\begin{proof}[Proof]
By Lemma \ref{lem5}, the centralizer of the identity is a morphism object $\mor_{{\mathcal{C}_{\mathfrak{m}}}_{M/}}(\text{id}_M,\text{id}_M)$. By Proposition \ref{cor6}, this morphism object admits an essentially unique structure of an algebra object in $(\mathcal{C}_{\mathfrak{a}})_{\mathbb{1}/}$, and $\text{id}_M$ lifts to a module over this algebra structure. In particular, $\mathfrak{Z}(\text{id}_M)$ admits a canonical algebra structure making it into the center of $\text{id}_M$ in ${\mathcal{C}_{\mathfrak{m}}}_{M/}$. Now use \cite[Lemma 5.3.1.10]{HA} to see that the forgetful functor $\lmod({\mathcal{C}_{\mathfrak{m}}}_{M/})\times_{{\mathcal{C}_{\mathfrak{m}}}_{M/}}\{\text{id}_M\} \rightarrow \lmod(\mathcal{C})\times_{\mathcal{C}_{\mathfrak{m}}}\{M\}$ preserves final objects.
\end{proof}
\subsection{Tensor product of $\infty$-operads}
The Boardman-Vogt tensor product on ordinary operads is designed such that algebras over the tensor product $\mathcal{P}\boxtimes_{\text{BV}} \mathcal{O}$ are given by $\mathcal{P}$-algebras in the category of $\mathcal{O}$-algebras. However, it is well-known that this tensor product does not make the category of (reduced) operads into a monoidal model category. We briefly review the corresponding construction for $\infty$-operads following \cite[Section 2.5.5]{HA}. \\ \par
We want to capture bilinearity of a map between $\infty$-operads. To this end, define a functor $\wedge: \fin \times \fin \rightarrow \fin$ by sending $(\langle m\rangle,\langle n\rangle)$ to the pointed set $(\langle m\rangle^{\circ}\times \langle n\rangle^{\circ})_{+}\cong \langle mn\rangle$, where the isomorphism is given by the lexicographic ordering, and by sending $(f: \langle m\rangle \rightarrow \langle n\rangle,g: \langle m'\rangle \rightarrow \langle n'\rangle)$ to 
\begin{align*}
    \langle mm'\rangle \xrightarrow{\cong} (\langle m\rangle^{\circ}\times \langle m'\rangle^{\circ})_+ \xrightarrow{f\times g} (\langle n\rangle^{\circ}\times \langle n'\rangle^{\circ})_+ \xrightarrow{\cong} \langle nn'\rangle.
\end{align*}
\begin{definition}
We call a map of simplicial sets $F: \mathcal{O}^{\otimes} \times \mathcal{O}'^{\otimes} \rightarrow \mathcal{O}''^{\otimes}$ a \textbf{bifunctor of $\infty$-operads} if the diagram below commutes, and if $F$ sends pairs of inert maps to an inert map in $\mathcal{O}''^{\otimes}$.
\begin{center}
\begin{tikzcd}
\mathcal{O}^{\otimes}\times \mathcal{O}'^{\otimes} \arrow[d] \arrow[r, "F"] & \mathcal{O}''^{\otimes} \arrow[d] \\
\fin \times \fin \arrow[r, "\wedge"]                                        & \fin                             
\end{tikzcd}
\end{center}
Define $\text{Bil}(\mathcal{O}^{\otimes},\mathcal{O}'^{\otimes};\mathcal{O}''^{\otimes})$ to be the full subcategory of $\text{Fun}_{\fin}(\mathcal{O}^{\otimes}\times \mathcal{O}'^{\otimes},\mathcal{O}''^{\otimes})$ spanned by the bifunctors.
\end{definition}
\begin{proposition}[Proposition 3.2.4.3, \cite{HA}]\label{prop10}
    Let $\mathcal{C}^{\otimes}$ be a symmetric monoidal $\infty$-category and $\mathcal{O}^{\otimes}$ an $\infty$-operad. Then the functor 
    \begin{align*}
        \mathbf{sSet}_{/\fin} \rightarrow \mathbf{Set}
    \end{align*}
    sending $K \rightarrow \fin$ to the set of diagrams
    \begin{center}
 \begin{tikzcd}
K \times \mathcal{O}^{\otimes} \arrow[d] \arrow[r, "F"] & \mathcal{C}^{\otimes} \arrow[d] \\
\fin\times \fin \arrow[r, "\wedge"]                      & \fin                           
\end{tikzcd}
\end{center}
such that for $v\in K$ a vertex and $f$ an inert morphisms in $\mathcal{O}^{\otimes}$, the map $F(s_0(v),f)$ is inert in $\mathcal{C}^{\otimes}$, is representable. The representing object, which we denote by $\alg_{\mathcal{O}}(\mathcal{C})^{\otimes}\rightarrow \fin$, is a coCartesian fibration of $\infty$-operads. 
\end{proposition}
The fiber $\alg_{\mathcal{O}}(\mathcal{C})^{\otimes}_{\langle 1\rangle}$ over $\langle 1\rangle\in \fin$ is given by the full subcategory of $\text{Fun}_{\fin}(\mathcal{O}^{\otimes},\mathcal{C}^{\otimes})$ of maps that preserve inert morphisms, and hence can be identified with the $\infty$-category $\text{Alg}_{\mathcal{O}}(\mathcal{C})$. Hence, this proposition equips the $\infty$-category $\alg_{\mathcal{O}}(\mathcal{C})$ with the structure of a symmetric monoidal $\infty$-category. For $X\in \mathcal{O}$, the evaluation map $\alg_{\mathcal{O}}(\mathcal{C})^{\otimes} \rightarrow \mathcal{C}^{\otimes}$ is a morphism of $\infty$-operads, hence a lax symmetric monoidal functor, and we see that the symmetric monoidal structure on $\alg_{\mathcal{O}}(\mathcal{C})$ is given by the pointwise tensor product in $\mathcal{C}^{\otimes}$.
\begin{corollary}\label{cor11}
There is an equivalence of $\infty$-categories 
\begin{align*}
     \textnormal{Bil}(\mathcal{O}^{\otimes},\mathcal{O}'^{\otimes};\mathcal{O}''^{\otimes}) \simeq \alg_{\mathcal{O}}(\alg_{\mathcal{O}'}(\mathcal{C})),
\end{align*}
where the category $\alg_{\mathcal{O}'}(\mathcal{C})$ is viewed as a symmetric monoidal $\infty$-category via Proposition \ref{prop10}.
\end{corollary}
\begin{proof}
An $\mathcal{O}$-algebra in $\alg_{\mathcal{O}'}(\mathcal{C})^{\otimes}$ is given by a morphism of simplicial sets $\mathcal{O}^{\otimes} \rightarrow \alg_{\mathcal{O}'}(\mathcal{C})^{\otimes}$ over $\fin$ sending inert morphisms to inert morphisms. In particular, by the construction of $\alg_{\mathcal{O}'}(\mathcal{C})^{\otimes}$, such an $\mathcal{O}$-algebra is given by a diagram
\begin{center}
    \begin{tikzcd}
\mathcal{O}^{\otimes}\times \mathcal{O}'^{\otimes} \arrow[d] \arrow[r] & \mathcal{C}^{\otimes} \arrow[d] \\
\fin\times \fin \arrow[r, "\wedge"]                                    & \fin                           
\end{tikzcd}
\end{center}
such that for every $X\in \mathcal{O}^{\otimes}$ and every inert map $f$ in $\mathcal{O}'^{\otimes}$, the tuple $(\text{id}_X,f)$ is sent to an inert map in $\mathcal{C}^{\otimes}$. The condition that inert morphisms in $\mathcal{O}^{\otimes}$ are sent to inert maps in $\alg_{\mathcal{O}'}(\mathcal{C})^{\otimes}$ translates to the fact that for inert maps $f$ in $\mathcal{O}^{\otimes}$ and $X\in \mathcal{O}'$, the tuple $(f,\text{id}_X)$ is sent to an inert map in $\mathcal{C}^{\otimes}$; and together those two conditions say exactly that tuples of inert maps are sent to an inert map. This is true if and only if $F$ is a bifunctor $\mathcal{O}^{\otimes} \times \mathcal{O}'^{\otimes}\rightarrow \mathcal{C}^{\otimes}$.
\end{proof}
\begin{definition}
We say that a bifunctor $F:\mathcal{O}^{\otimes}\times \mathcal{O}'^{\otimes} \rightarrow \mathcal{O}''^{\otimes}$ \textbf{exhibits $\mathcal{O}''^{\otimes}$ as a tensor product} of $\mathcal{O}^{\otimes}$ and $\mathcal{O}'^{\otimes}$ if for every $\infty$-operad $\mathcal{C}^{\otimes}$, pre-composition with $F$ determines an equivalence of $\infty$-categories 
\begin{align*}
    \alg_{\mathcal{O}''}(\mathcal{C}) \rightarrow \text{Bil}(\mathcal{O}^{\otimes},\mathcal{O}'^{\otimes};\mathcal{C}^{\otimes}).
\end{align*}
We say that $\mathcal{O}''^{\otimes}$ is a tensor product of $\mathcal{O}^{\otimes}$ and $\mathcal{O}'^{\otimes}$ if there exists a bifunctor with this property.
\end{definition}
In particular, in this case, if $\mathcal{C}^{\otimes}$ is a symmetric monoidal $\infty$-category, the above discussion shows that we have an equivalence of $\infty$-categories
\begin{align*}
    \alg_{\mathcal{O''}}(\mathcal{C}) \rightarrow \alg_{\mathcal{O}}(\alg_{\mathcal{O}'}(\mathcal{C})).
\end{align*}
\par In this sense, the tensor product of $\infty$-operads is a derived version of the Bordmann-Vogt tensor products of operads.
\subsection{The little cubes operads}\label{the_little_cubes_operads}
As seen in the introduction, we are mainly interested in algebras over the little $k$-disks operads. These are topological operads, and we present a construction to view them as $\infty$-operads. For ease of construction, we consider "little $k$-cubes" instead of "little $k$-disks". \\ \par
Let $\square^k := [0,1]^k$ the $k$-unit cube. Consider the topological one-colored operad $\mathbb{E}_k^T$ with $\mathbb{E}_k^T(n) = \text{Rect}(\square^k \times \{1,\dots,n\},\square^k)$ the space of rectilinear embeddings. Let $\mathbb{E}_k^{T,\otimes}$ denote its topological category of operators, and consider the corresponding simplicial category $\text{Sing}_{\bullet}(\mathbb{E}_k^{T,\otimes})$. Taking the homotopy coherent nerve we obtain an $\infty$-category $\mathbb{E}_k^{\otimes}$, which is an $\infty$-operad since the underlying simplicial operad is fibrant. In particular, objects in $\mathbb{E}_k^{\otimes}$ are given by $\langle n\rangle$ for $n\in \mathbb{N}$, morphisms are given by points in the space
\begin{align*}
    \map_{\mathbb{E}_k^{T,\otimes}}(\langle n\rangle,\langle m\rangle) = \coprod_{f: \langle n\rangle \rightarrow\langle m\rangle} \prod_{j\in \langle m\rangle^{\circ}}  \text{Rect}(\square^k \times \{1,\dots,n\},\square^k),
\end{align*}
and to give a 2-simplex with boundary as shown below is equivalent to giving a path from $F\circ E$ to $G$ in the space $\map_{\mathbb{E}_k^{T,\otimes}}(\langle m\rangle,\langle k\rangle)$.
\begin{center}
    \begin{tikzcd}
                                                  & \langle m\rangle \arrow[rd, "F"] &                  \\
\langle n\rangle \arrow[rr, "G"'] \arrow[ru, "E"] &                                  & \langle k\rangle
\end{tikzcd}
\end{center}
We sometimes also denote the object $\langle n\rangle$ of $\mathbb{E}_k^{\otimes}$ by $\{\underbrace{\mathfrak{a},\dots,\mathfrak{a}}_{n \text{ times}}\}$, where we view $\mathfrak{a}$ as the single color of $\mathbb{E}_k^T$.
\par Following \cite[Notation 2.1.1.16]{HA}, we denote by $\text{Mul}_{\mathbb{E}_k}(\langle n\rangle, \langle 1\rangle)\in h\an$ the fiber of the Kan complex $\map_{\mathbb{E}_k^{\otimes}}(\langle n\rangle, \langle 1\rangle)$ over the active map $\langle n \rangle \rightarrow \langle 1 \rangle$. Note that the mapping spaces of the $\infty$-category $\mathbb{E}_k^{\otimes}$ are weakly equivalent to the mapping spaces in the topological category $\mathbb{E}_k^{T,\otimes}$. We hence have a weak equivalence $\text{Mul}_{\mathbb{E}_k}(\langle n\rangle,\langle 1\rangle)\simeq \text{Rect}(\square^k\times \{1,\dots,n\},\square^k) = \mathbb{E}_k^T(n)$, and we will implicitly use the space $\mathbb{E}_k^T(n)$ as a representative for the weak homotopy type $\text{Mul}_{\mathbb{E}_k}(\langle n\rangle, \langle 1\rangle)$.\\ \par
For $k=1$ we obtain the $\mathbb{E}_1^{\otimes}$-operad, which is equivalent to the $\infty$-operad $\ass$ and governs homotopy associative algebras. We will generally identify algebras over these two $\infty$-operads. For $k=2$, we instead recover the $\infty$-operadic version of the little 2-disks operad $D_2$. In particular, an element in $\text{Mul}_{\mathbb{E}_1}(\langle n\rangle,\langle 1\rangle)$ is given by a rectangular embedding of $n$ copies of the interval $[0,1]$ into the interval $[0,1]$. An element in $\text{Mul}_{\mathbb{E}_2}(\langle n\rangle,\langle 1\rangle)$ is given by a rectangular embedding of $n$ copies of the square $[0,1]\times[0,1]$ into the square $[0,1]\times [0,1]$. 
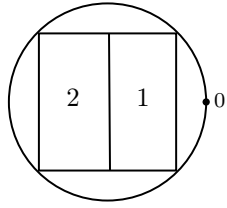
\begin{figure}[h!]
    \centering
    \tikzset{every picture/.style={line width=0.75pt}} 
\begin{tikzpicture}[x=0.75pt,y=0.75pt,yscale=-1,xscale=1]

\draw   (230,109.75) .. controls (230,82.27) and (252.27,60) .. (279.75,60) .. controls (307.23,60) and (329.5,82.27) .. (329.5,109.75) .. controls (329.5,137.23) and (307.23,159.5) .. (279.75,159.5) .. controls (252.27,159.5) and (230,137.23) .. (230,109.75) -- cycle ;
\draw   (245.14,75.14) -- (314.36,75.14) -- (314.36,144.36) -- (245.14,144.36) -- cycle ;
\draw    (280.5,75) -- (281,144.5) ;
\draw  [fill={rgb, 255:red, 0; green, 0; blue, 0 }  ,fill opacity=1 ] (328.17,109.75) .. controls (328.17,109.01) and (328.76,108.42) .. (329.5,108.42) .. controls (330.24,108.42) and (330.83,109.01) .. (330.83,109.75) .. controls (330.83,110.49) and (330.24,111.08) .. (329.5,111.08) .. controls (328.76,111.08) and (328.17,110.49) .. (328.17,109.75) -- cycle ;

\draw (292.83,101.5) node [anchor=north west][inner sep=0.75pt]   [align=left] {1};
\draw (257.5,101.5) node [anchor=north west][inner sep=0.75pt]   [align=left] {2};
\draw (332,104) node [anchor=north west][inner sep=0.75pt]  [font=\footnotesize] [align=left] {0};

\end{tikzpicture}
    \caption{The element $\mu_0\in \map_{\mathbb{E}_2}(\langle 2\rangle,\langle 1\rangle)_0$.}
\end{figure}
\FloatBarrier
Note that for $k\geq 1$, we have a homotopy equivalence $\text{Mul}_{\mathbb{E}_k}(\langle 2\rangle,\langle 1\rangle)\simeq S^{k-1}$. To construct this, fix a circumscribed $(k-1)$-sphere about $[0,1]^k$. Then to each point in $\text{Mul}_{\mathbb{E}_k}(\langle 2\rangle,\langle 1\rangle)\simeq \mathbb{E}_k^T(2)$ given by a rectangular embedding of two $k$-squares into a $k$-square, assign the intersection $S^{k-1}\cap r_{\vec{21}}$, where $r_{\vec{21}}$ is the ray through the centers of the two embedded copies of $[0,1]^k$ which starts at the center of the second copy. We frequently use this homotopy equivalence as a convenient method to label morphisms in the little cubes operads. In particular, fix a homotopy inverse $S^0 \rightarrow \text{Mul}_{\mathbb{E}_1}(\langle 2\rangle,\langle 1\rangle)$ and a homotopy inverse $S^1 \rightarrow \text{Mul}_{\mathbb{E}_2}(\langle 2\rangle,\langle 1\rangle)$. Then we get two points in $\text{Mul}_{\mathbb{E}_1}(\langle 2\rangle,\langle 1\rangle)$ named $\mu_0$ and $\mu_{1}$, and for every $t\in [0,2\pi)$ we get a point $\mu_t\in \text{Mul}_{\mathbb{E}_2}(\langle 2\rangle,\langle 1\rangle)$. We denote by the same letters representatives of these operations as 0-simplices in the respective mapping spaces.\\ \par
 Recall that 2-morphisms\footnote{We use "2-morphism" and "2-simplex" in a $\infty$-category interchangeably. However, we will often favor "2-morphism" in a context where one edge of the 2-simplex is degenerate to emphasize the globular nature of topological and dg categories.} in $\mathbb{E}^{\otimes}_k$ are given by paths in the relevant mapping spaces. There are two such 2-morphisms in $\mathbb{E}^{\otimes}_2$ that will play a special role in the subsequent discussion. On the one hand, for each $t\in [0,2\pi)$, there is a 2-morphisms $\sigma_t\in \map_{\mathbb{E}_2^{\otimes}}(\langle 2\rangle,\langle 1\rangle)_{1}$ with boundary given by $\mu_t$ and $\mu_{t+\pi (\text{mod }2\pi)}$ that is represented by the braid
\begin{figure}[h!]
    \centering
    \tikzset{every picture/.style={line width=0.75pt}} 

\begin{tikzpicture}[x=0.75pt,y=0.75pt,yscale=-1,xscale=1]

\draw    (170.33,70.67) -- (220.67,170.67) ;
\draw    (220.67,70.67) -- (198.67,114.67) ;
\draw    (192,126.67) -- (170,170.67) ;

\draw (223,57.4) node [anchor=north west][inner sep=0.75pt]    {$1$};
\draw (157.33,165.07) node [anchor=north west][inner sep=0.75pt]    {$1$};
\draw (224,166.73) node [anchor=north west][inner sep=0.75pt]    {$2$};
\draw (155,60.07) node [anchor=north west][inner sep=0.75pt]    {$2$};

\end{tikzpicture}.
\end{figure}
\FloatBarrier
On the other hand, for each $t\in [0,2\pi)$, there is a non-trivial 2-morphism $\gamma_t\in  \map_{\mathbb{E}_2^{\otimes}}(\langle 2\rangle,\langle 1\rangle)_{1}$ between $\mu_t$ and itself represented by the double-braid
\begin{figure}[h!]
    \centering
   \tikzset{every picture/.style={line width=0.75pt}} 

\begin{tikzpicture}[x=0.75pt,y=0.75pt,yscale=-1,xscale=1]

\draw    (169.67,50.67) -- (220,150.67) ;
\draw    (220,50.67) -- (198,94.67) ;
\draw    (191.33,106.67) -- (169.33,150.67) ;
\draw    (169.33,150.67) -- (219.67,250.67) ;
\draw    (220,150.67) -- (198,194.67) ;
\draw    (190.67,206) -- (168.67,250) ;

\draw (222.33,37.4) node [anchor=north west][inner sep=0.75pt]    {$1$};
\draw (222.67,249.73) node [anchor=north west][inner sep=0.75pt]    {$1$};
\draw (158.33,250.73) node [anchor=north west][inner sep=0.75pt]    {$2$};
\draw (154.33,40.07) node [anchor=north west][inner sep=0.75pt]    {$2$};

\end{tikzpicture}.
\end{figure}
\FloatBarrier
 The classical Eckmann-Hilton argument shows that, in the 1-categorical case, algebra objects in the category of algebra objects yield commutative algebra objects. In particular, if $(A,\cdot)$ is an associative algebra in a symmetric monoidal category ${C}$ and $\ast: (A,\cdot)\otimes(A,\cdot)\rightarrow (A,\cdot)$ endows $(A,\cdot)$ with the structure of an associative algebra object in $\alg({C})$, then both operations $\cdot$ and $\ast$ agree and they are commutative. In this sense, an $\mathbb{E}_1$-algebra inside the symmetric monoidal category of $\mathbb{E}_1$-algebras in ${C}$ is the same as a commutative algebra inside ${C}$, which in the 1-categorical case is the same as an $\mathbb{E}_2$-algebra. In fact, this pattern continues for all the little $k$-cubes operads, as was shows by Dunn for topological operads and later by Lurie for $\infty$-operads. To explain this, for $k,k'\geq 0$, define a topological functor
\begin{align*}
    \rho: \mathbb{E}_k^{T,\otimes}\times \mathbb{E}_{k'}^{T,\otimes} \rightarrow \mathbb{E}_{k+k'}^{T,\otimes}
\end{align*}
given on objects by $\rho(\langle m\rangle,\langle n\rangle) = \langle m\rangle \wedge \langle n\rangle$, and sending a pair of morphisms $(\alpha,\{f_j: \square^k\times \alpha^{-1}(\{j\})\rightarrow \square^k\}_{j\in \langle n\rangle^{\circ}})$ and $(\beta, \{g_i: \square^{k'}\times \beta^{-1}(\{i\})\rightarrow \square^{k'}\}_{i\in\langle n'\rangle^{\circ}})$ to 
\begin{align*}
    (\alpha \wedge \beta, \{f_j\times g_i:\square^{k+k'}\times \alpha^{-1}(\{j\})\times\beta^{-1}(\{i\})\rightarrow \square^{k+k'}\}_{j\in \langle n\rangle^{\circ},i\in \langle n'\rangle^{\circ}}).
\end{align*}
In order for this to make sense, we note that viewing a tuple $(j,i)\in \langle n\rangle^{\circ}\times \langle n'\rangle^{\circ}$ as an element of $\langle nn'\rangle^{\circ}$, we have $(\alpha\wedge \beta)^{-1}((j,i)) = \alpha^{-1}(\{j)\}\times \beta^{-1}(\{i\})$. This descends to a simplicial functor, and then taking the homotopy coherent nerve to a map of $\infty$-categories $\rho: \mathbb{E}_k^{\otimes}\times \mathbb{E}_{k'}^{\otimes} \rightarrow\mathbb{E}_{k+k'}^{\otimes}$. By construction, the diagram
\begin{center}
    \begin{tikzcd}
\mathbb{E}_k^{\otimes}\times \mathbb{E}_{k'}^{\otimes} \arrow[d] \arrow[r, "\rho"] & \mathbb{E}_{k+k'}^{\otimes} \arrow[d] \\
\fin\times \fin \arrow[r, "\wedge"]                                                & \fin                                 
\end{tikzcd}
\end{center}
commutes, and clearly $\rho$ sends pairs of inert morphisms to inert morphisms. Thus, $\rho$ is a bifunctor of $\infty$-operads.
\begin{theorem}[Dunn Additivity Theorem, Theorem 5.1.2.2\cite{HA}]\label{thm7}
    The bifunctor $\rho: \mathbb{E}_k^{\otimes} \times \mathbb{E}_{k'}^{\otimes} \rightarrow \mathbb{E}_{k+k'}^{\otimes}$ exhibits the $\infty$-operad $\mathbb{E}_{k+k'}^{\otimes}$ as a tensor product of $\mathbb{E}_k^{\otimes}$ and $\mathbb{E}_{k'}^{\otimes}$.
\end{theorem}
This means that for every symmetric monoidal $\infty$-category $\mathcal{C}^{\otimes}$, precomposition with $\rho$ determines an equivalence of $\infty$-categories
\begin{align*}
    \rho^{\ast}: \text{Alg}_{\mathbb{E}_{k+k'}}(\mathcal{C}) \rightarrow \text{Alg}_{\mathbb{E}_{k}}(\text{Alg}_{\mathbb{E}_{k'}}(\mathcal{C})).
\end{align*}
In particular, for every $\mathbb{E}_k$-algebra $A$ in $\mathbb{E}_{k'}$-algebras in $\mathcal{C}^{\otimes}$, there exists a $\mathbb{E}_{k+k'}$-algebra $\tilde{A}$ in $\mathcal{C}^{\otimes}$ such that $\tilde{A}\circ \rho$ is equivalent to $A$ in the $\infty$-category $\text{Alg}_{\mathbb{E}_{k}}(\text{Alg}_{\mathbb{E}_{k'}}(\mathcal{C}))$.\\ \par
Let $\mathcal{C}^{\otimes}$ be a symmetric monoidal $\infty$-category and consider the symmetric monoidal $\infty$-category $\alg_{\mathbb{E}_k}(\mathcal{C})^{\otimes}$ constructed in \ref{prop10}. Recall from Definition \ref{def5} that the center $\mathfrak{Z}(A)$ of an $\mathbb{E}_k$-algebra $A\in \alg_{\mathbb{E}_k}(\mathcal{C})$ carries the structure of an $\mathscr{A}ssoc$-algebra in $\alg_{\mathbb{E}_k}(\mathcal{C})^{\otimes}$. Identifying $\ass$ with $\mathbb{E}_1^{\otimes}$, the Dunn Additivity Theorem then implies that the center is in fact an $\mathbb{E}_{k+1}$-algebra
\begin{align*}
    \mathfrak{Z}(A) \in \alg_{\mathbb{E}_1}(\alg_{\mathbb{E}_k}(\mathcal{C})) \simeq \alg_{\mathbb{E}_{k+1}}(\mathcal{C}).
\end{align*}
\subsection{The center as endomorphism object of bimodules} \label{the_center_as_endomorphism_object_of_bimodules}
We have a particularly nice description for the center in symmetric monoidal $\infty$-categories that arise as the category of $\mathcal{O}$-algebras for a coherent $\infty$-operad $\mathcal{O}$.\\ \par 
First consider the 1-categorical case. There is an equivalence of 1-categories between associative algebra objects in the monoidal category of $A$-bimodules and associative algebras under $A$
\begin{align*}
    \alg(\text{Bimod}_A) \xrightarrow{\simeq} \alg(\text{Mod}_{\field{k}})^{A/}.
\end{align*}
The equivalence is given by sending an algebra in bimodules to its unit morphism. Note that $A$ is the monoidal unit in $\text{Bimod}_A$ and hence $\alg(\text{Bimod}_A)$. The 1-centralizer of an algebra morphism $f: A\rightarrow B$ is defined as 
    \begin{align*}
        Z(f) = \{b\in B: \forall a\in A: f(a)b=bf(a)\}.
    \end{align*}
Viewing the data of $f$ as an $A$-bimodule structure on $B$, we can see that this agrees with the set of $A$-bimodule maps from $A$ to $B$:
    \begin{align*}
        {Z}(f) \cong \hom_{\text{Bimod}_A}(A,B).
    \end{align*}
We explain how this picture generalizes to the $\infty$-categorical case. Let $\mathcal{C}^{\otimes} \rightarrow \fin$ be a symmetric monoidal $\infty$-category and $\mathcal{O}^{\otimes}$ be a coherent $\infty$-operad; we will mostly be interested in the case $\mathcal{O}^{\otimes} = \mathbb{E}_1^{\otimes}$. Consider the unique bifunctor of $\infty$-operads $\mathcal{O}^{\otimes} \times \lm \rightarrow \fin \times \fin \xrightarrow{\wedge} \fin$. We have a coCartesian fibration $\alg_{\mathcal{O}}(\mathcal{C})^{\otimes} \rightarrow \lm$ with fibers over $\mathfrak{a}$ and $\mathfrak{m}$ respectively both equivalent to $\alg_{\mathcal{O}}(\mathcal{C})$ coming from the fact that $\alg_{\mathcal{O}}(\mathcal{C})$ admits the structure of a symmetric monoidal $\infty$-category as constructed in Proposition \ref{prop10}, and is hence left tensored over itself.\\ \par 
Let $A\in \alg_{\mathcal{O}}(\mathcal{C})_{\mathfrak{m}}$. To generalize the relationship between a centeralizer and bimodule morphisms to $\infty$-operads, one needs to first recover the statement about algebra objects in the monoidal category of bimodules. Let $\overline{A}\in \lmod_{\mathbb{1}}(\alg_{\mathcal{O}}(\mathcal{C}))$ be a lift of $A$ as a module over the trivial algebra $\mathbb{1}\in \alg_{\mathbb{E}_1}(\alg_{\mathcal{O}}(\mathcal{C})_{\mathfrak{a}})$. We obtain a coCartesian $\lm$-family of $\mathcal{O}$-operads as defined in \cite[Definition 5.3.1.19]{HA} 
\begin{align*}
    \mathcal{C}^{\otimes} \times_{\fin} (\mathcal{O}^{\otimes} \times \lm) \rightarrow \mathcal{O}^{\otimes} \times \lm,
\end{align*}
and we can regard $\overline{A}$ as a coCartesian $\lm$-family of $\mathcal{O}$-algebras as defined in \cite[Remark 5.3.1.22]{HA} by noting that there is a bijection
\begin{align*}
    \text{Fun}_{\lm}(\lm, \alg_{/\mathcal{O}}^{\lm}(\mathcal{C}^{\otimes} \times_{\fin} (\mathcal{O}^{\otimes} \times \lm))) \cong \text{Fun}_{\lm}(\lm, \alg_{\mathcal{O}}(\mathcal{C})^{\otimes}).
\end{align*}
We can view $A$ as an $\mathcal{O}$-module over itself, and hence $\overline{A}$ also determines a coCartesian $\lm$-family of $\mathcal{O}$-algebras in the coCartesian $\lm$-family of $\mathcal{O}$-operads\footnote{See \cite[Definition 5.3.1.23]{HA}.}
\begin{align*}
    \overline{C}^{\otimes} := \text{Mod}^{\mathcal{O},\lm}_{\overline{A}}(\mathcal{C}^{\otimes} \times_{\fin} (\mathcal{O}^{\otimes} \times \lm))^{\otimes} \rightarrow \mathcal{O}^{\otimes} \times \lm.
\end{align*}
This allows us to identify algebra objects in the category of $A$-bimodules with algebras under $A$:
\begin{proposition}[Proposition 5.3.1.27 \cite{HA}]
The forgetful functor 
\begin{align*}
    \theta: \alg_{/\mathcal{O}}^{\lm}(\textnormal{Mod}^{\mathcal{O},\lm}_{\overline{A}}(\mathcal{C}^{\otimes} \times_{\fin} (\mathcal{O}^{\otimes} \times \lm))^{\overline{A}_{\lm/}} \rightarrow \alg_{/\mathcal{O}}^{\lm}(\mathcal{C}^{\otimes} \times_{\fin} (\mathcal{O}^{\otimes}\times \lm))^{\overline{A}_{\lm/}}
\end{align*}
is an equivalence of $\infty$-categories. 
\end{proposition}
\begin{remark}
Note that for all $s\in\lm$, the algebra $\overline{A}_s\in \alg_{/\mathcal{O}}(\textnormal{Mod}^{\mathcal{O}}_{\overline{A}_s}(\mathcal{C}^{\otimes} \times_{\fin} (\mathcal{O}^{\otimes} \times \{s\})))$ is a trivial algebra, so in particular for $s= \mathfrak{m}$ we get an equivalence 
\begin{align*}
    \theta_{\mathfrak{m}}: \alg_{/\mathcal{O}}(\textnormal{Mod}_{A}^{\mathcal{O}}(\mathcal{C}^{\otimes} \times_{\fin} (\mathcal{O}^{\otimes} \times \{\mathfrak{m}\}))) \rightarrow \alg_{/\mathcal{O}}(\mathcal{C}^{\otimes} \times_{\fin} (\mathcal{O}^{\otimes} \times \{\mathfrak{m}\}))^{A/} \simeq \alg_{\mathcal{O}}(\mathcal{C})_{\mathfrak{m}}^{A/}.
\end{align*}
\end{remark}
Since $\overline{A}(\mathfrak{a})$ is the trivial algebra in $\alg_{\mathcal{O}}(\mathcal{C})_{\mathfrak{a}}$, we also have an equivalence $\overline{\mathcal{C}}^{\otimes}_{\mathfrak{a}} \simeq \mathcal{C}^{\otimes} \times_{\fin} (\mathcal{O}^{\otimes} \times \{\mathfrak{a}\})$, and therefore $\alg^{\lm}_{/\mathcal{O}}(\overline{\mathcal{C}})_{\mathfrak{a}} \simeq \alg_{\mathcal{O}}(\mathcal{C})_{\mathfrak{a}}$. To find the centralizer of $\text{id}_A$ in $\alg_{\mathcal{O}}(\mathcal{C})_{\mathfrak{m}}$, it hence suffices to find the centralizer of $\text{id}_A$ in $\alg_{/\mathcal{O}}(\textnormal{Mod}^{\mathcal{O}}_{A}(\mathcal{C}^{\otimes} \times_{\fin} (\mathcal{O}^{\otimes} \times \{\mathfrak{m}\})))$, in which $A$ is the trivial algebra. We can then use Proposition \ref{prop7} to upgrade this to a center $\mathfrak{Z}(A) \in \alg_{\mathbb{E}_1}(\alg_{\mathcal{O}}(\mathcal{C})_{\mathfrak{a}})$ of $A$. 
\begin{theorem}[Proposition 5.3.1.29 \cite{HA}]\label{thm8}
Suppose that for all $X\in \mathcal{O}$, there exists a morphism object $\mor_{\overline{\mathcal{C}}_{X,\mathfrak{m}}}({A}(X),{A}(X))\in \overline{ \mathcal{C}}_{X,\mathfrak{a}}$. Then there exists a centralizer $\mathfrak{Z}(\textnormal{id}_{A})\in \alg_{/\mathcal{O}}^{\lm}(\overline{\mathcal{C}})_{\mathfrak{a}}$. Furthermore, if $Z\in \alg_{/\mathcal{O}}^{\lm}(\overline{\mathcal{C}})_{\mathfrak{a}}$, then a commutative diagram 
\begin{center}
    \begin{tikzcd}
                                              & Z\otimes A \arrow[rd] &   \\
A \arrow[rr, "\textnormal{id}_A"'] \arrow[ru] &                       & A
\end{tikzcd}
\end{center}
exhibits $Z$ as the centralizer of $\textnormal{id}_A$ if and only if for all $X\in \mathcal{O}$, the induced map $Z(X)\otimes A(X) \rightarrow A(X)$ exhibits $Z(X)$ as a morphism object of $A(X)$ and $A(X)$.
\end{theorem}
\begin{proof}
By definition, the centralizer is a final object of the $\infty$-category
\begin{align*}
    \mathcal{A} := (\alg_{/\mathcal{O}}^{\lm}(\overline{\mathcal{C}})_{\mathfrak{a}})_{\mathbb{1}} \times_{(\alg_{/\mathcal{O}}^{\lm}(\overline{\mathcal{C}})_{\mathfrak{m}})_{A/}} (\alg_{/\mathcal{O}}^{\lm}(\overline{\mathcal{C}})_{\mathfrak{m}})_{A//A}.
\end{align*}
Since $\overline{A}_s$ is the trivial algebra in $\overline{\mathcal{C}}_s^{\otimes}$ for $s\in \{\mathfrak{a},\mathfrak{m}\}$, we can use \cite[Theorem 2.2.2.4]{HA} to get an $\mathcal{O}^{\otimes}$-monoidal $\infty$-category 
\begin{align*}
    \mathcal{E}^{\otimes} := (\overline{\mathcal{C}}_{\mathfrak{a}}^{\otimes})_{\mathbb{1}_{\mathcal{O}/}} \times_{(\overline{\mathcal{C}}_{\mathfrak{m}}^{\otimes})_{A_{\mathcal{O}/}}} (\overline{\mathcal{C}}_{\mathfrak{m}}^{\otimes})_{A_{\mathcal{O}//A_{\mathcal{O}}}} \rightarrow \mathcal{O}^{\otimes}
\end{align*}
such that $\alg_{/\mathcal{O}}(\mathcal{E}) \simeq \mathcal{A}$. Finally, use that limits in algebra categories are computed object-wise by \cite[Corllary 3.2.2.5]{HA} to argue that we are reduced to showing that for each $X\in \mathcal{O}$, the fiber $\mathcal{E}_X$ admits a final object. But a final object in 
\begin{align*}
    \mathcal{E}_X \simeq (\overline{\mathcal{C}}_{X,\mathfrak{a}})_{\mathbb{1}(X)/} \times_{(\overline{\mathcal{C}}_{X,\mathfrak{m}})_{A(X)/}}(\overline{\mathcal{C}}_{X,\mathfrak{m}})_{A(X)//A(X)}
\end{align*}
is equivalent to a morphism object $\mor_{\overline{\mathcal{C}}_{X,\mathfrak{m}}}(A(X),A(X))$ by Proposition \ref{prop6}.
\end{proof}
\begin{corollary}\label{cor8}
Let $\mathcal{O}^{\otimes}= \mathbb{E}_1^{\otimes}$, and let $A\in \alg_{\mathbb{E}_1}(\mathcal{C})$ be an $\mathbb{E}_1$-algebra in a symmetric monoidal $\infty$-category $\mathcal{C}^{\otimes}$. Assume that the morphism object $\mor_{\textnormal{Mod}_A^{\mathbb{E}_1}(\mathcal{C}^{\otimes} \times_{\fin} \mathbb{E}_1^{\otimes})_{\mathfrak{a}}}(A,A)\in \mathcal{C}$ exists\footnote{Note that here we implicitly view $A$ as a module over itself, i.e. as an object $A\in \textnormal{Mod}_A^{\mathbb{E}_1}(\mathcal{C}^{\otimes} \times_{\fin} (\mathbb{E}_1^{\otimes}\times \{\mathfrak{m}\}))_{\mathfrak{a}}$.}. Then there exists a centralizer $\mathfrak{Z}(\textnormal{id}_A) \in \alg_{\mathbb{E}_1}(\mathcal{C})$ with underlying object 
\begin{align*}
\mathfrak{Z}(\textnormal{id}_A)(\mathfrak{a}) \simeq \mor_{\textnormal{Mod}_A^{\mathbb{E}_1}(\mathcal{C}^{\otimes} \times_{\fin} \mathbb{E}_1^{\otimes})_{\mathfrak{a}}}(A,A),
\end{align*}
and the action of the centralizer has underlying map given by the evaluation $\alpha$ of the morphism object on $A$. Further, the multiplication of the $\mathbb{E}_1$-algebra structure on $\mathfrak{Z}(\textnormal{id}_A)$ is induced by the action of the tensor product $\mathfrak{Z}(\textnormal{id}_A)(\mathfrak{a}) \otimes \mathfrak{Z}(\textnormal{id}_A)(\mathfrak{a})$ on $A$ given by the tensor product $\alpha \otimes \alpha$ in $\mathcal{E}_{\mathfrak{a}}$:
\begin{align*}
    \ast: A \xrightarrow{\simeq} A\otimes A \rightarrow (\mathfrak{Z}(\textnormal{id}_A)(\mathfrak{a}) \otimes A) \otimes (\mathfrak{Z}((\textnormal{id}_A)(\mathfrak{a}) \otimes A) \rightarrow A \otimes A \xrightarrow{\textnormal{mult.}} A.
\end{align*}
\end{corollary}
\begin{proof}
    The first part follows directly from Theorem \ref{thm8}. For the claim about the multiplication, note that by the proof of Theorem \ref{thm8} the algebra structure on the centralizer is induced from $\bigl(\mathfrak{Z}(\text{id}_A)(\mathfrak{a}), \alpha\bigr)$ being final in $\mathcal{E}_{\mathfrak{a}}$. In particular, in $\mathcal{E}_{\mathfrak{a}}$ we have a unique-up-to-contractible-choice map 
    \begin{align*}
        \mathfrak{Z}(\text{id}_A)(\mathfrak{a}) \otimes \mathfrak{Z}(\text{id}_A)(\mathfrak{a}) \rightarrow \mathfrak{Z}(\text{id}_A)(\mathfrak{a}).
    \end{align*}
    Now recall the construction of the tensor product in operadic slice categories in \cite[Theorem 2.2.2.4]{HA}. This shows that up to homotopy, the monoidal product in $\mathcal{E}_{\mathfrak{a}}$ is given in the first component by 
    \begin{center}
        \begin{tikzcd}
\mathbb{1} \arrow[dd] &          & \mathbb{1} \arrow[dd] &   & \mathbb{1} \arrow[d, "\simeq"']         \\
                      & \otimes  &                       & = & \mathbb{1} \otimes \mathbb{1} \arrow[d] \\
X                     &          & Y                     &   & X\otimes Y                             
\end{tikzcd}
    \end{center}
    and in the second component by 
    \begin{center}
       \begin{tikzcd}
A \arrow[d, "\simeq"'] \arrow[rdd, "\textnormal{id}_{A}"] &                 &                                                                                            & A \arrow[d, "\simeq"'] \arrow[rdd, "\textnormal{id}_{A}"] &                 \\
\mathbb{1} \otimes A \arrow[d]                                          &                 & \otimes                                                                                    & \mathbb{1} \otimes A \arrow[d]                                          &                 \\
X \otimes A \arrow[r]                                                   & A &                                                                                            & Y\otimes A \arrow[r]                                                    & A \\
                                                                                      &                 & A \arrow[d, "\simeq"'] \arrow[rrddd, "\textnormal{id}_{A}"]    &                                                                                       &                 \\
                                                                                      &                 & A\otimes A \arrow[d, "\simeq"']                                &                                                                                       &                 \\
                                                                                      & =               & (\mathbb{1} \otimes A) \otimes (\mathbb{1}\otimes A) \arrow[d] &                                                                                       &                 \\
                                                                                      &                 & (X\otimes A) \otimes (Y \otimes A) \arrow[r]                   & A\otimes A \arrow[r, "\textnormal{mult}"']                & A
\end{tikzcd}
    \end{center}
    This corresponds to the tensor product action given in the statement. 
\end{proof}
\begin{corollary}\label{cor7}
    Let $\mathcal{O}^{\otimes} = \mathbb{E}_1^{\otimes}$ and $A\in \alg_{\mathbb{E}_1}(\mathcal{C})$, and assume that the morphism object exists. Then there exists a center $\mathfrak{Z}(A)\in \alg_{\mathbb{E}_1}(\alg_{\mathbb{E}_1}(\mathcal{C}))$ with underlying object $\mathfrak{Z}(A)(\mathfrak{a})(\mathfrak{a}) \simeq \mor_{\textnormal{Mod}_A^{\mathbb{E}_1}(\mathcal{C}^{\otimes} \times_{\fin} \mathbb{E}_1^{\otimes})_{\mathfrak{a}}}(A,A)$. The outer multiplication is given by the composition product
    \begin{align*}
       \circ:  (\mathfrak{Z}(A)(\mathfrak{a})(\mathfrak{a}) \otimes \mathfrak{Z}(A)(\mathfrak{a})(\mathfrak{a})) \otimes A \xrightarrow{\textnormal{id}\otimes \alpha} \mathfrak{Z}(A)(\mathfrak{a})(\mathfrak{a}) \otimes A \xrightarrow{\alpha} A.
    \end{align*}
    The inner multiplication is given by the convolution product $\ast$ described in Corollary \ref{cor8}. View the center as an object in $\alg_{\mathbb{E}_1}(\alg_{/\mathbb{E}_1}(\mathcal{E}))$ with the monoidal structure on $\alg_{/\mathbb{E}_1}(\mathcal{E})$ given by Proposition \ref{prop9}. The two multiplications assemble into a square 
    \begin{center}
        \begin{tikzcd}
\Bigl(\mathfrak{Z}(A)(\mathfrak{a})\otimes \mathfrak{Z}(A)(\mathfrak{a})\Bigr)(\mathfrak{a})\otimes \Bigl(\mathfrak{Z}(A)(\mathfrak{a})\otimes \mathfrak{Z}(A)(\mathfrak{a})\Bigr)(\mathfrak{a}) \arrow[d] \arrow[rr] &  & \mathfrak{Z}(A)(\mathfrak{a})(\mathfrak{a}) \otimes \mathfrak{Z}(A)(\mathfrak{a})(\mathfrak{a}) \arrow[d, "\ast"] \\
\Bigl(\mathfrak{Z}(A)(\mathfrak{a})\otimes \mathfrak{Z}(A)(\mathfrak{a})\Bigr)(\mathfrak{a}) \arrow[rr, "\circ"']                                                                                           &  & \mathfrak{Z}(A)(\mathfrak{a})(\mathfrak{a})                                                                      
\end{tikzcd}
    \end{center}
    in $\mathcal{E}_{\mathfrak{a}}$. Since $\mathfrak{Z}(A)(\mathfrak{a})(\mathfrak{a})$ is the morphism object and thus final, there is a contractible choice of 2-simplices filling this square.
\end{corollary}
\begin{proof}
    By Proposition \ref{prop7}, the action $\circ$ above is the one induced by the monoidal structure on $\alg_{/\mathbb{E}_1}(\mathcal{E})$. The square is given by evaluating $\mathfrak{Z}(A)$ at the square 
    \begin{equation}\label{dig2}
          \begin{tikzcd}[column sep = large]
{(\langle 2\rangle,\langle 2\rangle)} \arrow[d, "{(\mu_0,\textnormal{id}_{\langle 2\rangle})}"'] \arrow[r, "{(\textnormal{id}_{\langle 2\rangle},\mu_0)}"] & {(\langle 2\rangle,\langle 1\rangle)} \arrow[d, "{(\mu_0,\textnormal{id}_{\langle 1\rangle})}"] \\
{(\langle 1\rangle,\langle 2\rangle)} \arrow[r, "{(\textnormal{id}_{\langle 1\rangle},\mu_0)}"']                                                 & {(\langle 1\rangle,\langle 1\rangle)}                                                
\end{tikzcd}
     \end{equation}
in $\mathbb{E}_1^{\otimes} \times \mathbb{E}_1^{\otimes}$, i.e. it is given by 
    \begin{center}
        \begin{tikzcd}
{\mathfrak{Z}(A)(\{\mathfrak{a},\mathfrak{a}\})(\{\mathfrak{a},\mathfrak{a}\})} \arrow[d] \arrow[r] & {\mathfrak{Z}(A)(\mathfrak{a})(\{\mathfrak{a},\mathfrak{a}\})} \arrow[d] \\
{\mathfrak{Z}(A)(\{\mathfrak{a},\mathfrak{a}\})(\mathfrak{a})} \arrow[r]                            & \mathfrak{Z}(A)(\mathfrak{a})(\mathfrak{a})                             
\end{tikzcd}.
    \end{center}
\end{proof}
\begin{remark}
    Note that Corollary \ref{cor7} combined with Theorem \ref{thm7} implies that the morphism object completely determines the $\mathbb{E}_2$-algebra structure of the center.
\end{remark}
\subsection{Deligne's conjecture on Hochschild cochains}\label{delignes_conjecture_on_hochschild_cochains}
We recall Deligne's Conjecture on the algebraic structure on Hochschild cochains in cases of an associative $\field{k}$-algebra and a smooth variety. \\ \par
Let $A$ be an associative $\field{k}$-algebra. The 1-category of left modules over the algebra $A^e:=A\otimes A^{\text{op}}$ is isomorphic to the 1-category of $A$-bimodules.
\begin{definition}\label{def2}
The \textbf{Hochschild cochain complex} of $A$ with coefficients in an $A$-bimodule $M$ is given by
\begin{align*}
    C^{\ast}(A,M) = \mathbb{R}\hom_{A\otimes A^{\text{op}}}(A,M).
\end{align*}
This is well-defined up to quasi-isomorphism. The Hochschild cohomology $\text{HH}^{\ast}(A,M)$ of $A$ with coefficients in $M$ is given by the cohomology of this complex. We will mainly be interested in the case $M=A$.
\end{definition}
The Hochschild cohomology of $A$ encodes the deformation theory of $A$ as a bimodule over itself. In particular
\begin{align*}
    \text{HH}^0(A,A) = \hom_{A\otimes A^{\text{op}}}(A,A) = Z(A).
\end{align*}
\par The Hochschild cohomology groups can be viewed as the Ext-groups in the category of bimodules, and can be computed as the cohomology of 
\begin{align*}
    \hom_{A\otimes A^{\text{op}}}(P,A)
\end{align*}
for any projective resolution $P \xtwoheadrightarrow{\simeq} A$.
In practice, one usually takes the bar complex $B(A)$ of $A$ to get the Hochschild cochain complex
\begin{align*}
    \hom_{A\otimes A^{\text{op}}}(B(A),A) \cong \hom_{\field{k}}(A^{\otimes \ast},A).
\end{align*}
\par In 1962, M. Gerstenhaber \cite{G} noticed that Hochschild cohomology is equipped with a shifted Lie bracket $[-,-]_G$ and a commutative product $\smile$ such that $[f,-]_G$ is a derivation for the product. Such a structure is now called a \textbf{Gerstenhaber algebra}. We denote the dg operad governing Gerstenhaber algebras by $\ger$. By a result of F. Cohen \cite{Coh}, the Gerstenhaber operad is the singular homology of the topological operad of little 2-disks $D_2$. In particular, if $A$ is an algebra over $C_{\ast}(D_2)$, then $H_{\ast}(A)$ is an algebra over $\ger$. This inspired P. Deligne in a 1993 letter to make the following conjecture.
\begin{conjecture}[Deligne's Conjecture]
    The Hochschild cochain complex of an associative $\field{k}$-algebra $A$ is an algebra over the chains on little 2-disks operad in such a way that the induced Gerstenhaber algebra structure on cohomology recovers Gerstenhaber's original one.
\end{conjecture}
Note that the category of $A$-bimodules admits a monoidal structure by taking the tensor product over $A$. In his letter, Deligne explained that the $C_{\ast}(D_2)$-algebra structure on the cochain level should come from the two multiplications on $\mathbb{R}\hom_{A\otimes A^{\text{op}}}(A,A)$ induced by $A$ being a monoidal unit in this category. We have an "inner multiplication" coming from the bialgebra structure of $A$ in the bimodule category: Let $\e(A) := \hom_{A\otimes A^{\text{op}}}(A,A)$. Then we get
\begin{align*}
    \e(A) \otimes_A \e(A) \otimes_A A \xrightarrow{\cong} \e(A) \otimes_A \e(A) \otimes_A A \otimes_A A \\\cong \e(A) \otimes_A A \otimes_A \e(A) \otimes_A A \rightarrow A \otimes_A A \xrightarrow{\cong} A
\end{align*}
inducing a multiplication $\e(A) \otimes_A \e(A) \rightarrow \e(A)$. We also have an "outer multiplication" given by composition. \\ \par
Multiple different proofs have been given for Deligne's Conjecture, see for example \cite{Tam}, \cite{Vor}, \cite{MCS}. In particular, D. Tamarkin in \cite{Tam} constructed a map $\Psi_T: \ger_{\infty} \rightarrow \braces$ from the operad of homotopy Gerstenhaber algebras to the Braces-operad, and also proved a formality result for the Gerstenhaber operad, implying that $\ger_{\infty}$ is quasi-isomorphic to $C_{\ast}(D_2)$. Since $C^{\ast}(A,A)$ is canonically a Braces-algebra \cite{GV}, this solves Deligne's Conjecture. Notably, Tamarkin's map $\Psi_T$ depends on the choice of a Drinfeld Associator.\\ \par
There is a global version of Deligne's Conjecture. Let $X$ be a separated finite type scheme over $\field{k}$. Then the Hochschild cohomology of $X$ should be a global version of the above definitions, such that if $X = \text{Spec}(A)$, we recover Definition \ref{def2}. There have been multiple proposed definitions, for example by Grothendieck \cite{Gr} and Loday \cite{Lo}, Gerstenhaber-Schack \cite{GS} and Swan \cite{S}. The perhaps most straight-forward generalization from the affine case was given by Swan, who defined the Hochschild cohomology of $X$ to be
\begin{align*}
    \text{HH}_S^{\ast}(X) = \text{Ext}^{\ast}_{\mathcal{O}_{X\times_{\field{k}} X}}(\Delta_{\ast}\mathcal{O}_X,\Delta_{\ast} \mathcal{O}_X).
\end{align*}
Unfortunately, this does not carry a Gerstenhaber bracket. For the case that $X$ is smooth, Kontsevich \cite{K} gave an alternative definition, which does carry the structure of a Gerstenhaber algebra. He defined the complex of polydifferential operators on a regular algebra $A$ to be the subcomplex
\begin{align*}
    D_{\text{poly}}(A) \subseteq C^{\ast}(A,A)
\end{align*}
of maps $f: A^{\otimes n} \rightarrow A$ that are differential operators in each variable separately. These glue together to yield the \textbf{sheaf of polydifferential operators} $\mathcal{D}_{\text{poly}}(X)$ on $X$ with sections
\begin{align*}
    \mathcal{D}_{\text{poly}}(X)(\text{Spec}(A)) = D_{\text{poly}}(A).
\end{align*}
Then the Hochschild cohomology of $X$ is defined as the hypercohomology of the sheaf of polydifferential operators,
\begin{align*}
    \text{HH}_K^{\ast}(X) = \mathbb{H}^{\ast}(X,\mathcal{D}_{\text{poly}}(X)).
\end{align*}
Since $\mathcal{D}_{\text{poly}}(X)$ is a sheaf of homotopy Gerstenhaber algebras by Tamarkin's solution of the Deligne conjecture, the hypercohomology inherits a Gerstenhaber algebra structure.\\ \par
In Section\ref{the_infty_category_of_dg_sheaves}, we suggest another definition for the Hochschild complex of a separated quasi-compact scheme $X$ over $\field{k}$, that is given in terms of the $\mathbb{E}_1$ center of the structure sheaf in an $\infty$-category of dg sheaves on $X$. This definition automatically comes equipped with a canonical little 2-disks algebra structure, and therefore a priori satisfies Deligne's Conjecture. The difficulty instead lies in proving that this is a sensible definition for the Hochschild complex. This is the interpretation of Theorem \ref{thm1} in the affine case and Theorem \ref{thm6} in the case of a smooth scheme.
\section{The bracket operation on 2-algebras}
The classical Eckmann-Hilton argument shows that two compatible associative algebra structures on a set are equivalent to a commutative algebra structure. If the two associative multiplications are labeled $\ast$ and $\cdot$, the argument first identifies the multiplication $\cdot$ with the opposite multiplication of $\ast$
\begin{align*}
    a\cdot b = (1\ast a)\cdot (b\ast 1) = (1\cdot b)\ast(a \cdot 1) = b \ast a,
\end{align*}
and then further identifies the opposite multiplication of $\ast$ with the opposite multiplication of $\cdot$
\begin{align*}
    b\ast a = (b\cdot 1)\ast(1\cdot a) = (b\ast 1)\cdot (1\ast a) = b \cdot a.
\end{align*}
Together, this yields $a\cdot b = b\cdot a$. The interchange law $(a\ast b) \cdot (c \ast d) = (a\cdot c) \ast( b\cdot d)$ can be expressed as the commutativity of a certain square\footnote{This is the first diagram in the Theorem \ref{thm11}.} in $\mathbb{E}_1^{\otimes} \times \mathbb{E}_1^{\otimes}$. The identities of the form $a\cdot b = (1\ast a)\cdot (b\ast 1)$ correspond to certain semi-inert restriction maps\footnote{These are the triangle diagrams in Theorem \ref{thm11}.} in $\mathbb{E}_1^{\otimes} \times \mathbb{E}_1^{\otimes}$. In particular, in the $\infty$-categorical setting, all the identities used in the Eckmann-Hilton argument are witnessed by 2-simplices in the product $\mathbb{E}_1^{\otimes} \times \mathbb{E}_1^{\otimes}$. This yields a non-trivial homotopy $a\cdot b \simeq b\cdot a$. Repeating this argument with the respective opposite multiplications, i.e. with the roles of $a$ and $b$ reversed, yields another homotopy $b\cdot a \simeq a\cdot b$. Together, we obtain a non-trivial homotopy $a\cdot b \simeq a\cdot b$. In this section, we show that this corresponds to the non-trivial 2-morphism in $\mathbb{E}_2^{\otimes}$ given by the double twist $\gamma_0$.
\subsection{The bracket operation of an $\mathbb{E}_2$-algebra}\label{the_bracket_operation_of_an_e2_algebra}
Consider the topological operad $\mathbb{E}^T_2$ of little 2-squares defined in Section \ref{the_little_cubes_operads} and its corresponding dg operad $C_{\ast}(\mathbb{E}^T_2)$. If $C_{\ast}(\mathbb{E}^T_2) \rightarrow \mathbf{End}(A)$ is an operad algebra in $\ch(\field{k})$, we have an action of the 2-ary operation space
    \begin{align*}
        C_{\ast}(\mathbb{E}^T_2(2)) \otimes A^{\otimes 2} \rightarrow A,
    \end{align*}
and recalling that $\mathbb{E}^T_2(2) \simeq S^1$, taking homology yields a map
    \begin{align*}
        H_{\ast}(S^1) \otimes H_{\ast}(A)^{\otimes 2} \rightarrow H_{\ast}(A).
    \end{align*}
Since $H_{\ast}(S^1)\cong \mathbb{Z}[p] \oplus \mathbb{Z}[\gamma]$ for some choice of basepoint $p\in S^1$ and generating loop $\gamma: [0,1] \rightarrow S^1$, this yields two 2-ary operations on $H_{\ast}(A)$; one of degree 0 induced by $[p]$
    \begin{align*}
        \smile: H_{\ast}(A)\otimes H_{\ast}(A) \rightarrow H_{\ast}(A)
    \end{align*}
and one of degree 1 induced by $[\gamma]$
    \begin{align*}
        [\cdot,\cdot]: H_{\ast}(A) \otimes H_{\ast}(A) \rightarrow H_{\ast}(A)[-1].
    \end{align*}
Cohen showed in \cite{Coh} that these two operations make $H_{\ast}(A)$ into a Gerstenhaber algebra. In particular, the bracket operation on $H_{\ast}(A)$ is induced by the chain level operation $A^{\otimes 2} \rightarrow A$ corresponding to a choice of generating loop $\gamma$ of the homology of $S^1$. 
\begin{definition}
    Let $\mathcal{C}^{\otimes}$ be a symmetric monoidal $\infty$-category and let $A: \mathbb{E}_2^{\otimes} \rightarrow \mathcal{C}^{\otimes}$ be an $\mathbb{E}_2$-algebra in $\mathcal{C}^{\otimes}$. For $t\in [0,2\pi)$, we call the image under $A$ of $\gamma_t\in \map_{\mathbb{E}_2^{\otimes}}(\langle 2\rangle,\langle 1\rangle)_{1}$ the \textbf{bracket operation} of $A$ at $\mu_t\in \map_{\mathbb{E}_2^{\otimes}}(\langle 2\rangle, \langle 1 \rangle)_0$.
\end{definition}
\subsection{The bracket operation of a 2-algebra}\label{the_bracket_operation_of_a_2_algebra}
We now consider the special case of $\mathbb{E}_2$-algebras obtained via the Dunn Additivity Theorem \ref{thm7}. We have seen at the end of Section \ref{the_little_cubes_operads} that these are precisely the types of $\mathbb{E}_2$-algebras obtained as the center of an $\mathbb{E}_1$-algebra in a symmetric monoidal $\infty$-category. The main result of this section, Corollary \ref{cor1}, will enable us to recover the classical Gerstenhaber bracket in cohomology when defining the Hochschild complex as a center.
\begin{definition}
Let $\mathcal{C}^{\otimes}$ be a symmetric monoidal $\infty$-category. A \textbf{2-algebra} in $\mathcal{C}^{\otimes}$ is a bifunctor of $\infty$-operads $A: \mathbb{E}_1^{\otimes} \times \mathbb{E}_1^{\otimes} \rightarrow \mathcal{C}^{\otimes}$.
\end{definition}
\begin{remark}
Note that by Corollary \ref{cor11}, a 2-algebra is equivalently an object $A\in \alg_{\mathbb{E}_1}(\alg_{\mathbb{E}_1}(\mathcal{C}))$.
\end{remark}
Fix a 2-algebra $A: \mathbb{E}_1^{\otimes} \times \mathbb{E}_1^{\otimes} \rightarrow \mathcal{C}^{\otimes}$. The Dunn Additivity Theorem \ref{thm7} tells us that there exists an $\mathbb{E}_2$-algebra $\tilde{A}: \mathbb{E}_2^{\otimes} \rightarrow \mathcal{C}^{\otimes}$ such that the restriction of $\tilde{A}$ along $\rho: \mathbb{E}_1^{\otimes} \times \mathbb{E}_1^{\otimes} \rightarrow \mathbb{E}_2^{\otimes}$ is equivalent to $A$ in the category of bifunctors. Fixing such an $\mathbb{E}_2^{\otimes}$-algebra $\tilde{A}$, we can ask whether it is possible to express the bracket operations $\tilde{A}(\gamma_t)$ in terms of the original 2-algebra $A$. 
\begin{notation}
\begin{itemize}
    \item By abuse of notation, we denote by $A\in \mathcal{C}$ the image $A(\langle 1\rangle,\langle1 \rangle)$ and similarly by $\tilde{A} \in \mathcal{C}$ the image $\tilde{A}(\langle 1\rangle)$.
    \item We commonly label a morphism $f:\langle m\rangle \rightarrow \langle n\rangle$ in $\fin$ by its sequence of images $(f(1),\dots,f(m))$. We denote by $\tau: \langle 2\rangle \rightarrow \langle 2\rangle$ the switch map $\tau = (2,1)$, and by $\iota_{i,j}: \langle 2\rangle \rightarrow \langle 4\rangle$ the inclusion $1\mapsto i$, $2\mapsto j$.
    \item A map $F\in \map_{\mathbb{E}_k^{\otimes}}(\langle m\rangle,\langle n\rangle)_0$ is given by specifying a map $f: \langle m\rangle \rightarrow \langle n\rangle$ in $\fin$ and, for every $j\in \langle n\rangle^{\circ}$, specifying an element $F_j\in \mathbb{E}_k^T(f^{-1}(j))$. If the underlying map $f$ is clear, we denote such a morphism by the sequence $(F_1,\dots, F_n)$ of rectangular embeddings.
    \item Let $f: \langle m\rangle \rightarrow \langle n\rangle$ be semi-inert\footnote{See \cite[Definition 3.3.1.1]{HA}.} in $\fin$. We fix a coCartesian lift $F$ of $f$ in $\mathbb{E}_k^{\otimes}$ by choosing $F_j \in \mathbb{E}_k^T(f^{-1}(j))$ to be the identity map $\square^k \times \{f^{-1}(j)\} \rightarrow \square^k$ for each $j\in \langle n\rangle^{\circ}$ such that $f^{-1}(j)$ is non-empty. We denote those maps by the semi-inert map in $\fin$ they lift. 
\end{itemize}
\end{notation}
Recall that $\mu := \mu_0\in \map_{\mathbb{E}_1^{\otimes}}(\langle 2\rangle, \langle 1\rangle)_0$ corresponds to the following embedding of little intervals
\begin{center}
    \begin{tikzpicture}[x=0.75pt,y=0.75pt,yscale=-.8,xscale=.8]
    \draw    (220,100.5) -- (341,100.5) ;
    \draw [shift={(341,100.5)}, rotate = 180] [color={rgb, 255:red, 0; green, 0; blue, 0 }  ][line width=0.75]    (0,5.59) -- (0,-5.59)   ;
    \draw [shift={(280.5,100.5)}, rotate = 180] [color={rgb, 255:red, 0; green, 0; blue, 0 }  ][line width=0.75]    (0,5.59) -- (0,-5.59)   ;
    \draw [shift={(220,100.5)}, rotate = 180] [color={rgb, 255:red, 0; green, 0; blue, 0 }  ][line width=0.75]    (0,5.59) -- (0,-5.59)   ;
    \draw (306,82) node [anchor=north west][inner sep=0.75pt]   [align=left] {1};
    \draw (244,80.5) node [anchor=north west][inner sep=0.75pt]   [align=left] {2};
\end{tikzpicture}
\end{center}
The key observation in expressing the bracket operations in terms of the orignal 2-algebra is given by the following theorem.
\begin{theorem}\label{thm11}
     The images under $A$ of the 2-simplices in $\mathbb{E}^{\otimes}_1\times \mathbb{E}^{\otimes}_1$ 
     \begin{center}
          \begin{tikzcd}[column sep = large]
{(\langle 2\rangle,\langle 2\rangle)} \arrow[d, "{(\mu,\textnormal{id}_{\langle 2\rangle})}"'] \arrow[r, "{(\textnormal{id}_{\langle 2\rangle},\mu)}"] & {(\langle 2\rangle,\langle 1\rangle)} \arrow[d, "{(\mu,\textnormal{id}_{\langle 1\rangle})}"] \\
{(\langle 1\rangle,\langle 2\rangle)} \arrow[r, "{(\textnormal{id}_{\langle 1\rangle},\mu)}"']                                                 & {(\langle 1\rangle,\langle 1\rangle)}                                                
\end{tikzcd}
     \end{center}
     
     \begin{center}
          \begin{tikzcd}[column sep = large]
{(\langle 1\rangle,\langle 2\rangle)} \arrow[d, "{(\textnormal{id}_{\langle 1\rangle},(2,3))}"'] \arrow[rd, "{(\textnormal{id}_{\langle 1\rangle},\textnormal{id}_{\langle 2\rangle})}"] &                                       \\
{(\langle 1\rangle,\langle 4\rangle)} \arrow[r, "{(\textnormal{id}_{\langle 1\rangle},(\mu,\mu))}"']                                                                             & {(\langle 1\rangle,\langle 2\rangle)}
\end{tikzcd}
     \end{center}

     \begin{center}
         \begin{tikzcd}[column sep = large]
{(\langle 2\rangle,\langle 1\rangle)} \arrow[rd, "{(\tau,\textnormal{id}_{\langle 1\rangle})}"'] \arrow[r, "{((2,3),\textnormal{id}_{\langle 1\rangle})}"] & {(\langle 4\rangle,\langle 1\rangle)} \arrow[d, "{((\mu,\mu),\textnormal{id}_{\langle 1\rangle})}"] \\
                                                                                                                                               & {(\langle 2\rangle,\langle 1\rangle)}                                                    
\end{tikzcd}
     \end{center}
     can be composed to yield a 2-simplex in $\mathcal{C}^{\otimes}$ with boundary 
     \begin{center}
         \begin{tikzcd}[column sep = large]
{(A,A)} \arrow[rd, "{\iota_{2,3}}"] \arrow[rrd, "{\textnormal{id}_{(A,A)}}", bend left] \arrow[rdd, "\tau"', bend right] &                                                                               &                          \\
                                                                                                                   & {(A,A,A,A)} \arrow[d, "{(m_1,m_1)\circ \tau_{2,3}}"] \arrow[r, "{(m_2,m_2)}"] & {(A,A)} \arrow[d, "m_1"] \\
                                                                                                                   & {(A,A)} \arrow[r, "m_2"']                                                     & A                       
\end{tikzcd}
     \end{center}
     whose homotopy class is identified under the equivalence $\tilde{A}\circ \rho \simeq A$ with the image under $\tilde{A}$ of the half twist $\sigma_0$ between $\mu_0$ and $\mu_\pi$ in $\hom_{\mathbb{E}_2^{\otimes}}(\langle 2\rangle,\langle 1\rangle)_1$.
\end{theorem}
\begin{proof}
    First check that the 2-simplices in $\mathbb{E}_1^{\otimes}\times \mathbb{E}_1^{\otimes}$ indeed induce composable 2-simples in $\mathcal{C}^{\otimes}$ with the depicted boundaries. Recall that a bifunctor of $\infty$-operads sends pairs of inert maps to inert maps. Therefore, if $\rho^i: \langle m\rangle \rightarrow \langle 1 \rangle$ is the inert map $i\mapsto 1$, $i\neq j \mapsto \ast$, the map $A(\rho^i,\text{id}): A(\langle m\rangle, \langle n\rangle) \rightarrow A(\langle 1\rangle, \langle n\rangle)$ must be an inert lift in $\mathcal{C}^{\otimes}$ of $\rho^i: \langle mn\rangle \rightarrow \langle 1\rangle$. A similar argument holds for $\rho^j$ in the second component. This shows that under the identification $\mathcal{C}^{\otimes}_{\langle n\rangle} \simeq \mathcal{C}^{\times n}$, the object $A(\langle m\rangle, \langle n\rangle)$ is equivalent to the tuple $(A,\dots, A)$ with $m$ entries. Applying $A$ to the diagrams in $\mathbb{E}_1^{\otimes} \times \mathbb{E}_1^{\otimes}$, we hence get diagrams
    \begin{center}
       \begin{tikzcd}
                                                                                  & {(A,A,A,A)} \arrow[rr] \arrow[dd] &                                                                        & {(A,A)} \arrow[dd] \\
{A(\langle 2\rangle,\langle 2\rangle)} \arrow[dd] \arrow[rr] \arrow[ru, "\simeq"] &                                   & {A(\langle 2\rangle,\langle 1\rangle)} \arrow[dd] \arrow[ru, "\simeq"] &                    \\
                                                                                  & {(A,A)} \arrow[rr]                &                                                                        & A                  \\
{A(\langle 1\rangle,\langle 2\rangle)} \arrow[rr] \arrow[ru, "\simeq"]            &                                   & {A(\langle 1\rangle,\langle 1\rangle)} \arrow[ru, "\simeq"]            &                   
\end{tikzcd}
    \end{center}

    \begin{center}
        \begin{tikzcd}
                                                                                 & {(A,A)} \arrow[d] \arrow[rd]                                &         \\
{A(\langle 1\rangle,\langle 2\rangle)} \arrow[d] \arrow[rd] \arrow[ru, "\simeq"] & {(A,A,A,A)} \arrow[r]                                       & {(A,A)} \\
{A(\langle 1\rangle,\langle 4\rangle)} \arrow[r] \arrow[ru, "\simeq" near end]            & {A(\langle 1\rangle,\langle 2\rangle)} \arrow[ru, "\simeq"] &        
\end{tikzcd}
    \end{center}

    \begin{center}
       \begin{tikzcd}
                                                                                 & {(A,A)} \arrow[r] \arrow[rd]                                          & {(A,A,A,A)} \arrow[d] \\
{A(\langle 2\rangle,\langle 1\rangle)} \arrow[r] \arrow[rd] \arrow[ru, "\simeq"] & {A(\langle 4\rangle,\langle 1\rangle)} \arrow[d] \arrow[ru, "\simeq" near start] & {(A,A)}               \\
                                                                                 & {A(\langle 2\rangle,\langle 1\rangle)} \arrow[ru, "\simeq"]           &                      
\end{tikzcd}
    \end{center}
    It hence suffices to show that the respective back sides of the diagrams fit together into the depicted 2-simplex. To this end we check that the maps $(A,A,A,A) \rightarrow (A,A)$ in the square agree with the respective maps $(A,A,A,A)\rightarrow (A,A)$ in the triangles, and similarly for the two maps $(A,A) \rightarrow (A,A,A,A)$ in the different triangles. Note that by point (2) of the definition \cite[Definition 2.1.1.10]{HA} of $\infty$-operads, it suffices to show that each of those pairs of maps agrees after post-composition with coCartesian lifts of the $\rho^i$. Consider first the maps induced by $(\langle 2\rangle,\langle 2\rangle) \xrightarrow{\textnormal{id}_{\langle 2\rangle},\mu} (\langle 2\rangle,\langle 1\rangle)$ and $(\langle 1\rangle,\langle 4\rangle) \xrightarrow{\textnormal{id}_{\langle 1\rangle},(\mu,\mu)} (\langle 1\rangle,\langle 2\rangle)$. We have a factorization 
    \begin{center}
        \begin{tikzcd}
{(\langle 2\rangle,\langle 2\rangle)} \arrow[r, "{\textnormal{id},\mu}"] \arrow[rd, "{\rho^i,\textnormal{id}}"'] & {(\langle 2\rangle,\langle 1\rangle)} \arrow[r, "{\rho^i,\textnormal{id}}"] & {(\langle 1\rangle,\langle 1\rangle)} \\
                                                                                                                 & {(\langle 1\rangle,\langle 2\rangle)} \arrow[ru, "{\textnormal{id},\mu}"']  &                                      
\end{tikzcd}
    \end{center}
    in $\mathbb{E}_1^{\otimes}\times \mathbb{E}_1^{\otimes}$, where the unique bifunctor $\mathbb{E}_1^{\otimes} \times \mathbb{E}_1^{\otimes} \rightarrow \fin$ sends the lower left hand side map to the inert map $f_1= 1,2,\ast,\ast: \langle 4\rangle \rightarrow \langle 2\rangle$ if $i=1$ and to the inert map $f_2 =\ast, \ast, 1,2$ if $i = 2$. Since inert pairs are sent to inert maps by $A$, this diagram maps to 
    \begin{center}
        \begin{tikzcd}
{(A,A,A,A)} \arrow[r, "{A(\textnormal{id},\mu)}"] \arrow[rd] & {(A,A)} \arrow[r, "\rho^i"] & A \\
                                                            & {(A,A)} \arrow[ru, "m_2"']  &  
\end{tikzcd}.
    \end{center}
    where $(A,A,A,A)\rightarrow (A,A)$ is an inert lift of $f_1$ or $f_2$ respectively. On the other hand, we also have a factorization 
    \begin{center}
        \begin{tikzcd}
{(\langle 1\rangle,\langle 4\rangle)} \arrow[r, "{\textnormal{id},(\mu,\mu)}"] \arrow[rd, "{\textnormal{id},f_i}"'] & {(\langle 1\rangle,\langle 2\rangle)} \arrow[r, "{\textnormal{id},\rho^i}"] & {(\langle 1\rangle,\langle 1\rangle)} \\
                                                                                                                    & {(\langle 1\rangle,\langle 2\rangle)} \arrow[ru, "{\textnormal{id},\mu}"']  &                                      
\end{tikzcd}
    \end{center}
    This again is sent by $A$ to 
    \begin{center}
        \begin{tikzcd}[column sep = large]
{(A,A,A,A)} \arrow[rd] \arrow[r, "{A(\textnormal{id},(\mu,\mu))}"] & {(A,A)} \arrow[r, "\rho^i"] & A \\
                                                                   & {(A,A)} \arrow[ru, "m_2"']  &  
\end{tikzcd},
    \end{center}
    showing that our two maps $(A,A,A,A)\rightarrow (A,A)$ indeed agree up to homotopy. An analogous analysis can be carried out with the other two maps $(A,A,A,A)\rightarrow (A,A)$.\par 
    For the two inclusions $(\langle 1\rangle,\langle 2\rangle) \rightarrow (\langle 1\rangle,\langle 4\rangle)$ and $(\langle 2\rangle,\langle 1\rangle)\rightarrow (\langle 4\rangle,\langle 1\rangle)$, it suffices to show that the two unit maps coming from $(\langle 1\rangle,\langle 0\rangle) \rightarrow (\langle 1\rangle,\langle 1\rangle)$ and $(\langle 0\rangle,\langle 1\rangle)\rightarrow (\langle 1\rangle,\langle 1\rangle)$ agree. To see this, note that the composition
    \begin{align*}
        (\langle 1\rangle, \langle 2\rangle) \rightarrow (\langle 1\rangle, \langle 4\rangle) \xrightarrow{\rho^i} (\langle 1\rangle, \langle 1\rangle)
    \end{align*}
    is either inert (for $i=2,3$), or it factors as 
    \begin{align*}
        (\langle 1\rangle, \langle 2\rangle) \xrightarrow{\text{inert}} (\langle 1\rangle, \langle 0\rangle) \rightarrow (\langle 1\rangle, \langle 1\rangle)
    \end{align*}
    for $i=1,4$. The same holds for the other inclusion with the entries switched. To show that the two unit maps agree, note that both $(\langle 0\rangle,\langle 0\rangle) \rightarrow (\langle 1\rangle,\langle 0\rangle)$ and $(\langle 0\rangle,\langle 0\rangle) \rightarrow (\langle 0\rangle,\langle 1\rangle)$ are sent to the identity on the empty tuple by $A$, up to homotopy, since $\mathcal{C}^{\otimes}_{\langle 0\rangle}$ is contractible. Then the image of the diagram
    \begin{center}
        \begin{tikzcd}
                                                                       & {(\langle 1\rangle,\langle 0\rangle)} \arrow[rd] &                                       \\
{(\langle 0\rangle,\langle 0\rangle)} \arrow[rr] \arrow[ru] \arrow[rd] &                                                  & {(\langle 1\rangle,\langle 1\rangle)} \\
                                                                       & {(\langle 0\rangle,\langle 1\rangle)} \arrow[ru] &                                      
\end{tikzcd}
    \end{center}
    under $A$ then shows that the two unit maps are homotopic.\par
    We now show that the depicted 2-simplex in $\mathcal{C}^{\otimes}$ indeed corresponds to the image of the half twist under $\tilde{A}$. To this end, examine the images of the three 2-simplices in $\mathbb{E}_1^{\otimes} \times \mathbb{E}_1^{\otimes}$ under $\rho: \mathbb{E}_1^{\otimes}\times \mathbb{E}_1^{\otimes} \rightarrow \mathbb{E}_2^{\otimes}$. We get 
    \begin{center}
        \includegraphics[scale=.35]{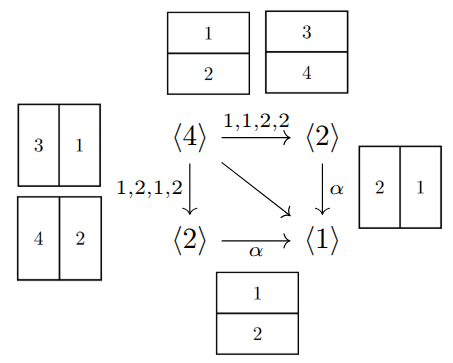}
    \end{center}

    \begin{center}
        \includegraphics[scale=.35]{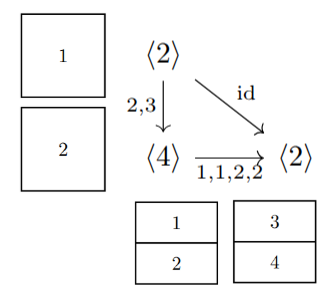}
    \end{center}

    \begin{center}
       \includegraphics[scale=.35]{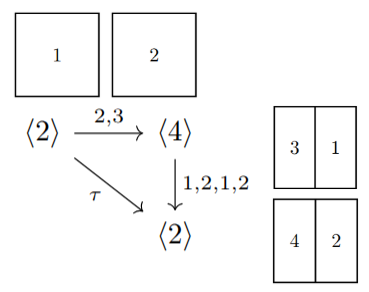}
    \end{center}
    with the paths forming the fillings of the two triangles in the square diagram both the constant path at
    \begin{center}
        \includegraphics[scale=.35]{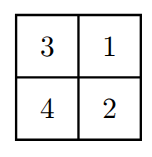},
    \end{center}
    while the fillings of the triangle diagrams are given by continuously enlarging the respective rectangles. Composing those (as paths) in the topological category $\mathbb{E}_2^{T,\otimes}$ as depicted here
    \begin{equation}\label{dig3}
        \begin{tikzcd}
	{\langle 2\rangle} \\
	& {\langle 4\rangle} & {\langle 2\rangle} \\
	& {\langle 2\rangle} & {\langle 1\rangle}
	\arrow["{2,3}", from=1-1, to=2-2]
	\arrow["{\text{id}}", curve={height=-18pt}, from=1-1, to=2-3]
	\arrow["\tau"', curve={height=18pt}, from=1-1, to=3-2]
	\arrow["{1,1,2,2}", from=2-2, to=2-3]
	\arrow["{1,2,1,2}"', from=2-2, to=3-2]
	\arrow["\alpha", from=2-3, to=3-3]
	\arrow["\alpha"', from=3-2, to=3-3]
\end{tikzcd}
    \end{equation}
    we hence get the half twist 
    \begin{center}
        \includegraphics[scale=.3]{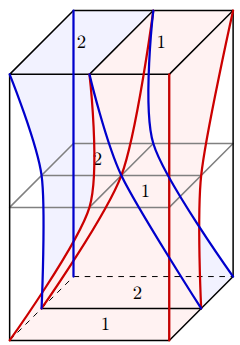}.
    \end{center}
    Applying $\tilde{A}$ to the above simplices (\ref{dig3}) induces a diagram of 2-simplices in $\mathcal{C}^{\otimes}$
    \begin{center}
        \begin{tikzcd}
{(\tilde{A},\tilde{A})} \arrow[rd] \arrow[rrd, bend left] \arrow[rdd, bend right] &                                                                  &                                   \\
                                                                                  & {(\tilde{A},\tilde{A},\tilde{A}, \tilde{A})} \arrow[r] \arrow[d] & {(\tilde{A},\tilde{A})} \arrow[d] \\
                                                                                  & {(\tilde{A},\tilde{A})} \arrow[r]                                & \tilde{A}                        
\end{tikzcd}
    \end{center}
    and by construction the isomorphism between $\tilde{A}\circ \rho$ and $A$ identifies this diagram with the one in the statement. Since maps between $\infty$-categories respect composition, this proves the claim.
\end{proof}
We can repeat this analysis for the other "parts" of the classical Eckmann-Hilton argument; using the inclusion $(1,4):\langle 2\rangle \rightarrow \langle 4\rangle$ and the opposite multiplication $\mu_1\in \map_{\mathbb{E}_1^{\otimes}}(\langle 2\rangle,\langle 1\rangle)_0$. In particular, we get representations of the image under $\tilde{A}$ of all four different parts of the double twist in terms of compositions of 2-simplices in the image of $A$. As a corollary, we obtain
\begin{corollary}\label{cor1}
    The homotopy class of the bracket on $\tilde{A}$ at $t=0$ can be produced as a composition of the following 2-simplices in $\mathcal{C}^{\otimes}$
    \begin{center}
         \begin{tikzcd}
{(A,A)} \arrow[rd, "{\iota_{2,3}}"] \arrow[rrd, "{\textnormal{id}}", bend left] \arrow[rdd, "\tau"', bend right] &                                                                               &                          \\
                                                                                                                   & {(A,A,A,A)} \arrow[d, "{(m_1,m_1)\circ \tau_{2,3}}"] \arrow[r, "{(m_2,m_2)}"] & {(A,A)} \arrow[d, "m_1"] \\
                                                                                                                   & {(A,A)} \arrow[r, "m_2"']                                                     & A                       
\end{tikzcd},
    \end{center}

    \begin{center}
        \begin{tikzcd}
{(A,A)} \arrow[rd, "{\iota_{1,4}}"] \arrow[rrd, "\textnormal{id}", bend left] \arrow[rdd, "\textnormal{id}"', bend right] &                                                                                                                                                        &                                            \\
                                                                                                              & {(A,A,A,A)} \arrow[d, "{(m_1^{\textnormal{op}},m_1^{\textnormal{op}})\circ \tau_{2,3}}"] \arrow[r, "{(m_2^{\textnormal{op}},m_2^{\textnormal{op}})}"] & {(A,A)} \arrow[d, "m_1^{\textnormal{op}}"] \\
                                                                                                              & {(A,A)} \arrow[r, "m_2^{\textnormal{op}}"']                                                                                                            & A                                         
\end{tikzcd},
    \end{center}

    \begin{center}
        \begin{tikzcd}
{(A,A)} \arrow[rd, "{\iota_{2,3}}"] \arrow[rrd, "\textnormal{id}", bend left] \arrow[rdd, "\tau"', bend right] &                                                                                                                                                        &                                            \\
                                                                                                         & {(A,A,A,A)} \arrow[d, "{(m_1^{\textnormal{op}},m_1^{\textnormal{op}}))\circ \tau_{2,3}}"] \arrow[r, "{(m_2^{\textnormal{op}},m_2^{\textnormal{op}})}"] & {(A,A)} \arrow[d, "m_1^{\textnormal{op}}"] \\
                                                                                                         & {(A,A)} \arrow[r, "m_2^{\textnormal{op}}"']                                                                                                            & A                                         
\end{tikzcd},
    \end{center}

    \begin{center}
        \begin{tikzcd}
{(A,A)} \arrow[rd, "{\iota_{1,4}}"] \arrow[rrd, "\textnormal{id}", bend left] \arrow[rdd, "\textnormal{id}"', bend right] &                                                                               &                          \\
                                                                                                                          & {(A,A,A,A)} \arrow[d, "{(m_1,m_1)\circ \tau_{2,3}}"] \arrow[r, "{(m_2,m_2)}"] & {(A,A)} \arrow[d, "m_1"] \\
                                                                                                                          & {(A,A)} \arrow[r, "m_2"']                                                     & A                       
\end{tikzcd}.
    \end{center}
\end{corollary}
\section{Recovering the Hochschild complex as $\mathbb{E}_1$-center}\label{Recovering_the_Hochschild_complex_as_e1_center}
Often the Hochschild cohomology of a $\field{k}$-algebra is called its "derived center". In this section we will show that this statement is true in a very precise sense. Namely, the Hochschild complex is the $\mathbb{E}_1$-center in the derived $\infty$-category of chain complexes. \par
In contrast to \cite[Definition 5.1]{BZFN}, we use the definition of the derived center via its universal property as in \ref{def5}, and then show, using Lurie's theory on higher centers, that we can reduce to considering the derived endomorphism complex of $A$ as an $A$-bimodule. As a consequence, we get a direct proof that the $\mathbb{E}_2$-algebra structure is obtained via the higher Eckmann-Hilton argument \ref{cor1} from the Yoneda and convolution products, and we can use this to show that this definition of the Hochschild cochain complex yields the correct Gerstenhaber algebra in cohomology.
\subsection{Symmetric monoidal dg model categories}\label{symmetric_monoidal_dg_model_categories}
We first examine how a symmetric monoidal model category with a compatible structure of a dg category yields a symmetric monoidal $\infty$-category. If $\mathbf{S}$ is a monoidal model category, one can define the notion of an $\mathbf{S}$-enriched model category as in \cite[Definition A.3.1.5]{HTT}. In this section, we consider $\mathbf{S}$ to be the (symmetric) monoidal model category $\ch(\field{k})$ of chain complexes of $\field{k}$-modules with the projective model structure as in \cite[Proposition 4.2.13]{H}. We call $\ch(\field{k})$-enriched model categories \textbf{dg model categories}. Analogously to simplicial model categories, they are tensored and cotensored over $\ch(\field{k})$, and they satisfy an (SM7) property for the cotensoring.\\ \par
If $C$ is a dg category, denote by $C_0$ the underlying 1-category. If $C$ is a dg model category, we view $C_0$ as an ordinary model category.
\begin{theorem}
Let $C$ be a dg category that is tensored over the category $\ch(\field{k})$ of chain complexes of $\field{k}$-modules, let $C'$ be a full dg subcategory, and let $W$ be a collection of morphisms in $C'$ that are isomorphisms in the homotopy category of the dg category $C$. Assume that the following conditions are satisfied:
\begin{itemize}
    \item Every isomorphism in $C'$ belongs to $W$.
    \item The set $W$ satisfies the 2-out-of-3 property.
    \item For all $X\in C'$, we also have $N_{\ast}(\Delta^1)\otimes X \in C'$.
    \item For each $X\in C'$, the map $N_{\ast}(\Delta^1)\otimes X \rightarrow X$ induced by the map $[1]\rightarrow [0]$ belongs to $W$.
\end{itemize}
Then the canonical map $\theta: N(C'_0) \rightarrow N_{\text{dg}}(C')$ constructed in \cite[Remark 1.3.1.9]{HA} induces an equivalence of $\infty$-categories $\theta': N(C'_0)[W^{-1}]$ $\simeq N_{\text{dg}}(C')$.
\end{theorem}
\begin{proof}
The above conditions are exactly what is needed to repeat the proof of \cite[Proposition 1.3.4.5]{HA} replacing $\ch(\mathcal{A})$ with ${C}$ and $\ch(\mathcal{A})'$ with ${C}'$.
\end{proof}
\begin{corollary}\label{cor2}
Let $C$ be a dg model category, and let $C^{\circ}$ be the full dg subcategory on bifibrant objects. Then the map $N(C^{c}_0) \rightarrow N_{\text{dg}}(C^{\circ})$ exhibits the dg nerve as the underlying $\infty$-category of $C_0$.
\end{corollary}
\begin{proof}
Let $C_{\Delta}$ be the simplicial category obtained from ${C}$ via the Dold-Kan correspondence. Since $N({C}_0^{c})[W^{-1}] \simeq N({C}_0^{\circ})[W^{-1}]$ it suffices to take ${C}' = {C}^{\circ}$ above and $W$ the set of homotopy equivalences. In a dg model category, left homotopy between bifibrant objects agrees with chain homotopy in the dg sense, i.e. a homotopy is a map $h: N(\Delta^1) \otimes X \rightarrow Y$ with the correct restrictions to $\{0\}$ and $\{1\}$. In particular,
\begin{align*}
    \hom_{{C}_0}(N_{\ast}(\Delta^1)\otimes X,Y)  &\cong \hom_{\ch(\field{k})}(N_{\ast}(\Delta^1),\map_{{C}}(X,Y)) \\&\cong \hom_{\ch_{\geq 0}(\field{k})}(N_{\ast}(\Delta^1),\tau_{\geq 0} \map_{{C}}(X,Y) \\&\cong \hom_{\text{sSet}}(\Delta^1, \text{DK}_{\bullet}\tau_{\geq 0}\map_{{C}}(X,Y)) \\&\cong \map_{{C}_{\Delta}}(X,Y)_1
\end{align*}
so left homotopies correspond to 1-chains of the mapping complex of ${C}$. One checks directly that the diagram making $h: N_{\ast}(\Delta^1)\otimes X \rightarrow Y$ into a homotopy between $f,g\in \hom_{{C}_0}(X,Y)$ forces the corresponding 1-chain $z\in \map_{{C}}(X,Y)_1$ to satisfy $dz = f-g$. This shows that homotopy equivalences in ${C}^{\circ}_0$ become isomorphisms in the homotopy category $h{C}_{\Delta}$, which is isomorphic to the dg homotopy category $h{C}$. Clearly every isomorphism is a homotopy equivalence and the set of homotopy equivalences satisfies 2-out-of-3. By assumption, the map $\otimes: \ch(\field{k}) \times {C} \rightarrow {C}$ is a left Quillen bifunctor. The complex $N_{\ast}(\Delta^1)$ is cofibrant, and hence $N_{\ast}(\Delta^1)\otimes -$ preserves cofibrant objects. It also preserves fibrant objects, since $N_{\ast}(\Delta^1)$ is dualizable and so $N_{\ast}(\Delta^1) \otimes - \cong \map_{C}(N_{\ast}(\Delta^1)^{\vee},-)$. This shows that $N_{\ast}(\Delta^1)\otimes X\in {C}^{\circ}$. Finally, note that $d_0: \field{k}\rightarrow N_{\ast}(\Delta^1)$ is a trivial cofibration in $\ch(\field{k})$, and thus if $X\in {C}^{\circ}$, the map $X \cong \field{k}\otimes X \rightarrow N_{\ast}(\Delta^1) \otimes X$ is again a trivial cofibration. Now the map $N_{\ast}(\Delta^1)\rightarrow \field{k}$ is a left inverse to $d_0$ and in particular 
\begin{align*}
    X \rightarrow N_{\ast}(\Delta^1) \otimes X \rightarrow X
\end{align*}
is the identity on $X$ and thus a weak equivalence. By 2-out-of-3, this means that $N_{\ast}(\Delta^1)\otimes X \rightarrow X$ must be a weak equivalence.
\end{proof}
\begin{remark}
Note that for any dg category ${C}$, we have an equivalence of $\infty$-categories $N_{\text{hc}}({C}_{\Delta}) \rightarrow N_{\text{dg}}({C})$. For simplicial model categories, the homotopy coherent nerve of the bifibrant objects is always equivalent to the $\infty$-category underlying the model category. In contrast, $C_{\Delta}$ is only \textbf{weakly simplicially enriched}, i.e. it is neither tensored nor cotensored over $\mathbf{sSet}$, and therefore does not satisfy the requirements of this theorem. The above corollary then shows that we get this relationship between the homotopy coherent nerve and the model category regardless.
\end{remark}
If ${C}$ is a (symmetric) monoidal model category, then \cite[Example 4.1.7.6]{HA} shows that $N({C}^c)[W^{-1}]$ is a (symmetric) monoidal $\infty$-category. If ${C}$ is also a simplicial model category and the (symmetric) monoidal structure is compatible with the simplicial enrichment, then \cite[Corollary 4.1.7.16]{HA} shows that the (symmetric) monoidal structure on this $\infty$-category is given by $N_{\text{hc}}(({C}^{\circ})^{\otimes})$, and in fact one readily checks that the same hold if ${C}$ is just weakly simplicially enriched. Since the Dold-Kan functor $\mathrm{DK}_{\bullet}$ is not symmetric, Lurie's result does not imply that the dg nerve of a (symmetric) monoidal dg model category $C$ presents the (symmetric) monoidal structure of $N({C}^c)[W^{-1}]$. Nevertheless, $\mathrm{DK}_{\bullet}$ is homotopy symmetric lax monoidal, and in fact Hinich proved the following.
\begin{proposition}[\cite{Hin}, Theorem 3.2.3]\label{prop8}
The dg nerve  
\begin{align*}
    N_{\textnormal{dg}}: N(\textnormal{Cat}_{\textnormal{dg}})[W_{\textnormal{dg}}^{-1}] \rightarrow N(\textnormal{Cat}_{\Delta})[W_{\Delta}^{-1}]\simeq \textnormal{Cat}_{\infty}
\end{align*}
from the symmetric monoidal $\infty$-category of dg categories to the symmetric monoidal $\infty$-category of $\infty$-categories is lax symmetric monoidal. In particular, it is a morphism of $\infty$-operads. It thus induces a map from the $\infty$-category of symmetric monoidal dg categories to the $\infty$-category of symmetric monoidal $\infty$-categories.
\end{proposition}
\subsection{Rectification of algebras over an $\infty$-operad}\label{Rectification_of_algebras_over_an_infty_operad}
Let $\mathcal{O}$ be a topological operad. Applying the singular chain functor with $\field{k}$-coefficients produces a dg operad $C_{\ast}(\mathcal{O})$. In this section, we will generalize results of Hinich \cite{Hin} and D. Pavlov and J. Scholbach \cite{PS}, and show that $C_{\ast}(\mathcal{O})$-algebras in a symmetric monoidal dg model category $C$ correspond to algebras over the $\infty$-operad $N^{\otimes}(\text{Sing}_{\bullet}(\mathcal{O}))$ in the symmetric monoidal $\infty$-category $N_{\text{dg}}(C^{\circ})$.\\

Let ${C}$ be a symmetric monoidal dg model category. Suppose further that ${C}$ is cofibrantly generated and symmetrically flat as defined in \cite[Definition 2.1. (vii)]{PS}, and that $C_{\ast}(\mathcal{O})$ is admissible \cite[Definition 5.1.]{PS} and well-pointed \cite[Definition 6.1]{PS} in ${C}$, and that $C$ admits a lax symmetric monoidal fibrant replacement functor. Let $\alg_{C_{\ast}(\mathcal{O})}(C)$ denote the model category of strict $C_{\ast}(\mathcal{O})$-algebras in $C$ with the transfer model structure from the forgetful functor to $C$. Note that this model structure exists by the assumption that $C_{\ast}(\mathcal{O})$ is admissible in $C$. The construction in \cite[Section 4.2]{Hin} generalizes directly to give a functor
\begin{align*}
    \phi: N(\alg_{C_{\ast}({\mathcal{O}})}({C})^{c}) \rightarrow \alg_{N^{\otimes}(\text{Sing}_{\bullet}(\mathcal{O}))}(N_{\text{dg}}({C}^{\circ})).
\end{align*}
Roughly, a cofibrant $C_{\ast}(\mathcal{O})$-algebra in $C$ is given by a symmetric monoidal functor of dg categories from the PROP $P_{C_{\ast}(\mathcal{O})}$ generated by $C_{\ast}(\mathcal{O})$ to $C^c$, and applying the dg nerve as in Proposition \ref{prop8} yields a symmetric monoidal functor $P_{\mathcal{O}} \rightarrow \text{N}_{\text{dg}}(C^{\circ})$. Here, $P_{\mathcal{O}}$ is the symmetric monoidal envelope of the $\infty$-operad $N^{\otimes}(\text{Sing}_{\bullet}(\mathcal{O}))$ as defined in \cite[Construction 2.2.4.1.]{HA}, and hence this construction yields an $N^{\otimes}(\text{Sing}_{\bullet}(\mathcal{O}))$-algebra in $\text{N}_{\text{dg}}(C^{\circ})$.\\

This functor $\phi$ carries weak equivalences to equivalences as shown in \cite[Diagram 49]{Hin}, and therefore yields a comparison map
\begin{align*}
    \Phi:  N(\alg_{C_{\ast}(\mathcal{O})}({C})^{c})[W_{\alg_{C_{\ast}(\mathcal{O})}({C})}^{-1}] \rightarrow \alg_{N^{\otimes}(\text{Sing}_{\bullet}(\mathcal{O}))}(N_{\text{dg}}({C}^{\circ})).
\end{align*}
\begin{theorem}\label{thm3}
    Let $\mathcal{O}$ and ${C}$ as above. Then $\Phi$ is an equivalence of $\infty$-categories.
\end{theorem}
\begin{proof}
    We use Corollary 4.7.3.16 in \cite{HA}, which is an application of the $\infty$-categorical Barr-Beck Theorem \cite[Theorem 4.7.3.5]{HA}. Consider the diagram
    \begin{center}
        \begin{tikzcd}
{ N(\alg_{C_{\ast}(\mathcal{O})}({C})^{c})[W_{\alg_{C_{\ast}(\mathcal{O})}({C})}^{-1}]} \arrow[rr, "\Phi"] \arrow[rd, "G"'] &                                              & { \alg_{N^{\otimes}(\text{Sing}_{\bullet}(\mathcal{O}))}(N({C}^{c})[W^{-1}])} \arrow[ld, "G'"] \\
                                                                                                                                                & {(N({C}^c)[W^{-1}])^{[\mathcal{O}]}} &                                                                                
\end{tikzcd}
    \end{center}
    where $(N(C^c)[W^{-1}])^{[\mathcal{O}]}$ denotes the coproduct of $N(C^c)[W^{-1}]$ over the collection of colors $[\mathcal{O}]$ in $\mathcal{O}$, and $G$ and $G'$ are the forgetful functors given by evaluation at the colors of $\mathcal{O}$. To show that $\Phi$ is an equivalence of $\infty$-categories, it then suffices to show that $G$ and $G'$ satisfy the assumptions (1)-(5) of \cite[Corollary 4.7.3.16]{HA}, and that in addition $G$ is conservative. To verify these assumptions, we follow the proof of \cite[Theorem 4.5.4.7]{HA}; in particular we will refer to "steps (a)-(e)" from this proof.\par
    Steps (a)-(c) can be proven exactly like in \cite[Theorem 4.5.4.7]{HA}, by just replacing the commutative operad by $\mathcal{O}$. For step (d), it is clear that $G$ is conservative since the weak equivalences in $\alg_{C_{\ast}(\mathcal{O})}({C})$ are transferred from the ones in ${C}$ via the forgetful functor. To show that $G$ preserves geometric realization of simplicial objects, it suffices to show that it preserves homotopy sifted colimits. This is shown in \cite[Proposition 7.9]{PS}. Finally, for step (e), we need to show that the canonical transformation $G'\circ F' \rightarrow G\circ F$ is an equivalence, where $F$ and $F'$ are the left adjoints of $G$ and $G'$ respectively. This boils down to showing that for any cofibrant object $X\in {C}^{[\mathcal{O}]}$, the strict free $C_{\ast}(\mathcal{O})$-algebra generated by $X$ is also a free $N^{\otimes}(\text{Sing}_{\bullet}(\mathcal{O}))$-algebra in the sense of \cite[Definition 3.1.3.1]{HA}. To this end, Hinich \cite[Lemma 4.3.4]{Hin} proved an analogue of \cite[Proposition 3.1.3.13]{HA} for the setting of free algebras generated by objects of different colors. We have a map of $\infty$-operads\footnote{Here, $\mathscr{T}riv^{\otimes}$ is the trivial $\infty$-operad as defined in \cite[Example 2.1.3.5.]{HA}} $(\mathscr{T}riv^{\otimes})^{[\mathcal{O}]} \rightarrow {N}^{\otimes}(\text{Sing}_{\bullet}(\mathcal{O}))$, and a collection $X:=\{X_i\}_{i\in [\mathcal{O}]}$ of objects in ${N}_{\text{dg}}(C^{\circ})$ induces a $(\mathscr{T}riv^{\otimes})^{[\mathcal{O}]}$-algebra $\overline{X}$ in $N_{\text{dg}}(C^{\circ})^{\otimes} \times_{\fin} N^{\otimes}(\text{Sing}_{\bullet}(\mathcal{O}))$ over $N^{\otimes}(\text{Sing}_{\bullet}(\mathcal{O}))$. Let $\mathbb{F}_{C_{\ast}(\mathcal{O})}(X)$ be the strict free $C_{\ast}(\mathcal{O})$ algebra generated by $X$, and let $\mathbb{F}= \phi(\mathbb{F}_{C_{\ast}(\mathcal{O})}(X))$. For a color $c\in [\mathcal{O}]$, Hinich \cite[Lemma 4.3.4]{Hin} constructs a map $\text{Sym}_{\mathcal{O}}(\overline{X})_c \rightarrow \mathbb{F}_c$ which is an equivalence for all colors if and only if $\mathbb{F}$ is indeed the free $N^{\otimes}(\text{Sing}_{\bullet}(\mathcal{O}))$-algebra. Hinich checks this for the case of $C = \ch(\field{k})$, but one readily sees that all his arguments still work for any symmetric monoidal dg model category $C$.
\end{proof}
\begin{remark}
We call a homotopy preimage of an $N^{\otimes}(\text{Sing}_{\bullet}(\mathcal{O}))$-algebra $A$ under $\Phi$ a \textbf{strictification} of $A$.
\end{remark}
\begin{corollary}
The dg nerve induces a map $\alg_{\lm}(N(\textnormal{Cat}_{\textnormal{dg}})[W_{\textnormal{dg}}^{-1}]) \rightarrow \alg_{\lm}(\textnormal{Cat}_{\infty})$. By \ref{thm3}, this implies that a dg model category left tensored over a monoidal dg model category yields an $\infty$-category left tensored over a monoidal $\infty$-category.
\end{corollary}
We can now show that if $M$ is a dg category left tensored over a monoidal dg model category $C$, a dg morphism object yields an $\infty$-categorical morphism object in the sense of Definition \ref{def3}.
\begin{lemma}\label{lem1}
Let ${C}$ be a monoidal dg model category with underlying monoidal product $\otimes: {C}_0\times {C}_0 \rightarrow {C}_0$. Then the induced monoidal product $N_{\text{dg}}({C}^{\circ}) \times  N_{\text{dg}}({C}^{\circ})\rightarrow N_{\text{dg}}({C}^{\circ})$ sends $A,B\in {C}^{\circ}$ to an object equivalent to $R(A\otimes B)$. A similar statement holds for dg model categories left tensored over a monoidal dg model category.
\end{lemma}
\begin{proof}
    This follows directly from the description of the monoidal structure in Proposition \ref{prop8}.
\end{proof}
\begin{theorem}\label{thm4}
Let ${C}$ be a monoidal dg model category and let ${M}$ be a dg model category that is left tensored over ${C}$. In particular, we have a dg functor\footnote{Here $\boxtimes$ denotes the tensor product of dg categories.} $\otimes: {C}\boxtimes {M} \rightarrow {M}$ whose underlying functor is a left Quillen bifunctor. Assume that for $A,B\in {M}^{\circ}$ we have a dg morphism object $\mor_{{M}}(A,B)\in {C}$ together with map $\alpha: \mor_{{M}}(A,B)\otimes A \rightarrow B$ in ${M}$ such that composition with $\alpha$ induces an isomorphism 
\begin{align*}
    \map_{{C}}(C,\mor_{{M}}(A,B)) \cong \map_{{M}}(C\otimes A,B) \in \ch(\field{k}).
\end{align*}
Then
\begin{itemize}
    \item[(1)] The induced map $\tilde{\alpha}\in \map_{N_{\textnormal{dg}}({M}^{\circ})}(R(Q\mor_{{M}}(A,B)\otimes A),B)$  makes $Q\mor_{{M}}(A,B)\in N_{\textnormal{dg}}({C}^{\circ})$ into a morphism object for $A,B\in N_{\textnormal{dg}}({M}^{\circ})$ in the sense of Definition \ref{def3}.
    \item[(2)] If $\beta: R(X\otimes A) \rightarrow B$ is another morphism object for $A$ and $B$ in $N_{\textnormal{dg}}({M}^{\circ})$, then the induced map $f: X\xrightarrow{\simeq} Q\mor_{{M}}(A,B)$ is a weak equivalence in ${C}$, and $\tilde{\alpha}\circ R(f\otimes \textnormal{id}_A) \simeq \beta$ are homotopic.
\end{itemize}
\end{theorem}
\begin{proof}
For (1), note that if $C\in N_{\text{dg}}({C}^{\circ})$ is bifibrant, $Q\mor_{{M}}(A,B) \xtwoheadrightarrow{\simeq} \mor_{{M}}(A,B)$ is the cofibrant replacement map, and $C\otimes A\xhookrightarrow{\simeq} R(C\otimes A)$ is the fibrant replacement map, we get a weak equivalence
\begin{align*}
    \map_{{C}}(C,Q\mor_{{M}}(A,B)) \xrightarrow{\simeq} \map_{{C}}(C,\mor_{{M}}(A,B)) \cong \map_{{M}}(C\otimes A,B) \xrightarrow{\simeq} \map_{{M}}(R(C\otimes A),B)
\end{align*}
of chain complexes. Applying $\textnormal{DK}_{\bullet}\tau_{\geq 0}$, we get
\begin{align*}
    \map_{N_{\text{dg}}({C}^{\circ})}(C,Q\mor_{{M}}(A,B)) \simeq \map_{N_{\text{dg}}({M}^{\circ})}(R(C\otimes A),B).
\end{align*}
Together with Lemma \ref{lem1} this yields the result.\\
For (2), we automatically get $M\simeq \mor_{{M}}(A,B)$ in the $\infty$-category $N_{\textnormal{dg}}({C}^{\circ})$ since morphism objects are unique up to equivalence. Now recall that $N_{\textnormal{dg}}({C}^{\circ}) \simeq N({C}^c)[W^{-1}]$, and since model categories are saturated this implies the result.
\end{proof}
\subsection{Gerstenhaber algebras in the homotopy category}
Recall that in Section \ref{the_bracket_operation_of_an_e2_algebra}, we defined the bracket operations of an $\mathbb{E}_2$-algebra as the image of the 2-morphisms $\gamma_t \in \map_{\mathbb{E}_2^{\otimes}}(\langle 2\rangle,\langle 1\rangle)_1$. We explain how to explicitly describe the bracket operations of a 2-algebra in a symmetric monoidal dg model category, and how the bracket operations of an $\mathbb{E}_2$-algebra in this setting correspond to actual Gerstenhaber brackets. \\ \par
The following result will make it possible for us to compute the Gerstenhaber bracket on the center of an $\mathbb{E}_1$-algebra in the dg nerve of a symmetric monoidal dg model category.
\begin{corollary}\label{cor5}
If $\mathcal{C}^{\otimes}$ is the dg nerve of a symmetric monoidal dg model category $C$ and $A$ is a 2-algebra with associated $\mathbb{E}_2$-algebra $\tilde{A}$, then the  bracket operation at $t=0$ on $\tilde{A}$ is given by the chain homotopy 
\begin{align*}
    \tilde{A}(\gamma_0) \simeq h\iota_{2,3} + h^{\text{op}}\iota_{1,4} +  h^{\text{op}}\iota_{2,3} + h \iota_{1,4}.
\end{align*}
Here $h$ and $h^{\text{op}}$ are the chain homotopies corresponding to the image of the square diagram (\ref{dig2}) in $\mathbb{E}^{\otimes}_1\times \mathbb{E}^{\otimes}_1$ for the multiplications and their opposite multiplications respectively.
\end{corollary}
\begin{proof}
Note first that by construction of the symmetric monoidal structure on ${C}$, a morphism $\underbrace{(A,\dots,A)}_{m \text{ times}} \rightarrow \underbrace{(A,\dots,A)}_{n \text{ times}}$ in $\mathcal{C}^{\otimes}$ corresponds to a map $A^{\otimes m} \rightarrow A^{\otimes n}$ in $C$, and this map is unique up to chain homotopy. Similarly, a 2-morphism between such maps in $C^{\otimes}$ corresponds to a chain homotopy between the corresponding maps in $C$. Fix such maps in $C$ corresponding to all the involved diagrams. Horizontal composition of maps is strictly defined in dg categories, and hence we have well defined whiskering compositions $h\iota_{2,3}$, $h\iota_{1,4}$, $h^{\text{op}}\iota_{2,3}$ and $h^{\text{op}}\iota_{1,4}$. Finally, note that horizontal composition of chain homotopies is given by addition. Then the result follows from Corollary \ref{cor1}.
\end{proof}
We now want to argue that in the case of a symmetric monoidal dg model category, the abstract bracket operations correspond precisely to the Gerstenhaber bracket in the homotopy category. \\ \par
Let $A: \mathbb{E}_2^{\otimes} \rightarrow N_{\text{dg}}(C^{\circ})^{\otimes}$ be an $\mathbb{E}_2$-algebra in the dg nerve of a symmetric monoidal dg model category $C$. Then we get an induced map
\begin{align*}
    \mathbb{E}^T_2(2) \simeq \map_{\mathbb{E}_2^{\otimes}}^{\alpha}(\langle 2\rangle, \langle 1\rangle) \rightarrow \map_{N_{\text{dg}}(C^{\circ})^{\otimes}}(A(\langle 2\rangle),A(\langle 1\rangle)).
\end{align*}
Let $A= A(\langle 1\rangle)\in C^{\circ}$. Then we have a homotopy equivalence
\begin{align*}
    \map_{N_{\text{dg}}(C^{\circ})^{\otimes}}(A(\langle 2\rangle),A(\langle 1\rangle)) \simeq \map_{N_{\text{dg}}(C^{\circ})}(A^{\otimes 2},A).
\end{align*}
Hence we get a map (well-defined up to homotopy)
\begin{align*}
    \mathbb{E}^T_2(2) \rightarrow \map_{N_{\text{dg}}(C^{\circ})}(A^{\otimes 2},A) \simeq \text{DK}\tau_{\geq 0} \map_C(A^{\otimes 2},A),
\end{align*}
and therefore taking homology 
\begin{align*}
    (\ger(2))_n \cong H_n(\mathbb{E}^T_2(2)) \rightarrow \hom_{hC}(A^{\otimes 2},A[n]).
\end{align*}
This procedure yields a Gerstenhaber algebra structure on $A$ in the dg homotopy category of $C$ whose bracket is indeed given by the image of $\gamma_0$.\\ \par
Given an operation $\rho \in \map_{\mathbb{E}_2^{\otimes}}^{\alpha}(\langle n\rangle, \langle 1\rangle)_0$, the algebra $A$ yields a map $A(\rho)\in \map_{N_{\textnormal{dg}}(C^{\circ})^{\otimes}}(A(\langle n\rangle),A(\langle 1\rangle))_0$, which by the construction of the symmetric monoidal structure \ref{prop8} on $N_{\textnormal{dg}}(C^{\circ})^{\otimes}$ corresponds to a morphism $A(\langle 1\rangle)^{\otimes n} \rightarrow A(\langle 1\rangle)$ in $C$. By abuse of notation, we often denote the image of $\rho$ as a map $A(\rho): A^{\otimes n} \rightarrow A$.
\begin{proposition}
Let $A: \mathbb{E}_2^{\otimes} \rightarrow N_{\textnormal{dg}}(C^{\circ})^{\otimes}$ be an $\mathbb{E}_2$-algebra in the dg nerve of a symmetric monoidal dg model category. Let $A^{\textnormal{str}}$ be a homotopy preimage of $A$ under $\Phi$. Without loss of generality, assume that the underlying object of $A^{\textnormal{str}}$ is fibrant. Note that $\Phi(A^{\textnormal{str}})(\langle 1\rangle) \simeq A^{\textnormal{str}}$ agree in $C$ by construction.\\ There is a chain homotopy $h\in \map_C({A^{\textnormal{str}}}^{\otimes 2}, A)_1$
\begin{center}
\begin{tikzcd}
	{{A^{\textnormal{str}}}^{\otimes 2}} && {A^{\textnormal{str}}} \\
	{A^{\otimes 2}} && A
	\arrow[""{name=0, anchor=center, inner sep=0}, "{A^{\textnormal{str}}(\mu_0)}", from=1-1, to=1-3]
	\arrow["\simeq"', from=1-1, to=2-1]
	\arrow["\simeq", from=1-3, to=2-3]
	\arrow[""{name=1, anchor=center, inner sep=0}, "{A(\mu_0)}"', from=2-1, to=2-3]
	\arrow[shorten <=4pt, shorten >=4pt, Rightarrow, from=0, to=1, "h"]
\end{tikzcd}
\end{center}
and a degree 2 map $K\in \map_C({A^{\textnormal{str}}}^{\otimes 2}, A)_2$ filling the cylinder
\begin{center}
\begin{tikzcd}[column sep = large, row sep = large]
	{{A^{\textnormal{str}}}^{\otimes 2}} && {A^{\textnormal{str}}} \\
	\\
	{A^{\otimes 2}} && A
	\arrow[""{name=0, anchor=center, inner sep=0}, "{{A^{\textnormal{str}}(\mu_0)}}", curve={height=-18pt}, from=1-1, to=1-3]
	\arrow[""{name=1, anchor=center, inner sep=0}, "{{A^{\textnormal{str}}(\mu_0)}}"', curve={height=18pt}, from=1-1, to=1-3]
	\arrow["\simeq"', from=1-1, to=3-1]
	\arrow["\simeq", from=1-3, to=3-3]
	\arrow[""{name=2, anchor=center, inner sep=0}, "{{A(\mu_0)}}", curve={height=-18pt}, from=3-1, to=3-3]
	\arrow[""{name=3, anchor=center, inner sep=0}, "{{A(\mu_0)}}"', curve={height=18pt}, from=3-1, to=3-3]
	\arrow["{A^{\textnormal{str}}(\gamma_0)}", shorten <=5pt, shorten >=5pt, Rightarrow, from=0, to=1]
	\arrow["{A(\gamma_0)}", shorten <=5pt, shorten >=5pt, Rightarrow, from=2, to=3]
\end{tikzcd}.
\end{center}
These (higher) homotopies identify the product and bracket of $A$ and $A^{\textnormal{str}}$. 
\end{proposition}
\begin{proof}
There is a natural equivalence $\eta: \Delta^1 \times \mathbb{E}_2^{\otimes} \rightarrow N_{\text{dg}}(C^{\circ})^{\otimes}$ between $\Phi(A^{\text{str}})$ and $A$. Evaluating at the 1-simplex $(e, s_0\langle n\rangle)\in \Delta^1 \times \mathbb{E}_2^{\otimes}$, where $e$ denotes the non-degenerate 1-simplex of $\Delta^1$, yields equivalences $\Phi(A^{\text{str}})^{\otimes n} \xrightarrow{\simeq} A^{\otimes n}$ in $C$. Evaluating at the 1-simplex $(e,\mu_0)$ yields a map $\Phi(A^{\textnormal{str}})^{\otimes 2} \rightarrow A$, and the 2-simplex $(s_0e, s_1\mu_0)$ yields a chain homotopy $h_1$ between $\eta(e,\mu_0)$ and $\eta(e,s_0\langle 1 \rangle)\circ \Phi(A^{\textnormal{str}})(\mu_0)$. Similarly, the 2-simplex $(s_1e, s_0\langle 2\rangle)$ yields a chain homotopy $h_2$ between $\eta(e,\mu_0)$ and $A(\mu_0) \circ \eta(e,s_0 \langle 2\rangle)$. Composing these, we get the chain homotopy $h$ above. A similar analysis with 3-simplices yields the degree 2 map $K$.
\end{proof}
\begin{corollary}\label{cor10}
The induced Gerstenhaber algebra of $A^{\textnormal{str}}$ in the homotopy category agrees with the one constructed directly from $A$.
\end{corollary}
\subsection{The category of $\mathbb{E}_1$-modules in chain complexes}\label{the_category_of_e1_modules_in_chain_complexes}
We want to use the description of the center in Corollary \ref{cor7} to compute the center of an associative $\field{k}$-algebra. Recall that this description involves the morphism object in the category of $\mathbb{E}_1$-modules over the algebra. In this section, we examine this module category.\\ \par
Let $\mathcal{C}^{\otimes} = N_{\text{dg}}(\ch(\field{k}))^{\otimes}$ be the symmetric monoidal $\infty$-category corresponding to the symmetric monoidal dg model category $\ch(\field{k})$. Let $A$ be an $\eone$-algebra in $\ch(\field{k})$. By abuse of notation, we denote $\Phi(A)\in \alg_{\mathbb{E}_1}(\mathcal{C})$ again by $A$, where $\Phi$ is the equivalence constructed in Section \ref{Rectification_of_algebras_over_an_infty_operad}. We use the notation from Section \ref{the_center_as_endomorphism_object_of_bimodules}.\\ \par
Recall that $\mathbb{E}_1^{\otimes}$ only has a single color $\mathfrak{a}$, and by Corollary \ref{cor7} the underlying object of the center $\mathfrak{Z}({A})$ at this color is equivalent to the morphism object $\text{Mor}_{\bar{\mathcal{C}}_{\mathfrak{a},\mathfrak{m}}}({A},{A})\in \bar{\mathcal{C}}_{\mathfrak{a},\mathfrak{a}}$. Here ${A}$ is viewed as a module over itself, i.e. as an object of
\begin{align*}
        \overline{\mathcal{C}}_{\mathfrak{a},\mathfrak{m}}=\text{Mod}_{{A}}^{\mathbb{E}_1}(\mathcal{C}\times_{\fin}\mathbb{E}_1)_{\mathfrak{a}}.
\end{align*}
Note that $\overline{\mathcal{C}}_{\mathfrak{a},\mathfrak{a}} = \text{Mod}_{\mathbb{1}}^{\mathbb{E}_1}(\mathcal{C}\times_{N(\text{Fin}_{\ast})}\mathbb{E}_1)_{\mathfrak{a}} \simeq (\mathcal{C}^{\otimes}\times_{\fin}\mathbb{E}_1^{\otimes})_{\mathfrak{a}} \simeq \mathcal{C}$, so 
\begin{align*}
    \mathfrak{Z}({A})(\mathfrak{a})\simeq \mor_{\text{Mod}_{{A}}^{\mathbb{E}_1}(\mathcal{C}\times_{\fin}\mathbb{E}_1)_{\mathfrak{a}}}({A},{A}) \in \mathcal{C}.
\end{align*}
By \cite[Proposition B.1.2]{Hin}, we have an equivalence of $\infty$-categories
\begin{align*}
    \text{Mod}_{{A}}^{\mathbb{E}_1}(N_{\text{dg}}(\ch(\field{k}))\times_{\fin}\mathbb{E}_1)_{\mathfrak{a}} \simeq \alg_{M\mathbb{E}_1}(N_{\text{dg}}(\ch(\field{k}))) \times_{\alg_{\mathbb{E}_1}(N_{\text{dg}}(\ch(\field{k})))} \{{A}\}.
\end{align*}
Here $M\mathbb{E}_1$ is the $\infty$-operad defined as in \cite[5.2]{Hin} that governs pairs of $\mathbb{E}_1$-algebras and bimodules over them. We can now use Hinich's Rectification Theorem for modules \cite[Theorem 5.2.3]{Hin} to get an equivalence of $\infty$-categories
\begin{align*}
    N(\text{Mod}_{{A}}^{\eone}(\ch(\field{k}))^c)[W_{\text{Mod}}^{-1}] \xrightarrow{\simeq} \text{Mod}^{\mathbb{E}_1}_{{A}}(N_{\text{dg}}(\ch(\field{k}))\times_{\fin}\mathbb{E}_1)_{\mathfrak{a}}.
\end{align*}
Note here that since ${A}$ is cofibrant, by \cite[Theorem 2.6]{BM} the module category $\text{Mod}_{{A}}^{\eone}(\ch(\field{k}))$ indeed carries a model category structure transferred via the forgetful functor to $\ch(\field{k})$. 
\begin{notation}
    Let $\mathcal{O}$ be a dg operad and $B$ an $\mathcal{O}$-algebra. We denote by $U_{\mathcal{O}}(B)$ the enveloping algebra of $B$ as defined in \cite[Section 12.3.4]{LV}.
\end{notation}
By \cite[Proposition 2.7]{BM}, the category $\lmod_{U_{\eone}({A})}(\ch(\field{k}))$ can also be made into a model category via transfer from the forgetful functor. By \cite[Theorem 1.10]{BM} we have an isomorphism of categories making the following diagram commute 
\begin{equation}\label{dig1}
    \begin{tikzcd}
	{\text{Mod}_{{A}}^{\eone}(\ch(\field{k}))} && {\lmod_{U_{\eone}({A})}(\ch(\field{k}))} \\
	& {\ch(\field{k})}
	\arrow["\cong", from=1-1, to=1-3]
	\arrow[from=1-1, to=2-2]
	\arrow[from=1-3, to=2-2]
\end{tikzcd}.
\end{equation}
In particular, this isomorphism yields a Quillen equivalence between these two model categories. Consequently, the model category $\lmod_{U_{\eone}(A)}(\ch(\field{k}))$ has an underlying $\infty$-category equivalent to $\text{Mod}_A^{\mathbb{E}_1}(N_{\text{dg}}(\ch(\field{k})\times_{\fin} \mathbb{E}_1)_{\mathfrak{a}}$.
\subsection{The Hochschild complex as a center}\label{the_hochschild_complex_as_a_center}
Fix an associative $\field{k}$-algebra $A$. In particular, $A$ is an $\textbf{Assoc}$-algebra in the category $\ch(\field{k})$. In order to exhibit the Hochschild cochain complex of $A$ as an $\mathbb{E}_1$-center, we use the equivalence $\Phi$ constructed in Section \ref{Rectification_of_algebras_over_an_infty_operad} to view $A$ as an $\mathbb{E}_1$-algebra in the symmetric monoidal $\infty$-category $\mathcal{C}^{\otimes} := N_{\text{dg}}(\ch(\field{k}))^{\otimes}$.\\\par
To this end, let $\phi: C_{\ast}(\mathbb{E}^T_1) \xrightarrow{\simeq} \mathbf{Assoc}$ be the projection map. Then we get $\phi^{\ast}A \in \alg_{C_{\ast}(\mathbb{E}^T_1)}(\ch(\field{k}))$. If $\tilde{A}\xtwoheadrightarrow{\simeq} \phi^{\ast}A$ is a cofibrant replacement, we can then use Theorem \ref{thm3} to get an object $\Phi(\tilde{A})\in \alg_{\mathbb{E}_1}(N_{\text{dg}}(\ch(\field{k}))$, which we by abuse of notation again denote by $\tilde{A}$. We arrive at our next main result:
\begin{theorem}\label{thm1}
    For any projective resolution $P\xtwoheadrightarrow{\simeq}A$ of $A$ as an $A^e$-module, the evaluation map 
    \begin{align*}
        \textnormal{ev}: \map_{\ch(A^e)}(P,P) \otimes P \rightarrow P
    \end{align*}
    makes the Hochschild complex of $A$ into a center of $\tilde{A}\in \alg_{\mathbb{E}_1}(\ch(\field{k}))$. In particular, this makes $\map_{\ch(A^e)}(P,P)$ into an object of $\alg_{\mathbb{E}_1}(\alg_{\mathbb{E}_1}(N_{\textnormal{dg}}(\ch(\field{k}))))\simeq \alg_{\mathbb{E}_2}(N_{\textnormal{dg}}(\ch(\field{k})))$.
\end{theorem}
Clearly, to prove the theorem we will need to apply our results in Section \ref{the_center_as_endomorphism_object_of_bimodules} and \ref{symmetric_monoidal_dg_model_categories} on endomorphism objects in bimodules, and dg mapping objects. Hence, the first step is to identify derived modules over $A^e$ with modules over $\mathbb{E}_1$, or equivalently via the Quillen equivalence (\ref{dig1}), modules over the enveloping algebra $U_{\eone}(\tilde{A})$.
\begin{lemma}\label{prop2}
There exists a zig-zag of quasi-isomorphisms between $U_{\eone}(\tilde{A})$ and $A\otimes A^{\textnormal{op}}$.
\end{lemma}
\begin{proof}
Let $\theta: \mathbb{A}_{\infty} \xtwoheadrightarrow{\simeq} \mathbf{Assoc}$ be a cofibrant replacement of dg operads, and note that we have a diagram
\begin{center}
    \begin{tikzcd}
	& \eone \\
	{\mathbb{A}_{\infty}} & {\mathbf{Assoc}}
	\arrow["{\phi,\simeq}", two heads, from=1-2, to=2-2]
	\arrow["{\psi, \simeq}", from=2-1, to=1-2]
	\arrow["{\theta,\simeq}"', two heads, from=2-1, to=2-2]
\end{tikzcd}.
\end{center}
Since we work in characteristic zero, all of the involved operads are automatically admissible and $\Sigma$-cofibrant. In particular, all the above weak equivalences are strong equivalences of operads, and thus induce a Quillen equivalence between their respective algebra categories. 
Let $A'\xtwoheadrightarrow{\simeq} A$ be a cofibrant replacement in associative algebras, and let $\hat{A} \xtwoheadrightarrow{\simeq} \theta^{\ast}A'$ be a cofibrant replacement in $\mathbb{A}_{\infty}$-algebras. Then in particular, the unit map $\hat{A} \rightarrow \psi^{\ast}\psi_! \hat{A}$ is a weak equivalence, and hence using \cite[Theorem 17.4.A, 17.4.B]{F} we get the following diagram of weak equivalences of dg algebras
\begin{center}
    \begin{tikzcd}[sep=small]
	&&& {U_{\mathbb{A}_{\infty}}(\hat{A})} \\
	&& {U_{\mathbb{A}_{\infty}}(\theta^{\ast}A')} && {U_{\mathbb{A}_{\infty}}(\psi^{\ast}\psi_!\hat{A})} \\
	& {A'\otimes A'^{\text{op}} = U_{\text{Assoc}}(A')} &&&& {U_{\eone}(\psi_!\hat{A})} \\
	{A\otimes A^{\text{op}}}
	\arrow["\simeq"', from=1-4, to=2-3]
	\arrow["\simeq", from=1-4, to=2-5]
	\arrow["{\theta_{\flat}}"', from=2-3, to=3-2]
	\arrow["{\psi_{\flat}}", from=2-5, to=3-6]
	\arrow["\simeq"', from=3-2, to=4-1]
\end{tikzcd}
\end{center}
Hence it suffices to show that $U_{\eone}(\psi_!\hat{A})$ is quasi-isomorphic to $U_{\eone}(\tilde{A})$. To this end, let $A_1 \xtwoheadrightarrow{\simeq} \phi^{\ast}A'$ be a cofibrant replacement. Note that $\psi_! \hat{A}$ is still cofibrant, and we hence have a lift $f:\psi_! \hat{A} \rightarrow A_1$ in the diagram
\begin{center}
    \begin{tikzcd}
\emptyset \arrow[d] \arrow[r]                                & A_1 \arrow[d, "\simeq", two heads] \\
\psi_!\tilde{A} \arrow[r, "\simeq"'] \arrow[ru, "f", dashed] & \phi^{\ast}A'                     
\end{tikzcd}.
\end{center}
By 2-out-of-3, $f$ must be a weak equivalence. This is a weak equivalence between cofibrant objects, so again by \cite[Theorem 17.4.A]{F}, we get a quasi-isomorphism $U_{\eone}(\psi_!\hat{A}) \xrightarrow{\simeq} U_{\eone}(A_1)$. The map $\phi^{\ast} A' \rightarrow \phi^{\ast}A$ is a trivial fibration as $\phi^{\ast}$ is right Quillen, and in particular the composition
\begin{align*}
    A_1 \xtwoheadrightarrow{\simeq} \phi^{\ast}A' \xtwoheadrightarrow{\simeq} \phi^{\ast}A
\end{align*}
is again a cofibrant replacement. We now again find a lift $g: A_1 \rightarrow \tilde{A}$ in the diagram
\begin{center}
    \begin{tikzcd}
\emptyset \arrow[d] \arrow[r]                               & \tilde{A} \arrow[d, "\simeq", two heads] \\
A_1 \arrow[r, "\simeq"', two heads] \arrow[ru, "g", dashed] & \phi^{\ast}A                            
\end{tikzcd}
\end{center}
which again is a weak equivalence, and thus finally induces a quasi-isomorphism
\begin{align*}
    U_{\eone}(A_1) \xrightarrow{\simeq} U_{\eone}(\tilde{A}).
\end{align*}
Summarizing, we get the following zig-zag
\begin{align*}
    A\otimes A^{\text{op}} \xleftarrow{\simeq} U_{\mathbb{A}_{\infty}}(\hat{A}) \xrightarrow{\simeq} U_{\eone}(\tilde{A}).
\end{align*}
\end{proof}
\begin{proof}[Proof of Theorem \ref{thm1}]
By Corollary \ref{cor8} and Corollary \ref{cor7}, the center of $\tilde{A}$ is given by the morphism object $\mor_{\textnormal{Mod}^{\mathbb{E}_1}_{\tilde{A}}(\mathcal{C}^{\otimes} \times_{\fin} \mathbb{E}_1^{\otimes})_{\mathfrak{a}}}(\tilde{A},\tilde{A})$ together with its evaluation map 
\begin{align*}
\alpha: \mor_{\textnormal{Mod}^{\mathbb{E}_1}_{\tilde{A}}(\mathcal{C}^{\otimes} \times_{\fin} \mathbb{E}_1^{\otimes})_{\mathfrak{a}}}(\tilde{A},\tilde{A}) \otimes \tilde{A} \rightarrow \tilde{A}.
\end{align*}
It therefore suffices to exhibit $\text{ev}: \map_{\ch(A^e)}(P,P) \otimes P \rightarrow P$ as such a morphism object. By the rectification results in Section \ref{the_category_of_e1_modules_in_chain_complexes} and the Quillen equivalence (\ref{dig1}), the $\mathbb{E}_1$-module category $\textnormal{Mod}^{\mathbb{E}_1}_{\tilde{A}}(\mathcal{C}^{\otimes} \times_{\fin} \mathbb{E}_1^{\otimes})_{\mathfrak{a}}$ is the underlying $\infty$-category of the model category $\lmod_{U_{\eone}(\tilde{A})}(\ch(\field{k}))$. By Lemma \ref{prop2}, we further have a Quillen equivalence 
\begin{align*}
    \lmod_{U_{\eone}(\tilde{A})}(\ch(\field{k})) \simeq \lmod_{A\otimes A^{\text{op}}}(\ch(\field{k})) \cong \ch(A^e),
\end{align*}
and $\ch(A^e)$ is a dg model category; in particular it is left tensored over the symmetric monoidal dg model category $\ch(\field{k})$. We can hence apply Theorem \ref{thm4}: $P\in \ch(A^e)^{\circ}$ is bifibrant by assumption, and we already know that $\map_{\ch(A^e)}(P,P)$ together with the evaluation map is a dg endomorphism object for $P$. Therefore, it also is an endomorphism object for $P\in N_{\text{dg}}(\ch(A^e)^{\circ})$, and the equivalence of $\infty$-categories $N_{\text{dg}}(\ch(A^e)^{\circ})\simeq \text{Mod}^{\mathbb{E}_1}_{\tilde{A}}(\mathcal{C}^{\otimes}\times_{\fin} \mathbb{E}_1^{\otimes})_{\mathfrak{a}}$ identifies $P$ with $\tilde{A}$. This proves the claim.
\end{proof}
Now that we have established the Hochschild complex as a center, we can compute its bracket as an $\mathbb{E}_2$-algebra, as in Section \ref{the_bracket_operation_of_a_2_algebra}, and compare to the classical Gerstenhaber bracket. Recall that we view the Hochschild cochain complex 
\begin{align*}
    C^{\ast}(A,A) \simeq \hom_{\field{k}}(A^{\otimes \ast},A)
\end{align*}
as a chain complex concentrated in non-positive degrees with component in degree $-n$ given by $\hom_{\field{k}}(A^{\otimes n},A)$. We will follow the sign conventions in \cite{Wit}.
\begin{definition}
The \textbf{signed Gerstenhaber bracket} is the degree 1 map
\begin{align*}
    \hom_{\field{k}}(A^p,A) \otimes \hom_{\field{k}}(A^q,A) &\rightarrow \hom_{\field{k}}(A^{p+q-1},A)\\
    f\otimes g &\mapsto (-1)^{p+1} (f\{g\} - (-1)^{(p-1)(q-1)} g\{f\})
\end{align*}
with $-\{-\}$ the circle product as defined in \cite[Definition 1.4.1]{Wit}
\begin{align*}
    f\{g\}(a_1\otimes \dots \otimes a_{p+q-1}) = \sum_{i=1}^p (-1)^{\epsilon_i} f(a_1\otimes \dots \otimes a_{i-1} \otimes g(a_i \otimes \dots \otimes a_{i+q})\otimes a_{i+q+1} \otimes \dots \otimes a_{p+q-1})
\end{align*}
and $\epsilon = (q-1)(i-1)$. We call the Gerstenhaber algebra structure on the Hochschild cohomology with the signed cup product and the signed Gerstenhaber bracket the \textbf{signed classical Gerstenhaber algebra structure}.
\end{definition}
\begin{corollary}\label{cor9}
Take $P= B(A) \xtwoheadrightarrow{\simeq}A$ to be the bar resolution. The strictification of the Hochschild complex ${\hom}_{\field{k}}(A^{\otimes \ast},A)\in \ch(\field{k})$ as the center of $\tilde{A}$ naturally carries the structure of a $C_{\ast}(\mathbb{E}^T_2)$-algebra. This $C_{\ast}(\mathbb{E}^T_2)$-algebra structure recovers the signed classical Gerstenhaber algebra structure in cohomology.
\end{corollary}
\begin{proof}
By Theorem \ref{thm1}, the classical Hochschild complex is a center of $\tilde{A}$ and thus inherits a $\mathbb{E}_2$-algebra structure in the derived $\infty$-category. The first part of the corollary follows directly from the Rectification Theorem \ref{thm3}. For the second part, note that by Corollary \ref{cor10}, the underlying Gerstenhaber structure on cohomology of the strictification agrees with the Gerstenhaber structure computed directly from the $\mathbb{E}_2$-algebra $\hom_{\field{k}}(A^{\otimes \ast},A)$.\par 
By Corollary \ref{cor7}, the two multiplications of the 2-algebra structure on the Hochschild complex as center are given by composition and convolution respectively. But both of these recover the formula for the signed cup product for the bar resolution. Hence, we get the classical signed cup product in cohomology. \par
Again by Corollary \ref{cor7}, the compatibility square has a contractible choice of fillings, and it therefore suffices to find a chain homotopy $h: \hom_{\field{k}}(A^{\otimes \ast},A)^{\otimes 4} \rightarrow \hom_{\field{k}}(A^{\otimes (\ast -1)}, A)$ in the square
\begin{center}
    \begin{tikzcd}
	{{\hom}_{\field{k}}(A^{\otimes \ast},A)^{\otimes 4}} & {{\hom}_{\field{k}}(A^{\otimes \ast},A)^{\otimes 2}} \\
	{{\hom}_{\field{k}}(A^{\otimes \ast},A)^{\otimes 2}} & {{\hom}_{\field{k}}(A^{\otimes \ast},A)}
	\arrow["{\smile \otimes \smile}", from=1-1, to=1-2]
	\arrow["{(\smile \otimes \smile) \circ (\text{id}\otimes \tau \otimes \text{id})}"', from=1-1, to=2-1]
	\arrow["h"', shorten <=14pt, shorten >=11pt, Rightarrow, from=1-2, to=2-1]
	\arrow["\smile", from=1-2, to=2-2]
	\arrow["\smile"', from=2-1, to=2-2]
\end{tikzcd}.
\end{center}
Recall that the circle product $f\otimes g \mapsto f\{g\}$ witnesses the homotopy commutativity of the signed cup product $f\otimes g \mapsto f\smile g$ via the following equality \cite[Lemma 1.4.5]{Wit}:
\begin{align*}
    f\smile g - (-1)^{pq} g\smile f = (-1)^{(p+1)q} \partial g\{f\} + (-1)^{(p+1)q-1} \partial(g\{f\}) + (-1)^{pq+1} g\{\partial f\}.
\end{align*}
with $|f| = -p$ and $|g| = -q$, and $\partial f = (-1)^p df$ with $d$ the usual Hochschild codifferential. This in turn yields a homotopy for the square by setting 
\begin{align*}
    h: f_1\otimes g_1 \otimes f_2 \otimes g_2 \mapsto (-1)^{|f_1|+|f_1|+|f_2||g_1|-1} f_1 \smile f_2\{g_1\}\smile g_2,
\end{align*}
with codifferential $\partial$ on the Hochschild complex. Restricting along $\iota_{1,4}$ yields the trivial homotopy between the cup product and itself, and restricting along $\iota_{2,3}$ yields the homotopy 
\begin{align*}
    H: {\hom}_{\field{k}}(A^{\otimes p},A) \otimes {\hom}_{\field{k}}(A^{\otimes q},A) &\rightarrow {\hom}_{\field{k}}(A^{\otimes p+q-1},A)\\
    f\otimes g &\mapsto (-1)^{pq+q-1} g\{f\}
\end{align*}
between $\smile$ and $\smile \circ \tau$. For the opposite multiplications, we get 
\begin{align*}
    h^{\text{op}}: f_1\otimes g_1\otimes f_2\otimes g_2 \mapsto (-1)^{\alpha + |f_2||g_1| + |f_2| -1} g_2 \smile g_1\{f_2\} \smile f_1
\end{align*}
with $\alpha = |f_1||f_2| + |g_1||g_2| + |f_1||g_1| + |f_2||g_2| + |f_1||g_2| + |f_2||g_1|$. Restricting along $\iota_{1,4}$ again yields the trivial homotopy, and restricting along $\iota_{2,3}$ yields 
\begin{align*}
    H^{\text{op}}: f\otimes g \mapsto (-1)^{p+1} f\{g\}.
\end{align*}
Hence by the remark after Corollary \ref{cor5}, the bracket is given by
\begin{align*}
    [f,g] = H(f\otimes g) + H^{\text{op}}(f\otimes g) &= (-1)^{|f||g|+|g|-1} g\{f\} + (-1)^{|f|+1} f\{g\} \\&= (-1)^{|f|+1}(f\{g\} - (-1)^{(|f|+1)(|g|+1)} g\{f\}) \\&= (-1)^{|f|+1} [f,g]_G.
\end{align*}
This agrees with the signed classical Gerstenhaber bracket, proving the claim. 
\end{proof}
\section{The Hochschild complex of a scheme} \label{The_Hochschild_complex_of_a_scheme}
As explained in Section \ref{delignes_conjecture_on_hochschild_cochains}, we want to globalize the above results to a quasi-compact separated scheme $X$ over $\field{k}$. We follow the same reasoning as in the previous section. The structure sheaf $\mathcal{O}_X$ is an associative algebra object in the category of (pre)sheaves of $\field{k}$-modules on $X$. We consider the dg category $\dgpsh{X}$ of complexes of presheaves of $\field{k}$-modules on $X$, and show that it can be equipped with a dg model category structure that presents the $\infty$-category of dg sheaves on $X$. We can make $\mathcal{O}_X$ into an $C_{\ast}(\mathbb{E}_1^T)$-algebra in $\dgpsh{X}$, which by Theorem \ref{thm3} produces an $\mathbb{E}_1$-algebra in the associated $\infty$-category of dg sheaves. We then argue that the center of this $\mathbb{E}_1$-algebra is a good model for the Hochschild cochain complex of $X$. Note that this does not require $X$ to be smooth. 
\subsection{The $\infty$-category of dg sheaves}\label{the_infty_category_of_dg_sheaves}
Let $X$ be a quasi-compact separated scheme over $\field{k}$. Since the category of presheaves of complexes of $\field{k}$-modules on $X$ is a functor category, it admits an injective and a projective model structure. These model structures do not know anything about the geometry of $X$; in particular the bifibrant objects are not "homotopy sheaves" in any way. In order to present the $\infty$-category of dg sheaves, one needs to localize these model structures at Čech nerves of covering families. In this section, we will recall how to construct the "local projective" model structure, and prove some basic properties of this model category. Note that we use presheaves instead of sheaves since the category of dg sheaves does not have enough projectives, and the injective model structure does not behave well with respect to tensor products.
\begin{proposition}[\cite{Hin2}, Theorem 1.3.1]
Let $S$ be a site. There is a cofibrantly generated model structure on the category $\dgpsh{S}$ of presheaves of $\field{k}$-module complexes on $S$ with 
\begin{itemize}
    \item weak equivalences the maps $f: \mathcal{F} \rightarrow \mathcal{G}$ such that the degreewise sheafification $f^a: \mathcal{F}^a \rightarrow \mathcal{G}^a$ is a quasi-isomorphism of complexes of sheaves,
    \item cofibrations generated by maps $f: \mathcal{F} \rightarrow \mathcal{F}\langle x;dx = z\in \mathcal{F}(U)\rangle$ corresponding to adding a section to kill a cycle z over $U\in S$, and
    \item fibrations the maps $f: \mathcal{F} \rightarrow \mathcal{G}$ such that $f(U): \mathcal{F}(U) \rightarrow \mathcal{G}(U)$ is surjective for all $U\in S$ and for any hypercover $\epsilon: V_{\bullet} \rightarrow U$ the diagram
    \begin{center}
        \begin{tikzcd}
\mathcal{F}(U) \arrow[d, "f(U)"'] \arrow[r] & {\check{C}(V_{\bullet},\mathcal{F})} \arrow[d] \\
\mathcal{G}(U) \arrow[r]                    & {\check{C}(V_{\bullet},\mathcal{G})}          
\end{tikzcd}
    \end{center}
    is a homotopy pullback. 
\end{itemize}
\end{proposition}
This is the left Bousfield localization of the projective model structure on $\dgpsh{S}$ with respect to the Čech complexes of hypercoverings. We therefore call it the \textbf{local projective model structure}. In particular, acyclic fibrations in the local projective model structure are just acyclic fibrations in the unlocalized projective model structure. Note that if $S$ has enough points, weak equivalences can be detected at stalks, i.e. a map $f: \mathcal{F} \rightarrow \mathcal{G}$ is a weak equivalence if and only if for each point $x$ of $S$, the induced map $f_x: \mathcal{F}_x \rightarrow \mathcal{G}_x$ is a quasi-isomorphism.
\begin{notation}
Let $X$ be a scheme over $\field{k}$. Let $\textnormal{Aff}(X)$ be the site of affine open subsets of $X$, and let $\textnormal{Open}(X)$ the site of all open subsets of $X$. We call the associated categories of dg presheaves $\affpsh{X}$ and $\dgpsh{X}$ respectively. We have a natural inclusion $\iota: \textnormal{Aff}(X) \rightarrow \textnormal{Open}(X)$ that induces to a restriction functor
\begin{align*}
    \iota_{\ast}: \dgpsh{X} \rightarrow\affpsh{X}.
\end{align*}
\end{notation}
\begin{proposition}\label{prop3}
The restriction functor $\iota_{\ast}$ admits a left adjoint $\iota^{-1}$, and the pair $\iota^{-1}\dashv \iota_{\ast}$ forms a Quillen equivalence. Both $\iota_{\ast}$ and $\iota^{-1}$ preserve weak equivalences, and $\iota^{-1}$ preserves acyclic fibrations. The unit $\textnormal{id} \Rightarrow \iota_{\ast} \iota^{-1}$ is an isomorphism, and the counit $\iota^{-1}\iota_{\ast} \Rightarrow \textnormal{id}$ is a component-wise weak equivalence.
\end{proposition}
\begin{proof}
The left adjoint $\iota^{-1}$ is given by $(\iota^{-1}\mathcal{F})(V) = \colim_{V\subseteq U \in {\textnormal{Aff}(X)}}\mathcal{F}(U)$. The direct image $\iota_{\ast}$ clearly preserves acyclic fibrations, since these are pointwise. The sites of all opens and of affine opens have the same points, namely points in the topological space $X$. This follows because affine opens form a basis of the Zariski topology. Even more, $\iota_{\ast}$ and $\iota^{-1}$ preserve stalks at these points. Taking stalks is a left adjoint, so this follows trivially for the inverse image, and for the direct image we note that small enough neighborhoods of a point $x\in X$ always contain an affine open neighborhood of $x$. This shows that both adjoints preserve weak equivalences, and in particular $\iota^{-1}$ preserves acyclic cofibrations. The fact that $\iota^{-1}$ preserves acyclic fibrations follows from the fact that filtered colimits are exact in Grothendieck categories. Finally, note that if $U$ is affine, then $\colim_{U\subseteq W \in {\textnormal{Aff}(X)}}\mathcal{F}(W) \cong F(U)$ since $U$ is final in the index category. This shows that the unit is an isomorphism. The fact that the counit is a component-wise weak equivalence again follows from the fact that both adjoints preserve weak equivalences.
\end{proof}
\begin{definition}
    We call the underlying $\infty$-category of the local projective model structure on $\dgpsh{X}$ the \textbf{$\infty$-category of dg sheaves on $X$}
    \begin{align*}
        \textnormal{Sh}_{\infty}(X) := N(\dgpsh{X}^c)[W^{-1}].
    \end{align*}
\end{definition}
By Proposition \ref{prop3}, we have an equivalence of $\infty$-categories
\begin{align*}
    \iota_{\ast}: \textnormal{Sh}_{\infty}(X)\rightarrow N(\affpsh{X}^c)[W^{-1}] 
\end{align*}
with quasi-inverse $\iota^{-1}: N(\affpsh{X}^c)[W^{-1}]\rightarrow \textnormal{Sh}_{\infty}(X)$.
\begin{remark}
Even though the model categories of presheaves on affine opens and general opens yield the same $\infty$-category, the above model category structure depends on the choice of site. On affine open subsets, all quasi-coherent sheaves on $X$ are automatically fibrant, which is not true for general opens. In particular, on affine opens the structure sheaf $\mathcal{O}_X$ itself is fibrant.
\end{remark}
\begin{proposition}
    If the topos on $S$ has enough points and $S$ admits finite products, then the local projective model structure yields a closed symmetric monoidal model category. If in addition $S$ admits a final object, then $\dgpsh{S}$ is a symmetric monoidal dg model category.
\end{proposition}
\begin{proof}
By \cite[Proposition 7.9]{PS}, the global projective model structure on $\dgpsh{S}$ inherits the structure of a symmetric monoidal model category since $S$ admits finite products. By \cite[Theorem 4.6]{W}, to show that this symmetric monoidal model structure descends to the local projective model structure, it suffices to argue that for $f$ a local weak equivalence and $\mathcal{F}$ a cofibrant object, the map $f\otimes \textnormal{id}_{\mathcal{F}}$ is again a local weak equivalence. But this is clear if the topos has enough points, since we can then check local weak equivalences on stalks. Since the presheaf category admits an internal hom given by
\begin{align*}
    \mathcal{H}om(\mathcal{F},\mathcal{G})(U) = \map_{\dgpsh{S_U}}(\mathcal{F}|_U, \mathcal{G}|_U),
\end{align*}
this proves the first part. For the second part, note that presheaves of chain complexes are the same as chain complexes of presheaves of $\field{k}$-modules.  Since the later is an abelian category, this category automatically admits a dg enrichment. Recall that if $\ast\in S$ is terminal, we have the constant presheaf functor
\begin{align*}
    \textnormal{const}_{\ast}: \ch(\field{k}) &\rightarrow \dgpsh{S}\\
    C &\mapsto (U \mapsto C).
\end{align*}
We define a tensoring
\begin{align*}
    \ch(\field{k}) \times \dgpsh{S} \rightarrow \dgpsh{S}, \quad (C,\mathcal{F}) \mapsto \textnormal{const}_{\ast}(C) \otimes \mathcal{F}
\end{align*}
as well as a powering
\begin{align*}
    \ch(\field{k})^{\textnormal{op}}\times \dgpsh{S} \rightarrow \dgpsh{S}, \quad (C,\mathcal{F}) \mapsto \mathcal{H}om(\textnormal{const}_{\ast}(C),\mathcal{F}).
\end{align*}
One easily checks that these indeed satisfy the correct adjointness properties. It hence suffices to check the pushout-product axiom. The argument below shows that $\textnormal{const}_{\ast}$ preserves cofibrations. Hence if $i: C\rightarrow D$ is a cofibration in $\ch(\field{k})$, then $\textnormal{const}_{\ast}(i): \textnormal{const}_{\ast}(C) \rightarrow \textnormal{const}_{\ast}(D)$ is a cofibration in $\dgpsh{S}$, and therefore the pushout-product axiom follows directly from the pushout-product axiom in $\dgpsh{S}$.
\end{proof}
\begin{corollary}
Let $X$ be a separated quasi-compact scheme over $\field{k}$. The category $\dgpsh{X}$ admits the structure of a symmetric monoidal dg model category. In particular, the $\infty$-category $\textnormal{Sh}_{\infty}(X)$ admits the structure of a symmetric monoidal $\infty$-category $\textnormal{Sh}_{\infty}(X)^{\otimes}$.
\end{corollary}
Let $U\subseteq X$ be an affine open. Then we have adjoint functors
\begin{center}
    \begin{tikzcd}
            \ch(\field{k})\arrow[r, shift left=1ex, "C_U"{name=G}] & \affpsh{X}\arrow[l, shift left=.5ex, "\Gamma_U"{name=F}]
            \arrow[phantom, from=F, to=G, , "\scriptscriptstyle\boldsymbol{\dashv}" rotate=-90]
        \end{tikzcd}
\end{center}
where $\Gamma_U$ sends a complex of presheaves $\mathcal{F}$ to $\mathcal{F}(U)$ and $C_U$ is the constant presheaf functor sending $C$ to the presheaf 
\begin{align*}
    V\mapsto \begin{cases}
        C & \text{if } V\subseteq U\\
        0 & \text{otherwise}
    \end{cases}.
\end{align*}
If we equip $\affpsh{X}$ with the projective model structure, then $\Gamma_U$ preserves fibrations and weak equivalences by construction. Therefore we obtain a Quillen adjunction. We can compose this with the Quillen adjucntion
\begin{center}
    \begin{tikzcd}
            \affpsh{X}^{\text{proj}}\arrow[r, shift left=1ex, "\text{id}"{name=G}] & \affpsh{X}^{\text{loc}}\arrow[l, shift left=.5ex, "\text{id}"{name=F}]
            \arrow[phantom, from=F, to=G, , "\scriptscriptstyle\boldsymbol{\dashv}" rotate=-90]
        \end{tikzcd}
\end{center}
of the Bousfield localization to obtain a Quillen adjunction
\begin{center}
    \begin{tikzcd}
            \ch(\field{k})\arrow[r, shift left=1ex, "C_U"{name=G}] & \affpsh{X}^{\text{loc}}\arrow[l, shift left=.5ex, "\Gamma_U"{name=F}]
            \arrow[phantom, from=F, to=G, , "\scriptscriptstyle\boldsymbol{\dashv}" rotate=-90].
        \end{tikzcd}
\end{center}
Since $C_U$ is left Quillen, it preserves weak equivalences between cofibrant objects. But every object in $\ch(\field{k})$ is cofibrant, so $C_U$ preserves weak equivalences.\\

The same argument works if we instead consider the site of all opens $\textnormal{Open}(X)$. In this case, for any open $V\subseteq X$ we obtain a Quillen adjunction $C_V \dashv \Gamma_V$. Note that in particular we can then take $V=X$. If $V$ is affine, this agrees with the above construction. The functors $C_V$ and $\Gamma_V$ are both strong symmetric monoidal, since the tensor product of presheaves is taken section-wise. In particular, we obtain lax symmetric monoidal functors of $\infty$-categories $C_V: \mathcal{D}_{\infty}(\field{k})^{\otimes} \rightarrow \text{Sh}_{\infty}(X)^{\otimes}$ and $\mathbb{R}\Gamma_V: \text{Sh}_{\infty}(X)^{\otimes} \rightarrow \mathcal{D}_{\infty}(\field{k})^{\otimes}$.
\begin{notation}
    Let $\mathcal{O}$ be a dg operad. We then have an operad $C_X(\mathcal{O})$ in $\dgpsh{X}$. By abuse of notation, we will usually denote the operad $C_X(\mathcal{O})$ just by $\mathcal{O}$.
\end{notation}
\begin{lemma}\label{lem2}
    The functor $C_X$ preserves cofibrancy and $\Sigma$-cofibrancy of operads, as well as weak equivalences of operads. Every operad in $\dgpsh{X}$ is admissible \cite[Definition 5.1.]{PS}, and even strongly admissible \cite[Definition 5.1.]{PS} if it lies in the image of $C_X$.
\end{lemma}
\begin{proof}
The Quillen adjunction $C_X\dashv \Gamma_X$ induces adjunctions between the respective categories of symmetric collections and symmetric operads, since both are strong symmetric monoidal. The model structure on symmetric collection is transferred from the underlying model category, and hence the adjunction is again Quillen. Similarly, fibrations and weak equivalences of operads are pointwise, and hence $\Gamma_X$ preserves fibrations and trivial fibrations of operads. To see that $C_X$ preserves weak equivalences, note that these are point-wise in operads, and $C_X$ preserves weak equivalences on the underlying model categories. To see that operads in $\dgpsh{X}$ are admissible, use \cite[Theorem 5.11]{PS} and Section 8 of \cite{PS2}. Now to see that every operad in the image of $C_X$ is even strongly admissible, use \cite[Proposition 6.3]{PS} together with the fact that every operad in $\ch(\field{k})$ is $\Sigma$-cofibrant.
\end{proof}
This shows that we get a diagram of admissible $\Sigma$-cofibrant operads
\begin{center}
    \begin{tikzcd}
                                                                                             & C_X(C_{\ast}(\mathbb{E}^T_1)) \arrow[d, "{C_X(\phi), \simeq}"] \\
C_X(\mathbb{A}_{\infty}) \arrow[ru, "{C_X(\psi),\simeq}"] \arrow[r, "{C_X(\theta),\simeq}"'] & C_X(\mathbf{Assoc})                                           
\end{tikzcd}
\end{center}
and $C_X(\mathbb{A}_{\infty})$ is still cofibrant.\\\par
If $X$ is a quasi-compact seperable scheme over $\field{k}$ with structure sheaf $\mathcal{O}_X$, we consider the $C_{\ast}(\mathbb{E}^T_1)$-algebra $\phi^{\ast}\mathcal{O}_X$ and choose a cofibrant replacement $\tilde{\mathcal{O}}_X \xtwoheadrightarrow{\simeq}\phi^{\ast}\mathcal{O}_X$. Applying the functor $\Phi$ constructed in Section \ref{Rectification_of_algebras_over_an_infty_operad}, we get $\Phi(\tilde{\mathcal{O}}_X)\in \alg_{\mathbb{E}_1}(\text{Sh}_{\infty}(X))$. As in the affine case, we will keep denoting this by just $\tilde{\mathcal{O}}_X$.
\begin{definition}
Let $X$ be a quasi-compact separated scheme over $\field{k}$. We define the \textbf{Hochschild cochain complex} of $X$ to be the center
\begin{align*}
    \mathfrak{Z}(\tilde{\mathcal{O}}_X)\in \alg_{\mathbb{E}_1}(\alg_{\mathbb{E}_1}(\text{Sh}_{\infty}(X))) \simeq \alg_{\mathbb{E}_2}(\text{Sh}_{\infty}(X)).
\end{align*}
\end{definition}
\subsection{The Hochschild complex of a scheme is local}
In this section we will prove the following theorem, showing that the center of a scheme glues together the affine Hochschild complexes.
\begin{theorem}\label{thm5}
Let $U = \textnormal{Spec}(A)\subseteq X$ be an affine open. The map $\mathbb{R}\Gamma_U: \textnormal{Sh}_{\infty}(X) \rightarrow \mathcal{D}_{\infty}(\field{k})$ is lax symmetric monoidal and hence induces a map $\mathbb{R}\Gamma_U: \alg_{\mathbb{E}_2}(\textnormal{Sh}_{\infty}(X)) \rightarrow \alg_{\mathbb{E}_2}(\mathcal{D}_{\infty}(\field{k}))$. We have
\begin{align*}
    \mathbb{R}\Gamma_U(\mathfrak{Z}(\tilde{O}_X)) \simeq \mathfrak{Z}(\tilde{A})\in \alg_{\mathbb{E}_2}(\mathcal{D}_{\infty}(\field{k}))
\end{align*}
for any cofibrant replacement $\tilde{A}$ of $\phi^{\ast}A$.
\end{theorem}
Just like in the affine case, the center has underlying object
\begin{align*}
    \mathfrak{Z}(\tilde{\mathcal{O}}_X)(\mathfrak{a}) \simeq \text{Mor}_{\text{Mod}^{\mathbb{E}_1}_{\tilde{\mathcal{O}}_X}(\text{Sh}_{\infty}(X)\times_{\fin} \mathbb{E}_1)_{\mathfrak{a}}}(\tilde{\mathcal{O}}_X,\tilde{\mathcal{O}}_X) \in \text{Sh}_{\infty}(X),
\end{align*}
and it hence suffices to understand this endomorphism object. To this end, note that we can adapt Hinich's Rectification Theorem for modules \cite[Theorem 5.2.3]{Hin} to the local projective model structure on complexes of presheaves by using the generalized Rectification Theorem \ref{thm3}. We hence have an equivalence of $\infty$-categories
\begin{align*}
    N(\text{Mod}_{\tilde{\mathcal{O}}_X}^{C_{\ast}(\mathbb{E}^T_1)}(\dgpsh{X}^c))[W_{\text{Mod}}^{-1}] \simeq   \text{Mod}^{\mathbb{E}_1}_{\tilde{\mathcal{O}}_X}(\text{Sh}_{\infty}(X)\times_{\fin}\mathbb{E}_1)_{\mathfrak{a}}.
\end{align*}
By \cite[Theorem 1.5.1]{Hin2}, for any associative algebra object $\mathcal{O}$ in $\dgpsh{X}$, the category $\lmod_{\mathcal{O}}(\dgpsh{X})$ carries a model structure transferred from the local projective model structure on $\dgpsh{X}$. Again following the affine case, we have a Quillen equivalence
\begin{align*}
        \text{Mod}_{\tilde{\mathcal{O}}_X}^{C_{\ast}(\mathbb{E}^T_1)}(\dgpsh{X}) \cong \lmod_{U_{C_{\ast}(\mathbb{E}^T_1)}(\tilde{\mathcal{O}}_X)}(\dgpsh{X}).
\end{align*}
\begin{proposition}
    There exists a zig-zag of weak equivalences betwen $U_{C_{\ast}(\mathbb{E}^T_1)}(\tilde{\mathcal{O}}_X)$ and $\mathcal{O}_X\otimes \mathcal{O}_X$ in the category of associative algebras in $\dgpsh{X}$.
\end{proposition}
\begin{proof}
    We adapt the proof of Proposition \ref{prop2}. Let $\mathcal{O}_X' \xtwoheadrightarrow{\simeq} \mathcal{O}_X$ be a cofibrant resolution of $\mathbf{Assoc}$-algebras in $\dgpsh{X}$. Then
    \begin{align*}
        (\mathcal{O}_X'\otimes \mathcal{O}_X')_x \xrightarrow{\cong} \mathcal{O}_{X,x}' \otimes \mathcal{O}_{X,x}' \xrightarrow{\simeq} \mathcal{O}_{X,x} \otimes \mathcal{O}_{X,x} \xrightarrow{\cong} (\mathcal{O}_X\otimes \mathcal{O}_X)_x,
    \end{align*}
    showing that $\mathcal{O}_X' \otimes \mathcal{O}_X'$ is weakly equivalent to $\mathcal{O}_X\otimes \mathcal{O}_X$. The rest of the argument goes through exactly like before with the following amendment: Let $\hat{\mathcal{O}}_X \xtwoheadrightarrow{\simeq} \theta^{\ast}\mathcal{O}_X'$ be a cofibrant replacement of $\mathbb{A}_{\infty}$-algebras. To obtain the zig-zag 
    \begin{center}
        \begin{tikzcd}
                                                     & U_{\mathbb{A}_{\infty}}(\hat{\mathcal{O}}_X) \arrow[ld, "\simeq"'] \arrow[rd, "\simeq"] &                                                               \\
U_{\mathbb{A}_{\infty}}(\theta^{\ast}\mathcal{O}_X') &                                                                                         & U_{\mathbb{A}_{\infty}}(\psi^{\ast}\psi_!\hat{\mathcal{O}}_X)
\end{tikzcd}
    \end{center}
    using \cite[Theorem 17.4.A, 17.4.B]{F}, we have to argue that the underlying complexes of presheaves of these $\mathbb{A}_{\infty}$-algebras are cofibrant. To this end, recall from Lemma \ref{lem2} that $\mathbb{A}_{\infty}$ and $\mathbf{Assoc}$ are both strongly admissible, meaning that the forgetful functor from algebras preserves cofibrant objects. In particular, the underlying complexes of presheaves of $\mathcal{O}_X'$ and $\hat{\mathcal{O}}_X$ are both cofibrant. Now recall that $\psi_!$ is left Quillen and hence preserves cofibrancy, and the restriction of scalars functors $\theta^{\ast}$ and $\psi^{\ast}$ do not alter the underlying complex. This finishes the argument.
\end{proof}
Using \cite[Lemma 1.5.3]{Hin2}, we deduce that when equiped with the transfer model structure, the category $\lmod_{U_{\eone}(\tilde{\mathcal{O}}_X)}(\dgpsh{X})$ is Quillen equivalent to the category $\lmod_{\mathcal{O}_X\otimes \mathcal{O}_X}(\dgpsh{X})$. 
\begin{lemma}
    The category $\lmod_{\mathcal{O}_X\otimes \mathcal{O}_X}(\dgpsh{X})$ is a dg model category that is left tensored over the symmetric monoidal dg model category $\dgpsh{X}$. For $\mathcal{M},\mathcal{N}\in \lmod_{\mathcal{O}_X\otimes \mathcal{O}_X}(\dgpsh{X})$, we have a dg morphism object $\mathcal{H}om_{\mathcal{O}_X\otimes \mathcal{O}_X}(\mathcal{M},\mathcal{N})\in \dgpsh{X}$ with
    \begin{align*}
        \mathcal{H}om_{\mathcal{O}_X\otimes \mathcal{O}_X}(\mathcal{M},\mathcal{N})(V) =\map_{\mathcal{O}_V\otimes \mathcal{O}_V}(\mathcal{M}|_{V}, \mathcal{N}|_V).
    \end{align*}
\end{lemma}
\begin{proof}
Since $\lmod_{\mathcal{O}_X\otimes \mathcal{O}_X}(\dgpsh{X}))$ is the category of complexes of presheaves of $\mathcal{O}_X\otimes \mathcal{O}_X$-modules, we automatically have an enrichement over $\ch(\field{k})$. Now note that $\lmod_{\mathcal{O}_X\otimes \mathcal{O}_X}(\dgpsh{X}))$ is tensored and powered over $\dgpsh{X}$: If $\mathcal{F}\in \dgpsh{X}$ and $\mathcal{M} \in \lmod_{\mathcal{O}_X\otimes \mathcal{O}_X}(\dgpsh{X})$, we obtain an $\mathcal{O}_X\otimes \mathcal{O}_X$-module structure on the tensor product in complexes of presheaves by 
\begin{align*}
    (\mathcal{O}_X\otimes \mathcal{O}_X)\otimes (\mathcal{F} \otimes \mathcal{M}) \cong \mathcal{F} \otimes \mathcal{O}_X\otimes \mathcal{O}_X \otimes \mathcal{M} \rightarrow \mathcal{F} \otimes \mathcal{M}.
\end{align*}
We obtain a module structure on $\mathcal{H}om(\mathcal{F},\mathcal{M})$ by the pointwise module structure
\begin{align*}
        \mathcal{O}_X\otimes \mathcal{O}_X \otimes \mathcal{H}om(\mathcal{F},\mathcal{M}) \rightarrow \mathcal{H}om(\mathcal{M},\mathcal{M}) \otimes \mathcal{H}om(\mathcal{F},\mathcal{M}) \\\xrightarrow{\text{comp}} \mathcal{H}om(\mathcal{F},\mathcal{M}).
\end{align*}
One easily checks that these indeed yield a tensoring and powering respectively. Now to obtain the tensoring and powering over $\ch(\field{k})$, we pre-compose these operations with the constant presheaf functor $C_X: \ch(\field{k}) \rightarrow \dgpsh{X}$. The pushout-product axiom is checked in \cite[Lemma 1.6.3]{Hin2}, and Hinich also shows in the same section that the above actually yields a morphism object. 
\end{proof}
We can now again use Theorem \ref{thm4} to conclude that the center of $\tilde{\mathcal{O}}_X$ in the $\infty$-category $\alg_{\mathbb{E}_1}(\text{Sh}_{\infty}(X))$ is given by 
\begin{align*}
    \mathfrak{Z}(\tilde{\mathcal{O}}_X)(\mathfrak{a}) \simeq Q\mathcal{H}om_{\mathcal{O}_X\otimes\mathcal{O}_X}(\mathcal{O},\mathcal{O})
\end{align*}
for a bifibrant model $\mathcal{O}$ of $\mathcal{O}_X$ as an $\mathcal{O}_X\otimes \mathcal{O}_X$-module. The center action is given by the evaluation map
\begin{align*}
    R(Q\mathcal{H}om_{\mathcal{O}_X\otimes \mathcal{O}_X}(\mathcal{O},\mathcal{O}) \otimes \mathcal{O}) \rightarrow \mathcal{O}.
\end{align*}
\par Note that in contrast to the affine case, where $A$ was already fibrant as an $A \otimes A$-module, we need to take a \textit{bifibrant} resolution $\mathcal{O}$ of $\mathcal{O}_X$, since it is not fibrant in the local projective model structure on $\dgpsh{X}$. However, $\mathcal{O}_X$ is fibrant in the local projective model structure for the site of affine opens on $X$, and we have already seen that presheaves on this smaller site present the same $\infty$-category. We now show that we can indeed compute the restriction $\iota_{\ast}\mathfrak{Z}(\tilde{\mathcal{O}}_X)$ by just taking a cofibrant resolution of $\mathcal{O}_X$.
\begin{proposition}\label{prop4}
We have an equivalence
\begin{align*}
    \iota_{\ast}(Q\mathcal{H}om_{\mathcal{O}_X\otimes \mathcal{O}_X}(\mathcal{O},\mathcal{O})) \simeq Q\mathcal{H}om_{\iota_{\ast}(\mathcal{O}_X\otimes \mathcal{O}_X)}(\mathcal{P},\mathcal{P})
\end{align*}
for any cofibrant resolution $\mathcal{P} \xtwoheadrightarrow{\simeq} \iota_{\ast}\mathcal{O}_X$ in $\lmod_{\iota_{\ast}(\mathcal{O}_X\otimes\mathcal{O}_X)}(\affpsh{X})$. 
\end{proposition}
Before we prove the proposition, we need a few lemmas.
\begin{lemma}
We have a Quillen equivalence 
\begin{center}
    \begin{tikzcd}
           \lmod_{\iota_{\ast}(\mathcal{O}_X \otimes \mathcal{O}_X)}(\affpsh{X})\arrow[r, shift left=1ex, "\iota^{\ast}"{name=G}] & \lmod_{\mathcal{O}_X \otimes \mathcal{O}_X}(\dgpsh{X})\arrow[l, shift left=.5ex, "\iota_{\ast}"{name=F}]
            \arrow[phantom, from=F, to=G, , "\scriptscriptstyle\boldsymbol{\dashv}" rotate=-90].
        \end{tikzcd}
\end{center}
Both adjoints preserve weak equivalences.
\end{lemma}
\begin{proof}
The Quillen equivalence $\iota^{-1}\dashv \iota_{\ast}$ induces a Quillen equivalence
\begin{center}
    \begin{tikzcd}
           \lmod_{\iota_{\ast}(\mathcal{O}_X \otimes \mathcal{O}_X)}(\affpsh{X})\arrow[r, shift left=1ex, "\iota^{-1}"{name=G}] & \lmod_{\iota^{-1}\iota_{\ast}(\mathcal{O}_X \otimes \mathcal{O}_X)}(\dgpsh{X})\arrow[l, shift left=.5ex, "\iota_{\ast}"{name=F}]
            \arrow[phantom, from=F, to=G, , "\scriptscriptstyle\boldsymbol{\dashv}" rotate=-90].
        \end{tikzcd}
\end{center}
The counit yields a weak equivalence $\epsilon: \iota^{-1}\iota_{\ast}(\mathcal{O}_X\otimes \mathcal{O}_X) \xrightarrow{\simeq} \mathcal{O}_X\otimes\mathcal{O}_X$, and hence by \cite[Lemma 1.5.3]{Hin2} we get a Quillen equivalence
\begin{center}
    \begin{tikzcd}
           \lmod_{\iota^{-1}\iota_{\ast}(\mathcal{O}_X \otimes \mathcal{O}_X)}(\dgpsh{X})\arrow[r, shift left=1ex, "\epsilon^{\ast}"{name=G}] & \lmod_{\mathcal{O}_X \otimes \mathcal{O}_X}(\dgpsh{X})\arrow[l, shift left=.5ex, "\epsilon_{\ast}"{name=F}]
            \arrow[phantom, from=F, to=G, , "\scriptscriptstyle\boldsymbol{\dashv}" rotate=-90].
        \end{tikzcd}
\end{center}
composing these yields the Quillen equivalence in the statement. Clearly, $\epsilon_{\ast}$ preserves weak equivalences. But $\epsilon$ is an isomorphism on stalks, and hence $\epsilon^{\ast}$ also preserves weak equivalences.
\end{proof}
\begin{lemma}
For $\mathcal{F}\in \lmod_{\iota_{\ast}(\mathcal{O}_X\otimes \mathcal{O}_X)}(\affpsh{X})$ and $\mathcal{G} \in \lmod_{\mathcal{O}_X\otimes \mathcal{O}_X}(\dgpsh{X})$, we have
\begin{align*}
    \iota_{\ast}\mathcal{H}om_{\mathcal{O}_X\otimes \mathcal{O}_X}(\iota^{\ast}\mathcal{F},\mathcal{G}) \cong \mathcal{H}om_{\iota_{\ast}(\mathcal{O}_X\otimes \mathcal{O}_X)}(\mathcal{F},\iota_{\ast}\mathcal{G}).
\end{align*}
\end{lemma}
\begin{proof}
Let $U\in \textnormal{Aff}(X)$, and let $\iota': \textnormal{Aff}(U) \rightarrow \textnormal{Open}(U)$ be the inclusion. Then
\begin{align*}
    \map_{\mathcal{O}_U\otimes \mathcal{O}_U}(\iota^{\ast}\mathcal{F}|_U,\mathcal{G}|_U) &\cong \map_{\mathcal{O}_U\otimes \mathcal{O}_U}(\iota'^{\ast}(\mathcal{F}|_U),\mathcal{G}|_U) \\&\cong \map_{\iota'_{\ast}(\mathcal{O}_U\otimes \mathcal{O}_U)}(\mathcal{F}|_U,\iota'_{\ast}(\mathcal{G}|_U)) \\&\cong \map_{\iota_{\ast}(\mathcal{O}_X\otimes \mathcal{O}_X)|_U}(\mathcal{F}|_U,\iota_{\ast}\mathcal{G}|_U).
\end{align*}
\end{proof}
\begin{proof}[Proof of Proposition \ref{prop4}]
By definition, we have a diagram in $\lmod_{\mathcal{O}_X\otimes \mathcal{O}_X}(\dgpsh{X})$
\begin{center}
    \begin{tikzcd}
\mathcal{P}' \arrow[d, "\simeq"', tail] \arrow[r, "\simeq", two heads] & \mathcal{O}_X \\
\mathcal{O}                                                            &              
\end{tikzcd}
\end{center}
with $\mathcal{P}'$ cofibrant and $\mathcal{O}$ bifibrant. If $\mathcal{P} \xtwoheadrightarrow{\simeq} \iota_{\ast}\mathcal{O}_X$ is a cofibrant resolution, we get a weak equivalence $\iota^{\ast}\mathcal{P} \xrightarrow{\simeq} \iota^{\ast}\iota_{\ast} \mathcal{O}_X$ and $\iota^{\ast} \mathcal{P}$ is still cofibrant. We can hence solve the following lifting problem
\begin{center}
    \begin{tikzcd}
0 \arrow[rr] \arrow[d, tail]                                     &                                                            & \mathcal{P}' \arrow[d, "\simeq", two heads] \\
\iota^{\ast}\mathcal{P} \arrow[r, "\simeq"'] \arrow[rru, dashed] & \iota^{\ast}\iota_{\ast}\mathcal{O}_X \arrow[r, "\simeq"'] & \mathcal{O}_X                              
\end{tikzcd}
\end{center}
and by 2-out-of-3, the lift $\iota^{\ast}\mathcal{P} \rightarrow \mathcal{P}'$ is again a weak equivalence. Let $\iota^{\ast}\mathcal{P} \overset{\simeq}{\rightarrowtail} \mathcal{R}$ be a fibrant resolution. We can also solve the lifting problem
\begin{center}
    \begin{tikzcd}
\iota^{\ast}\mathcal{P} \arrow[d, "\simeq"', tail] \arrow[r, "\simeq"] & \mathcal{P}' \arrow[r, "\simeq", tail] & \mathcal{O} \arrow[d] \\
\mathcal{R} \arrow[rr] \arrow[rru, dashed]                             &                                        & 0                    
\end{tikzcd}
\end{center}
to obtain a weak equivalence $\mathcal{R} \xrightarrow{\simeq} \mathcal{O}$. In particular, this is a weak equivalence between bifibrant objects. Therefore,
\begin{align*}
    \iota_{\ast}Q\mathcal{H}om_{\mathcal{O}_X\otimes \mathcal{O}_X}(\mathcal{O}, \mathcal{O}) &\simeq \iota_{\ast}Q\mathcal{H}om_{\mathcal{O}_X\otimes \mathcal{O}_X}(\mathcal{R},\mathcal{R}) \\&\simeq \iota_{\ast}Q\mathcal{H}om_{\mathcal{O}_X\otimes \mathcal{O}_X}(\iota^{\ast}\mathcal{P},\mathcal{R}) \\&\simeq Q\mathcal{H}om_{\iota_{\ast}(\mathcal{O}_X\otimes \mathcal{O}_X)}(\mathcal{P},\iota_{\ast}\mathcal{R}) \\&\simeq Q\mathcal{H}om_{\iota_{\ast}(\mathcal{O}_X\otimes \mathcal{O}_X)}(\mathcal{P},  \mathcal{P}), 
\end{align*}
where in the last step we used that $\iota_{\ast}\iota^{\ast}(\mathcal{P}) \cong \mathcal{P}$ and that hence $\iota_{\ast} \mathcal{R} \leftarrow \iota_{\ast}\iota^{\ast}\mathcal{P}$ is a weak equivalence between fibrant objects. 
\end{proof}
Since $\iota_{\ast}$ is symmetric monoidal, Proposition \ref{prop4} shows that we get an induced $\mathbb{E}_2$-algebra structure on $Q\mathcal{H}om_{\iota_{\ast}(\mathcal{O}_X\otimes \mathcal{O}_X)}(\mathcal{P},\mathcal{P})$. Further, if $U\subseteq X$ is affine, then 
\begin{align*}
    \mathbb{R}\Gamma_U(\mathfrak{Z}(\tilde{\mathcal{O}}_X)) \simeq \mathbb{R}\Gamma_U(Q\mathcal{H}om_{\iota_{\ast}(\mathcal{O}_X\otimes \mathcal{O}_X)}(\mathcal{P},\mathcal{P}))
\end{align*}
as $\mathbb{E}_2$-algebras. \\ \par
In order to prove Theorem \ref{thm5}, it thus suffices to show that 
\begin{align*}
    \mathbb{R}\Gamma_U(Q\mathcal{H}om_{\iota_{\ast}(\mathcal{O}_X\otimes \mathcal{O}_X)}(\mathcal{P},\mathcal{P})) \simeq \map_{A\otimes A}(P,P)
\end{align*}
for $U = \textnormal{Spec}(A)$ and $P \xtwoheadrightarrow{\simeq} A$ a cofibrant replacement. Since we will solely work with the affine open site from now on, we will denote the restriction of a presheaf $\mathcal{F}$ on $\textnormal{Open}(X)$ to a presheaf on $\textnormal{Aff}(X)$ simply by $\mathcal{F}$ for the remainder of this subsection.
\begin{notation}
Let ${\diag}_X$ be the site of affine opens on $X\times_{\field{k}} X$ of the form $W\times_{\field{k}} W$ for $W\subseteq X$ affine open. Of course, this site is isomorphic to the affine open site on $X$, but it better conceptualizes sheaves coming from $A$-bimodules.
\end{notation}
\begin{lemma}\label{lem3}
\begin{enumerate}
    \item The functors
    \begin{align*}
        \Delta_{\ast}: \affpsh{X} \rightarrow \dgpsh{\diag_X}&, \quad \mathcal{F} \mapsto \Bigl(W\times_{\field{k}} W \mapsto \mathcal{F}\bigl(\Delta^{-1}(W\times_{\field{k}} W)\bigr) = \mathcal{F}(W)\Bigr),\\
        \Delta^{-1}: \dgpsh{\diag_X} \rightarrow \affpsh{X}&, \quad \mathcal{G} \mapsto \Bigl(U \mapsto \colim_{\Delta(U)\subseteq W\times_{\field{k}} W}\mathcal{G}(W\times_{\field{k}} W) \cong \mathcal{G}(U\times_{\field{k}} U)\Bigr)
    \end{align*}
    are isomorphisms of categories. 
    \item We have $\Delta_{\ast}(\mathcal{O}_X \otimes \mathcal{O}_X) \cong \mathcal{O}_{X\times_{\field{k}} X}$ and $\Delta^{-1}(\mathcal{O}_{X\times_{\field{k}} X}) \cong \mathcal{O}_X \otimes \mathcal{O}_X$. In particular, the above isomorphism yields an isomorphism
    \begin{align*}
        \lmod_{\mathcal{O}_{X\times_{\field{k}} X}}(\dgpsh{\diag_X}) \cong \lmod_{\mathcal{O}_X\otimes \mathcal{O}_X}(\affpsh{X}).
    \end{align*}
    \item Both $\Delta_{\ast}$ and $\Delta^{-1}$ preserve all three classes of fibrations, cofibrations and weak equivalences in the presheaf categories as well as in the left module categories.
    \item For any commutative $\field{k}$-algebra $A$, the adjunction $\tilde{(-)} \dashv \Gamma_{\textnormal{Spec}(A)}$ between complexes of $A$-modules and complexes of presheaves of $\mathcal{O}_{\textnormal{Spec}(A)}$-modules is a Quillen adjunction. In addition, $\tilde{(-)}$ preserves acyclic fibrations.
    \item The previous statement remains true if we consider $\tilde{(-)}$ as a functor from complexes of $A\otimes A$-modules to complexes of presheaves of $\mathcal{O}_{\textnormal{Spec}(A)}$-modules on the site $\diag_{\textnormal{Spec}(A)}$.
\end{enumerate}
\end{lemma}
\begin{proof}
Statement 1. is true by construction of the functors. \\
For 2., simply note that $\mathcal{O}_{X\times_k X}(W\times_{\field{k}} W) \cong \mathcal{O}_X(W)\otimes_{\field{k}} \mathcal{O}_X(W) \cong (\mathcal{O}_X\otimes \mathcal{O}_X)(W)$. \\
For 3., note that we get two adjoint equivalences $\Delta^{-1}\dashv \Delta_{\ast}$ and $\Delta_{\ast} \dashv \Delta^{-1}$. Clearly both $\Delta^{-1}$ and $\Delta_{\ast}$ preserve acyclic fibrations, and hence both also preserve cofibrations. Since $\diag_X$ is isomorphic as a site to $\textnormal{Aff}(X)$, the sheaf topos $\text{Sh}({\diag_X})_{\field{k}}$ also has enough points and hence we can check weak equivalences at stalks. But since $\diag_X$ is isomorphic to $\text{Aff}(X)$, \cite[Chapter VII.5, Corollary 4]{MLM} shows that all points in the sheaf topos on $\diag_X$ are of the form $\Delta(x)$ for $x\in X$, and $(\Delta_{\ast}\mathcal{F})_{\Delta(x)} \cong \mathcal{F}_x$ and $(\Delta^{-1}\mathcal{G})_{x} \cong \mathcal{G}_x$, proving that both $\Delta^{-1}$ and $\Delta_{\ast}$ preserve weak equivalences. \\
For 4., first note that this is indeed an adjunction. To see this, let $M$ be a complex of $A$-modules and consider a map $M\rightarrow \mathcal{F}(\textnormal{Spec}(A))$. If $U= \textnormal{Spec}(B)$ is an affine open of $X = \textnormal{Spec}(A)$, then we get a restriction map $\mathcal{F}(X) \rightarrow \mathcal{F}(\textnormal{Spec}(B))$ which is a map of $A$-modules. But $\mathcal{F}(\textnormal{Spec}(B))$ is a $B$-module, and hence we get a map $\mathcal{F}(X) \otimes_A B \rightarrow \mathcal{F}(\textnormal{Spec}(B))$. We can hence construct a map
\begin{align*}
    \tilde{M}(\textnormal{Spec}(B)) \cong M\otimes_A B \rightarrow \mathcal{F}(X)\otimes_A B \rightarrow \mathcal{F}(\textnormal{Spec}(B)).
\end{align*}
of $B$-modules. Now note that $\tilde{(-)}$ sends quasi-isomorphisms to pointwise weak equivalences: If $M \rightarrow N$ is a quasi-isomorphism of complexes of $A$-modules and $U = \textnormal{Spec}(B) \subseteq \textnormal{Spec}(A)$ is an affine open, then in particular $B$ is flat over $A$ and therefore $- \otimes_A B$ preserves quasi-isomorphisms. Hence $\tilde(M)(U) = M\otimes_A B \rightarrow N\otimes_A B = \tilde{N}(U)$ is again a quasi-isomorphism. Further, we already know that the global sections functor preserves acyclic fibrations. This shows that the above adjunction is Quillen. Now if $M \rightarrow N$ is an acyclic fibration, then so is $M\otimes_A B \rightarrow N\otimes_A B$. This finishes the proof.\\
For 5., just note that everything in the proof of 4. still works.
\end{proof}
\begin{proof}[Proof of Theorem \ref{thm5}]
We work over the affine open site. Let $\mathcal{P} \xtwoheadrightarrow{\simeq} \mathcal{O}_X$ be a cofibrant resolution in $\mathcal{O}_X\otimes \mathcal{O}_X$ modules. The $\mathcal{O}_X\otimes\mathcal{O}_X$-module $\mathcal{P}$ is bifibrant and $\mathcal{H}om_{\mathcal{O}_X\otimes \mathcal{O}_X}(-,-)$ is a right Quillen bifunctor, implying that $Q\mathcal{H}om_{\mathcal{O}_X\otimes \mathcal{O}_X}(\mathcal{P},\mathcal{P})$ is again bifibrant in $\affpsh{X}$. We hence get a weak equivalence
\begin{align*}
    \mathbb{R}\Gamma_U(Q\mathcal{H}om_{\mathcal{O}_X\otimes \mathcal{O}_X}(\mathcal{P}, \mathcal{P})) \simeq \mathcal{H}om_{\mathcal{O}_X\otimes \mathcal{O}_X}(\mathcal{P}, \mathcal{P})(U) \cong \map_{\mathcal{O}_U\otimes \mathcal{O}_U}(\mathcal{P}|_U,\mathcal{P}|_U)
\end{align*}
We then have the following chain of weak equivalences 
\begin{align*}
    \map_{\mathcal{O}_U\otimes \mathcal{O}_U}(\mathcal{P}|_U,\mathcal{P}|_U)&\cong \map_{\mathcal{O}_{U\times_{\field{k}} U}}((\Delta_U)_{\ast}(\mathcal{P}|_U), (\Delta_U)_{\ast}(\mathcal{P}|_U)) \\&\cong \map_{\mathcal{O}_{X\times_{\field{k}} X}|_{U\times_{\field{k}} U}}(\Delta_{\ast}(\mathcal{P})|_{U\times_{\field{k}} U}, \Delta_{\ast}(\mathcal{P})|_{U\times_{\field{k}}U}).
\end{align*} 
By the above lemma $\Delta_{\ast}\mathcal{P}$ is again bifibrant, and $(\Delta_{\ast}\mathcal{P})|_{U\times_{\field{k}}U}$ is fibrant. We can therefore use \cite[Proposition 1.7.3]{Hin2} with a choice of cofibrant resolution $\mathcal{P}'\xtwoheadrightarrow{\simeq} (\Delta_{\ast}(\mathcal{P}))|_{U\times_{\field{k}} U}$ to get weak equivalences
\begin{align*}
    \map_{\mathcal{O}_{X\times_{\field{k}} X}|_{U\times_{\field{k}} U}}(\Delta_{\ast}(\mathcal{P})|_{U\times_{\field{k}} U}, \Delta_{\ast}(\mathcal{P})|_{U\times_{\field{k}}U}) &\xrightarrow{\simeq} \map_{\mathcal{O}_{X\times_{\field{k}} X}|_{U\times_{\field{k}} U}}(\mathcal{P}', \Delta_{\ast}(\mathcal{P})|_{U\times_{\field{k}}U})\\& \xleftarrow{\simeq} \map_{\mathcal{O}_{X\times_{\field{k}} X}|_{U\times_{\field{k}} U}}(\mathcal{P}', \mathcal{P}').
\end{align*}
Note that $(\Delta_{\ast} \mathcal{P})|_{U\times_{\field{k}} U} \xtwoheadrightarrow{\simeq} (\Delta_U)_{\ast}\mathcal{O}_U \cong \tilde{A}$ is again a trivial fibration, and therefore $\mathcal{P}' \xtwoheadrightarrow{\simeq} \tilde{A}$ is a cofibrant resolution in $\mathcal{O}_{U\times_{\field{k}} U}$-modules. Now let $P \xtwoheadrightarrow{\simeq} A$ be a cofibrant resolution of $A$ as an $A^e$-module. Then $\tilde{P}$ is a cofibrant $\mathcal{O}_{U\times_{\field{k}} U}$-module, and we hence get a weak equivalence $\tilde{P} \xrightarrow{\simeq} \mathcal{P}'$ between bifibrant objects. Therefore,
\begin{align*}
    \map_{\mathcal{O}_{U\times_{\field{k}}U}}(\mathcal{P}',\mathcal{P}') &\xrightarrow{\simeq} \map_{\mathcal{O}_{U\times_{\field{k}}U}}(\tilde{P},\mathcal{P}') \\&\xleftarrow{\simeq} \map_{\mathcal{O}_{U\times_{\field{k}}U}}(\tilde{P},\tilde{P}) \\&\cong \map_{A\otimes A} (P,P).
\end{align*}
This proves that $\mathbb{R}\Gamma_U(\mathfrak{Z}(\tilde{\mathcal{O}}_X)(\mathfrak{a})) \simeq \mathfrak{Z}(\tilde{A})(\mathfrak{a})$ as complexes of $\field{k}$-modules. Recall that $\mathbb{R}\Gamma_U$ is lax symmetric monoidal, and therefore we get an induced evaluation map 
\begin{align*}
    Q\mathcal{H}om_{\mathcal{O}_X\otimes \mathcal{O}_X}(\mathcal{P}, \mathcal{P})(U) \otimes \mathcal{P}(U) \rightarrow R(Q\mathcal{H}om_{\mathcal{O}_X\otimes \mathcal{O}_X}(\mathcal{P},\mathcal{P})(U) \otimes \mathcal{P}(U)) \rightarrow \mathcal{P}(U)
\end{align*}
But $\mathcal{P}(U) \simeq \mathcal{O}_X(U) \cong A \simeq P$, so this is in fact equivalent to the evaluation map 
\begin{align*}
    \map_{A\otimes A}(P,P) \otimes P \rightarrow P
\end{align*}
of the center of $A$. This shows that the above equivalence is indeed an equivalences of $\mathbb{E}_2$-algebras.
\end{proof}
\subsection{Recovering polydifferential operators as the center of $\mathcal{O}_X$}
Now suppose that $X$ is quasi-compact, separated, of finite type and smooth over $\field{k}$. Let $\dgsh{X}$ denote the category of sheaves of complexes of $\field{k}$-modules on the site $\text{Open}(X)$. Recall that in this case the Hochschild cohomology of $X$ is given by the hypercohomology of the sheaf of polydifferential operators $\mathcal{D}_{\text{poly}}(X)\in \dgsh{X}$. If $U = \text{Spec}(A)\subseteq X$ is an affine open, then 
\begin{align*}
    \mathcal{D}_{\text{poly}}(X)_n(U) &= \{f\in \text{Hom}_{\field{k}}(A^{\otimes n},A): f \text{ is a differential operator in each factor}\}\\&\subseteq \text{Hom}_{\field{k}}(A^{\otimes n},A).
\end{align*}
The sheaf of polydifferential operators is quasi-coherent as an $\mathcal{O}_X$-module, and therefore fibrant in the local projective model structure on affine opens. We want to show that the sheaf of polydifferential operators is indeed a model of the center of $\mathcal{O}_X$.
\begin{theorem}\label{thm6}
Let $X$ be a smooth, quasi-compact, separated scheme of finite type over $\field{k}$. We have an equivalence
\begin{align*}
    Q\mathcal{D}_{\textnormal{poly}}(X) \simeq \mathfrak{Z}(\tilde{\mathcal{O}}_X)(\mathfrak{a})
\end{align*}
    in the $\infty$-category $\lmod_{\tilde{\mathcal{O}}_X}(\textnormal{Sh}_{\infty}(X))$ of $\tilde{\mathcal{O}}_X$-modules. 
\end{theorem}
It suffices to show this equivalence for the sites of affine opens, since they yield an equivalent $\infty$-category. Since $\iota_{\ast}\mathfrak{Z}(\tilde{\mathcal{O}}_X)(\mathfrak{a})\simeq Q\mathcal{H}om_{\mathcal{O}_X\otimes \mathcal{O}_X}(\mathcal{P},\mathcal{P})$, it thus suffices to show 
\begin{align*}
    \iota_{\ast}\mathcal{D}_{\text{poly}}(X) \simeq \mathcal{H}om_{\mathcal{O}_X\otimes \mathcal{O}_X}(\mathcal{P},\mathcal{P})
\end{align*}
as presheaves of $\mathcal{O}_X$-modules. In the following we will work in the site of affine opens and we suppress the restriction of sheaves to affine opens.
\begin{notation}
    Let $\mathcal{O}$ be an associative algebra in dg sheaves. If $\mathcal{F},\mathcal{G}$ are sheaves of left $\mathcal{O}$-modules, recall that $\mathbb{R}\mathcal{H}om_{\mathcal{O}}(\mathcal{F},\mathcal{G}) = \mathcal{H}om_{\mathcal{O}}(\mathcal{F},\mathcal{J})$ for some K-injective resolution $\mathcal{J}$ of $\mathcal{G}$ in the catgory of sheaves of left $\mathcal{O}$-modules. Let $\Delta: X\rightarrow X\times_{\field{k}} X$ be the diagonal. We have already used the adjunction $\Delta^{-1}\dashv \Delta_{\ast}$ induced by this in Lemma \ref{lem3} above, but we now want to consider the full site of affine opens on $X\times_{\field{k}} X$ instead of the smaller site $\diag_X$, and we also consider sheaves instead of presheaves. In particular, the map $\Delta^{-1}: \affsh{X\times_{\field{k}}X} \rightarrow \affsh{X}$ is now given by the presheaf version followed by sheafification. We then have $\Delta^{-1}\Delta_{\ast} \cong \text{id}$ since $X$ is separated. In particular, $\Delta^{-1}\mathcal{O}_{X\times_{\field{k}} X} \cong \mathcal{O}_X \overset{a}{\otimes} \mathcal{O}_X$, where we denote by $\overset{a}{\otimes}$ the tensor product of sheaves.
\end{notation}
\begin{lemma}\label{lem6}
\begin{enumerate}
    \item Let $\mathcal{P} \xtwoheadrightarrow{\simeq}\mathcal{O}_X$ be a cofibrant resolution of presheaves of $\mathcal{O}_X\otimes \mathcal{O}_X$-modules. We have a local quasi-isomorphism of complexes of presheaves 
    \begin{align*}
        \mathcal{H}om_{\mathcal{O}_X\otimes \mathcal{O}_X}(\mathcal{P},\mathcal{P}) \simeq \mathbb{R}\mathcal{H}om_{\mathcal{O}_X\overset{a}{\otimes}\mathcal{O}_X}(\mathcal{O}_X,\mathcal{O}_X)
    \end{align*}
    \item If $\mathcal{F}$ and $\mathcal{G}$ are sheaves, then $\Delta_{\ast} \mathcal{H}om_{\mathcal{O}_X\overset{a}{\otimes}\mathcal{O}_X}(\Delta^{-1}\mathcal{F},\mathcal{G}) \cong \mathcal{H}om_{\mathcal{O}_{X\times_{\field{k}} X}}(\mathcal{F},\Delta_{\ast}\mathcal{G})$.
    \item If $\mathcal{O}_X\xrightarrow{\simeq} \mathcal{I}$ is a K-injective resolution in sheaves of $\mathcal{O}_X\overset{a}{\otimes}\mathcal{O}_X$-modules, then $\Delta_{\ast}\mathcal{O}_X \rightarrow \Delta_{\ast}\mathcal{I}$ is a K-injective resolution in $\mathcal{O}_{X\times_{\field{k}} X}$-modules.
\end{enumerate}
\end{lemma}
\begin{proof}
For 1., let $\alpha: \mathcal{O}_X\otimes \mathcal{O}_X \xrightarrow{\simeq} \mathcal{O}_X \overset{a}{\otimes} \mathcal{O}_X$ be the unit of the sheafification adjunction. This is a weak equivalence of dg algebras in presheaves, and therefore by \cite[Lemma 1.5.3]{Hin2} induces a Quillen equivalence
\begin{center}
    \begin{tikzcd}
           \lmod_{\mathcal{O}_X \otimes \mathcal{O}_X}(\affpsh{X})\arrow[r, shift left=1ex, "\alpha^{\ast}"{name=G}] & \lmod_{\mathcal{O}_X \overset{a}{\otimes} \mathcal{O}_X}(\affpsh{X})\arrow[l, shift left=.5ex, "\alpha_{\ast}"{name=F}]
            \arrow[phantom, from=F, to=G, , "\scriptscriptstyle\boldsymbol{\dashv}" rotate=-90]
        \end{tikzcd}.
\end{center}
We therfore get the following chain of weak equivalences
\begin{align*}
\mathcal{H}om_{\mathcal{O}_X\otimes \mathcal{O}_X}(\mathcal{P},\mathcal{P}) &\simeq \mathcal{H}om_{\mathcal{O}_X\otimes \mathcal{O}_X}(\mathcal{P},\alpha_{\ast}\mathcal{O}_X) \\&\cong \mathcal{H}om_{\mathcal{O}_X\overset{a}{\otimes}\mathcal{O}_X}(\alpha^{\ast}\mathcal{P},\mathcal{O}_X) \\&\simeq \mathcal{H}om_{\mathcal{O}_X\overset{a}{\otimes} \mathcal{O}_X}(\alpha^{\ast}\mathcal{P},\mathcal{I}) \\&\cong \mathcal{H}om_{\mathcal{O}_X\overset{a}{\otimes}\mathcal{O}_X}((\alpha^{\ast}\mathcal{P})^a,\mathcal{I})  \\&\simeq \mathcal{H}om_{\mathcal{O}_X\overset{a}{\otimes} \mathcal{O}_X}(\mathcal{O}_X,\mathcal{I}) \\&= \mathbb{R}\mathcal{H}om_{\mathcal{O}_X\overset{a}{\otimes} \mathcal{O}_X}(\mathcal{O}_X,\mathcal{O}_X).
\end{align*}
The 2. statement follows directly from the definition of the presheaf hom and the adjoint properties of $\Delta^{-1}$ and $\Delta_{\ast}$. If $U$ is an affine open in $X\times_{\field{k}} X$, then
\begin{align*}
    \Delta_{\ast}\mathcal{H}om_{\mathcal{O}_X \overset{a}{\otimes} \mathcal{O}_X}(\Delta^{-1}\mathcal{F},\mathcal{G})(U) &= \map_{(\mathcal{O}_X \overset{a}{\otimes} \mathcal{O}_X)|_{\Delta^{-1}(U)}}(\Delta^{-1}\mathcal{F}|_{\Delta^{-1}(U)}, \mathcal{G}|_{\Delta^{-1}(U)}) \\&\cong \map_{(\mathcal{O}_X \overset{a}{\otimes} \mathcal{O}_X)|_{\Delta^{-1}(U)}}((\Delta|_{\Delta^{-1}(U)})^{-1}\mathcal{F}|_{U}, \mathcal{G}|_{\Delta^{-1}(U)}) \\&\cong \map_{\mathcal{O}_{(X\times_{\field{k}}X)|_U}}(\mathcal{F}_U, \Delta_{\ast}\mathcal{G}|_U) = \mathcal{H}om_{\mathcal{O}_{X\times_{\field{k}} X}}(\mathcal{F},\Delta_{\ast}\mathcal{G})(U).
\end{align*}
For the 3. statement, note first that $\mathcal{O}_X$ and $\mathcal{I}$ are both fibrant in the local projective model structure, and the presheaf version of the $\Delta^{-1}\dashv \Delta_{\ast}$ adjunction is Quillen for this model structure on the affine open sites, so $\Delta_{\ast} \mathcal{O}_X \rightarrow \Delta_{\ast}\mathcal{I}$ is again a local weak equivalence. Further $\Delta^{-1}$ is exact, and therefore preserves acyclic complexes. Therefore, if $\mathcal{S}$ is an acyclic $\mathcal{O}_{X\times_{\field{k}} X}$-module, then 
\begin{align*}
    \map_{\mathcal{O}_{X\times_k X}}(\mathcal{S},\Delta_{\ast}\mathcal{I}) \cong \map_{\mathcal{O}_X\overset{a}{\otimes}\mathcal{O}_X}(\Delta^{-1}\mathcal{S},\mathcal{I})
\end{align*}
is acyclic, proving that $\Delta_{\ast}\mathcal{I}$ is K-injective.
\end{proof}
\begin{proof}[Proof of Theorem \ref{thm6}]
By \cite[Corollary 2.9]{Y} we have a local weak equivalence 
\begin{align*}
\Delta_{\ast} \mathcal{D}_{\textnormal{poly}}(X)\simeq \mathbb{R}\mathcal{H}om_{\mathcal{O}_{X\times_{\field{k}} X}}(\Delta_{\ast}\mathcal{O}_X,\Delta_{\ast}\mathcal{O}_X).
\end{align*}
We then get the following chain of local weak equivalences
\begin{align*}
    \Delta_{\ast} \mathbb{R}\mathcal{H}om_{\mathcal{O}_X \overset{a}{\otimes} \mathcal{O}_X}(\mathcal{O}_X,\mathcal{O}_X) &= \Delta_{\ast}\mathcal{H}om_{\mathcal{O}_X \overset{a}{\otimes} \mathcal{O}_X}(\mathcal{O}_X,\mathcal{I})\\&\cong \Delta_{\ast}\mathcal{H}om_{\mathcal{O}_X \overset{a}{\otimes} \mathcal{O}_X}(\Delta^{-1}\Delta_{\ast}\mathcal{O}_X,\mathcal{I}) \\&\cong \mathcal{H}om_{\mathcal{O}_{X\times_{\field{k}}X}}(\Delta_{\ast}\mathcal{O}_X,\Delta_{\ast}\mathcal{I}) \\&\simeq \mathbb{R}\mathcal{H}om_{\mathcal{O}_{X\times_{\field{k}} X}}(\Delta_{\ast}\mathcal{O}_X,\Delta_{\ast}\mathcal{O}_X) \\&\simeq \Delta_{\ast}\mathcal{D}_{\textnormal{poly}}(X)
\end{align*}
where in the second to last step we used \ref{lem6}(3.). Now note that $\Delta^{-1}$ preserves local weak equivalences, and therefore 
\begin{align*}
\mathbb{R}\mathcal{H}om_{\mathcal{O}_X\overset{a}{\otimes} \mathcal{O}_X}(\mathcal{O}_X,\mathcal{O}_X) \simeq \mathcal{D}_{\textnormal{poly}}(X).
\end{align*}
Together with \ref{lem6}(1.) this finishes the proof.
\end{proof}

Let $\mathcal{B}(\mathcal{O}_X)$ denote the $\mathcal{O}_X\otimes \mathcal{O}_X$-module presheaf $U= \text{Spec}(A)\mapsto B(A)$, where $B(A)$ denotes the bar complex of $A$. We have an acyclic fibration $\mathcal{B}(\mathcal{O}_X) \xtwoheadrightarrow{\simeq} \mathcal{O}_X$ given by multiplication. Hence for a cofibrant resolution $\mathcal{P} \xtwoheadrightarrow{\simeq} \mathcal{O}_X$ of $\mathcal{O}_X$ as an $\mathcal{O}_X\otimes \mathcal{O}_X$-module, we get a lift $\mathcal{P} \rightarrow \mathcal{B}(\mathcal{O}_X)$. We get an evaluation map
\begin{align*}
    \mathcal{D}_{\text{poly}}(X) \otimes \mathcal{B}(\mathcal{O}_X) \rightarrow \mathcal{O}_X
\end{align*}
coming from the fact that $\mathcal{D}_{\text{poly}}(X)$ is affine locally a subcomplex of the Hochschild complex. This lifts to a map
\begin{center}
    \begin{tikzcd}
Q\mathcal{D}_{\text{poly}}(X)\otimes \mathcal{P} \arrow[d] \arrow[r, dashed]       & \mathcal{P} \arrow[d, "\simeq", two heads] \\
\mathcal{D}_{\text{poly}}(X)\otimes \mathcal{B}(\mathcal{O}_X) \arrow[r] & \mathcal{O}_X                             
\end{tikzcd}
\end{center}
and it is clear from the proof of Theorem \ref{thm6} that this map corresponds to the evaluation map
\begin{align*}
    Q\mathcal{H}om_{\mathcal{O}_X\otimes \mathcal{O}_X}(\mathcal{P},\mathcal{P}) \otimes \mathcal{P} \rightarrow \mathcal{P}.
\end{align*}
It therefore makes $Q\mathcal{D}_{\text{poly}}(X)$ into a center of $\tilde{\mathcal{O}}_X$. In particular, this equips the sheaf of polydifferential operators with a new $\mathbb{E}_2$-algebra structure in the $\infty$-category of dg sheaves on $X$. 
\subsection{Comparison to the classical homotopy Gerstenhaber algebra structure on polydifferential operators}
For a separated quasi-compact smooth finite type scheme $X$ over $\field{k}$, Tamarkin's proof of Deligne's Conjecture equips $\mathcal{D}_{\text{poly}}(X)$ with a $\ger_{\infty}$-algebra structure coming from the canonical $\braces$-algebra structure, and thus a Gerstenhaber algebra structure in the $\field{k}$-linear derived 1-category on $X$. On the other hand, exhibiting $\mathcal{D}_{\text{poly}}(X)$ as a center of $\tilde{\mathcal{O}}_X$ equips it with an $\mathbb{E}_2$-algebra structure in $\text{Sh}_{\infty}(X)$, and therefore another Gerstenhaber algebra structure in the derived 1-category, which is just the homotopy category of $\text{Sh}_{\infty}(X)$. We want to compare these two Gerstenhaber algebra structures. \\ \par 
To this end, note that by Corollary \ref{cor7}, the multiplication in the center Gerstenhaber algebra structure is given equivalently by the convolution product or the composition product, and the bracket is induced by any filling of the appropriate square in the action category. \\ \par
By \cite{Y}, we have an isomorphism of sheaves
\begin{align*}
    \mathcal{D}_{\textnormal{poly}}(X) \cong \mathcal{H}\textnormal{om}^{\textnormal{cont}}_{\mathcal{O}_{X\times_{\field{k}} X}}(\widehat{\mathcal{B}}(X),\mathcal{O}_X)
\end{align*}
where $\widehat{\mathcal{B}}_n(X) = \mathcal{O}_{\mathfrak{X}^{n+2}}$ is the complete bar complex with $\mathfrak{X}^n$ the formal completion of $X^n$ along the $n$-fold diagonal. From this presentation it is easy to compute the convolution and composition product. Note that for $U = \spec(A)\subseteq X$, we have
\begin{align*}
    \Gamma_U(\mathcal{H}\textnormal{om}^{\textnormal{cont}}_{\mathcal{O}_{X\times_{\field{k}} X}}(\widehat{\mathcal{B}}(X),\mathcal{O}_X)) \cong \map_{A\otimes A}^{\textnormal{cont}}(\widehat{B}(A),A)
\end{align*}
for $\widehat{B}_n(A)$ the adic completion of $B_n(A)$ at the kernel of the multiplication map $B_n(A) \rightarrow A$. In particular, the flat resolution
\begin{align*}
    \widehat{\mathcal{B}}(X) \rightarrow \mathcal{O}_X
\end{align*}
admits a section $s: \mathcal{O}_X \rightarrow \widehat{\mathcal{B}}_0(X)$ that is glued together from the sections of the resolutions $B(A) \rightarrow A$. We can hence build a diagonal 
\begin{align*}
    \Delta: \widehat{\mathcal{B}}(X) \rightarrow \mathcal{O}_X \xrightarrow{\cong} \mathcal{O}_X \overset{a}{\otimes}_{\mathcal{O}_X} \mathcal{O}_X \xrightarrow{s\otimes s}  \widehat{\mathcal{B}}(X)\overset{a}{\otimes}_{\mathcal{O}_X}  \widehat{\mathcal{B}}(X).
\end{align*}
\begin{lemma}\label{lem4}
    The convolution product on $\mathcal{H}\textnormal{om}^{\textnormal{cont}}_{\mathcal{O}_{X\times_{\field{k}} X}}(\widehat{\mathcal{B}}(X),\mathcal{O}_X)$ is homotopic to the cup product on $\mathcal{D}_{\textnormal{poly}}(X)$, which locally agrees with the classical cup product on Hochschild cochains. 
\end{lemma}
\begin{proof}
    The diagonal is zero on $\widehat{\mathcal{B}}_n(X)$ for $n>0$, and for $n=0$ it is given locally by the formula
    \begin{align*}
        a_0 \otimes a_1 \mapsto (1\otimes 1) \otimes_A (1 \otimes a_0a_1).
    \end{align*}
    The diagonal on $B(A)$ coming from its universal property is given for $n=0$ by 
    \begin{align*}
        a_0 \otimes a_1 \mapsto (a_0 \otimes 1) \otimes_A (1\otimes a_1),
    \end{align*}
    and for $n>0$ by 
    \begin{align*}
        a_0 \otimes a_1 \otimes \dots \otimes a_n \otimes a_{n+1} \mapsto \sum_{i = 0}^n (a_0 \otimes a_1 \otimes \dots \otimes a_i \otimes 1) \otimes_A (1 \otimes a_{i+1} \otimes \dots \otimes a_n \otimes a_{n+1}).
    \end{align*}
    Locally on $B(A)$, a homotopy between these two maps is given by 
    \begin{align*}
        H: B(A) &\rightarrow (B(A) \otimes_A B(A))[1]\\
        a_0 \otimes a_1 \otimes \dots \otimes a_n \otimes a_{n+1} &\mapsto \sum_{i=0}^{n+1} (1\otimes a_0 \otimes \dots \otimes a_{i-1} \otimes 1) \otimes_A (1\otimes a_i \otimes \dots \otimes a_{n+1}).
    \end{align*}
    Recall that the convolution product of two continuous maps $f$ and $g$ is given by 
    \begin{align*}
        \widehat{\mathcal{B}}(X) \xrightarrow{\Delta}  \widehat{\mathcal{B}}(X) \overset{a}{\otimes}_{\mathcal{O}} \widehat{\mathcal{B}}(X) \xrightarrow{f\otimes g} \mathcal{O}_X\overset{a}{\otimes}_{\mathcal{O}_X} \mathcal{O}_X \rightarrow \mathcal{O}_X.
    \end{align*}
    Locally, it suffices to consider the restriction to $B(A)$. We can then use the above homotopy $H$ to obtain a homotopy between the cup product and the above formula for the convolution product with our new diagonal. Inspecting the formula for $H$ we see that these glue together to yield a global homotopy between the global convolution product and the cup product.
\end{proof}
We already know that the local circle products coming from the $\braces$-algebra structure glue together to give a homtopy for the square
\begin{center}
    \begin{tikzcd}
\mathcal{D}_{\textnormal{poly}}(X)^{\overset{a}{\otimes} 4} \arrow[d, "(\smile \otimes \smile) \circ(\textnormal{id}\otimes \tau \otimes \textnormal{id})"'] \arrow[r, "\smile \otimes \smile"] & \mathcal{D}_{\textnormal{poly}}(X)^{\overset{a}{\otimes} 2} \arrow[d, "\smile"] \\
\mathcal{D}_{\textnormal{poly}}(X)^{\overset{a}{\otimes} 2} \arrow[r, "\smile"']                                                                                                                & \mathcal{D}_{\textnormal{poly}}(X)                                
\end{tikzcd}
\end{center}
analogous to the affine case. This implies the following:
\begin{corollary}\label{thm9}
    The Gerstenhaber structure on $\mathcal{D}_{\textnormal{poly}}(X)$ in the $\field{k}$-linear derived 1-category coming from the center is isomorphic to the signed classical one coming from the $\braces$-algebra structure.
\end{corollary}
\begin{proof}
   By the above Lemma \ref{lem4}, the center product is given by the usual signed cup product in the derived 1-category. Again using Corollary \ref{cor7} and the fact that the global circle product induces a homotopy
   \begin{align*}
       h: \mathcal{D}_{\text{poly}}(X)^{\overset{a}{\otimes} 4} &\rightarrow \mathcal{D}_{\text{poly}}(X)[1]\\
       \mathcal{D}_{\text{poly}}(\spec(A))^{\otimes 4} \ni f_1 \otimes g_1 \otimes f_2 \otimes g_2 &\mapsto (-1)^{|f_1|+|f_2|+|f_2||g_1|-1} f_1 \smile f_2\{g_1\} \smile g_2
   \end{align*}
   in the above square, we see that the bracket determined by Corollary \ref{cor5} of the center Gerstenhaber algebra structure is equal to the classical Gerstenhaber bracket on polydifferential operators.
\end{proof}

\appendix
\section{The endomorphism $\infty$-category}\label{appendixA}
In this section we show that our endomorphism $\infty$-category $\mathcal{C}_{\mathfrak{a}}^{\otimes} \times_{{\mathcal{C}_{\mathfrak{m}}}} {\mathcal{C}_{\mathfrak{m}}}_{/M}$ defined in Definition \ref{def1} is equivalent to Lurie's endomorphism $\infty$-category as defined in \cite[Definition 4.7.1.1]{HA}. In particular, this will imply that $\mathcal{C}_{\mathfrak{a}}^{\otimes} \times_{\mathcal{C}_{\mathfrak{m}}} (\mathcal{C}_{\mathfrak{m}})_{/M}$ is the underlying $\infty$-category of a monoidal $\infty$-category.\\ \par
Lurie's definition is given in terms of planar $\infty$-operads as defined in \cite[Section 4.1.3]{HA}. In particular it is given for a map $p: \mathscr{M}^{\circledast} \rightarrow \Delta^1 \times N(\mathbb{\Delta})^{\text{op}}$ exhibiting $\mathscr{M} := \mathscr{M}^{\circledast}_{0,[0]}$ as left tensored over $\mathcal{C}^{\circledast}:= \mathscr{M}^{\circledast}\times_{\Delta^1} \{1\}$ in the planar sense as in \cite[Variant 4.2.2.11]{HA}, and an element $M\in \mathscr{M}$. 
\begin{definition}[\cite{HA}, Definition 4.7.1.1]\label{def4}
Let $p: \mathscr{M}^{\circledast} \rightarrow \Delta^1 \times N(\mathbb{\Delta})^{\text{op}}$ be a map exhibiting $\mathscr{M}$ as left tensored over the planar $\infty$-operad $\mathcal{C}^{\circledast}$. An enriched morphism of $\mathscr{M}$ is a span 
\begin{align*}
    M \xleftarrow{\alpha} X \xrightarrow{\beta} N
\end{align*}
such that 
\begin{itemize}
    \item $p(\alpha)$ is the map $(0,[1]) \rightarrow (0, [0])$ given by $[0] \cong \{0\} \hookrightarrow [1]$ in $\mathbb{\Delta}$,
    \item $\beta$ is inert and $p(\beta)$ is the map $(0,[1]) \rightarrow (0, [0])$ given by $[0] \cong \{1\} \hookrightarrow [1]$ in $\mathbb{\Delta}$.
\end{itemize}
Denote by $\text{EnMor}(\mathscr{M}^{\circledast})$ the full subcategory of $\text{Fun}_{\Delta^1 \times N(\mathbb{\Delta})^{\text{op}}}(\Lambda_0^2,\mathscr{M}^{\circledast})$ spanned by the enriched morphisms. There is a pair of forgetful functors $\text{EnMor}(\mathscr{M}^{\circledast})\rightarrow \mathscr{M}$ sending a span $M \leftarrow X \rightarrow N$ to $M$ and $N$ respectively. The \textbf{endomorphism $\infty$-category} of $M\in \mathscr{M}$ is the fiber product
\begin{align*}
    \mathcal{C}[M]_{\text{Lurie}} := \{M\} \times_{\mathscr{M}} \text{EnMor}(\mathscr{M}^{\circledast}) \times_{\mathscr{M}} \{M\}.
\end{align*}
\end{definition}
Instead, we start from a coCartesian fibration of $\infty$-operads $q: \mathcal{C}^{\otimes} \rightarrow \lm$, exhibiting $\mathcal{C}_{\mathfrak{m}}$ as left tensored over the monoidal $\infty$-category $\mathcal{C}^{\otimes}_{\mathfrak{a}}$. Hence, we first need to construct from $q$ an $\infty$-category $\mathscr{M}^{\circledast}$ and a map $p: \mathscr{M}^{\circledast} \rightarrow \Delta^1 \times N(\mathbb{\Delta})^{\text{op}}$ as above.\\ \par

Construct the $\infty$-category $\mathscr{M}^{\circledast}$ as the fiber product
    \begin{center}
        \begin{tikzcd}
\mathscr{M}^{\circledast} :=\mathcal{C}^{\otimes}\times_{\lm}(N(\mathbb{\Delta}^{\text{op}})\times\Delta^1) \arrow[d] \arrow[r] & \mathcal{C}^{\otimes} \arrow[d, "q"] \\
N(\mathbb{\Delta}^{\text{op}})\times \Delta^1 \arrow[r, "\gamma"']                                 & \lm                                 
\end{tikzcd}
    \end{center}
    where $\gamma$ is defined in \cite[Remark 4.2.2.8]{HA}. In particular, $\mathscr{M}^{\circledast}$ comes equipped with a coCartesian fibration $p: \mathscr{M}^{\circledast} \rightarrow N(\mathbb{\Delta}^{\text{op}}) \times \Delta^1$. We have
    \begin{align*}
        \mathscr{M}^{\circledast}_{[0],0} = \mathcal{C}^{\otimes} \times_{\lm} \{[0],0\} \simeq \mathcal{C}_{\mathfrak{m}},
    \end{align*}
    since the functor $\text{LCut}: N(\mathbb{\Delta}^{\text{op}})\rightarrow \lm$ defined in \cite[Construction 4.2.2.6]{HA} sends $[0]$ to $(\langle 1\rangle, \{1\}) = \mathfrak{m}$. Consider the diagram
    \begin{center}
       \begin{tikzcd}
\mathscr{M}^{\circledast}\times_{\Delta^1}\{1\} \arrow[d] \arrow[r]  & \mathcal{C}^{\otimes}\times_{\lm}(N(\mathbb{\Delta}^{\text{op}})\times\Delta^1) \arrow[d] \arrow[r] & \mathcal{C}^{\otimes} \arrow[d, "q"] \\
N(\mathbb{\Delta}^{\text{op}})\times \{1\} \arrow[r, hook] \arrow[d] & N(\mathbb{\Delta}^{\text{op}})\times \Delta^1 \arrow[r, "\gamma"'] \arrow[d]                        & \lm                                  \\
\{1\} \arrow[r, hook]                                                & \Delta^1                                                                                           &                                     
\end{tikzcd}.
    \end{center}
    Then the upper right hand side square is a pullback by definition, the lower left hand side square is a pullback and the left hand side rectangle is a pullback, again by definition. By the pasting law, the upper left hand side square is a pullback, and hence, again by the pasting law, the large upper rectangle is a pullback. The lower horizontal arrow of this rectangle agrees with the map $\text{Cut}: N(\mathbb{\Delta}^{\text{op}}) \rightarrow \lm$ defined in \cite[Construction 4.1.2.9]{HA}, so 
    \begin{align*}
        \mathscr{M}^{\circledast} \times_{\Delta^1}\{1\} \simeq \mathcal{C}^{\otimes}\times_{\lm} N(\mathbb{\Delta}^{\text{op}}).
    \end{align*}
    But in the diagram
    \begin{center}
        \begin{tikzcd}
\mathcal{C}_{\mathfrak{a}}^{\otimes}\times_{\ass} N(\mathbb{\Delta}^{\text{op}}) \arrow[r] \arrow[d] & \mathcal{C}^{\otimes}_{\mathfrak{a}} \arrow[r] \arrow[d] & \mathcal{C}^{\otimes} \arrow[d, "q"] \\
N(\mathbb{\Delta}^{\text{op}}) \arrow[r, "\text{Cut}"']                                               & \ass \arrow[r, hook]                                     & \lm                                 
\end{tikzcd}
    \end{center}
    both squares are pullbacks, so the rectangle is as well, and we get
    \begin{align*}
        \mathscr{M}^{\circledast} \times_{\Delta^1} \{1\} \simeq \mathcal{C}_{\mathfrak{a}}^{\otimes} \times_{\ass} N(\mathbb{\Delta^{\text{op}}}),
    \end{align*}
    which is the $\mathbb{A}_{\infty}$-monoidal $\infty$-category corresponding to the monoidal $\infty$-category $\mathcal{C}^{\otimes}_{\mathfrak{a}}$. Call this $\mathbb{A}_{\infty}$-monoidal category $\mathcal{C}_{\mathfrak{a}}^{\circledast}$. Then $p$ exhibits ${\mathcal{C}_{\mathfrak{m}}}$ as left-tensored over $\mathcal{C}_{\mathfrak{a}}^{\circledast}$ in the planar sense. We are now able to state
\begin{proposition}
    Let $q: \mathcal{C}^{\otimes} \rightarrow \lm$ be a coCartesian fibration of $\infty$-operads. Let $p: {\mathcal{C}_{\mathfrak{m}}}^{\circledast} \rightarrow N(\mathbb{\Delta}^{\textnormal{op}})\times \Delta^1$ as constructed above. Then the $\infty$-category $\mathcal{C}_{\mathfrak{a}}[M]$ from Definition \ref{def1} is equivalent to $\mathcal{C}[M]_{\textnormal{Lurie}}$ as in Definition \ref{def4}.
\end{proposition}
\begin{proof}
    Under $\gamma: N(\mathbb{\Delta}^{\text{op}})\times \Delta^1 \rightarrow \lm$, the map
    \begin{align*}
        a: ([0],0) \rightarrow ([1],0)
    \end{align*}
    sending the point in $[0]$ to $0\in [1]$ maps to 
    \begin{align*}
        \text{LCut}(a): (\langle 2\rangle,\{2\}) &\rightarrow (\langle 1\rangle, \{1\})\\
        1 &\mapsto 1\\
        2 &\mapsto 1.
    \end{align*}
    Interpreting $(\langle 2\rangle,\{2\})$ as $(\mathfrak{a},\mathfrak{m})$ and $(\langle 1\rangle,\{1\})$ as $\mathfrak{m}$, this map corresponds to the unique element
    \begin{align*}
        \phi \in \text{Mul}_{\mathscr{LM}}(\{\mathfrak{a},\mathfrak{m}\},\mathfrak{m}).
    \end{align*}
    Similarly, the map
    \begin{align*}
        b: ([0],0) \rightarrow ([1],0)
    \end{align*}
    sending the point in $[0]$ to $1\in [1]$ maps to 
    \begin{align*}
    \text{LCut}(b): (\langle 2\rangle, \{2\}) &\rightarrow (\langle 1\rangle, \{1\})\\
    1 &\mapsto \ast \\
    2 &\mapsto 1.
    \end{align*}
    This map corresponds to the unique element
    \begin{align*}
        \psi \in \text{Mul}_{\mathscr{LM}}(\{\mathfrak{m}\},\mathfrak{m}).
    \end{align*}
    Therefore, to give an enriched morphism of ${\mathcal{C}_{\mathfrak{m}}}$ is equivalent to giving a diagram
    \begin{align*}
        M \xleftarrow{\alpha} X \xrightarrow{\beta} N
    \end{align*}
    in $\mathcal{C}^{\otimes}$ such that 
    \begin{enumerate}
        \item $q(\alpha) = \text{LCut}(a)$,
        \item $q(\beta) = \text{LCut}(b)$, and
        \item $\beta$ is inert, i.e. $q$-coCartesian.
    \end{enumerate}
    Unpacking this, $M$ and $N$ are objects in ${\mathcal{C}_{\mathfrak{m}}}$, and $X=(C,M')$ is an object in $\mathcal{C}^{\otimes}_{(\mathfrak{a},\mathfrak{m})}$, and $\mathcal{C}^{\otimes}_{(\mathfrak{a},\mathfrak{m})}\simeq \mathcal{C}_{\mathfrak{a}} \times \mathcal{C}_{\mathfrak{m}}$ by \cite[Proposition 2.1.2.12 (b)]{HA} and taking $T=(\mathfrak{a},\mathfrak{m})\in \lm_{\langle 2\rangle}$. The map $\alpha: (C,M') \rightarrow M$ and $\beta: (C,M') \rightarrow N$ are morphisms in $\mathcal{C}^{\otimes}$ lifting $\phi$ and $\psi$ respectively. Since $q$ is coCartesian, there is a $q$-coCartesian lift for $\phi$ and $X=(C,M')$, namely the map $(C,M') \rightarrow C\otimes M'$. Hence, the data of $\alpha$ is equivalent to a map $C\otimes M' \rightarrow M$ in ${\mathcal{C}_{\mathfrak{m}}}$. Similarly, there is $q$-coCartesian lift for $\psi$ and $X=(C,M')$, namely the map $(C,M') \rightarrow M'$. Hence the data of $\beta$ is equivalent to a map $M'\rightarrow N$ in ${\mathcal{C}_{\mathfrak{m}}}$, and since $\beta$ is $q$-coCartesian as well, this map is an equivalence. Hence, the $\infty$-category $\mathcal{C}[M]_{\textnormal{Lurie}}$ is equivalent to the $\infty$-category $\mathcal{C}_{\mathfrak{a}}\times_{{\mathcal{C}_{\mathfrak{m}}}} {\mathcal{C}_{\mathfrak{m}}}_{/M}$ with objects given by pairs $(C\in \mathcal{C}_{\mathfrak{a}}, C\otimes M\rightarrow M \in \mathcal{C}_{\mathfrak{m}})$.
\end{proof}
\bibliographystyle{alpha}
\inputencoding{utf8}    
\bibliography{main}
\end{document}